\documentclass[a4paper,10pt]{article}
\usepackage[utf8]{inputenc}
\usepackage{amsthm}
\usepackage{color}
\usepackage{graphicx}
\newtheorem{remark}{Remark}
\newtheorem{proposition}{Proposition}[section]
\newtheorem{corollary}{Corollary}[section]

\newcommand{\mycomment}[1]{ }
\usepackage{amssymb}
\usepackage{amsfonts}
\usepackage{amsmath}

\usepackage{amsthm}
\newcommand\beq{\begin{equation}}
\newcommand\eeq{\end{equation}}

\renewcommand{\to}{\rightarrow}

\renewcommand{\div}{\operatorname{div}}

\newcommand{\tr}{\operatorname{tr}}

\newcommand{\grad}{\boldsymbol{\nabla}}

\newcommand{\ba}{\boldsymbol{a}} 
\newcommand{\bA}{\boldsymbol{A}} %\newcommand{\bolda}{\boldsymbol{a}}

\newcommand{\bc}{\boldsymbol{c}}

\def\bD{\boldsymbol{D}}
\newcommand{\be}{\boldsymbol{e}}

\newcommand{\f}{\boldsymbol{f}}
\newcommand{\bF}{\boldsymbol{F}}

\newcommand{\bI}{\boldsymbol{I}} % \newcommand{\I}{\boldsymbol{I}}

\newcommand{\bK}{\boldsymbol{K}}

\def\bn{\boldsymbol n}

\def\bR{\boldsymbol R}

\def\bS{\boldsymbol S}

\def\bt{\boldsymbol t}

\newcommand{\bu}{\boldsymbol{u}}
\newcommand{\bU}{\boldsymbol{U}}

\newcommand{\bv}{\boldsymbol{v}}

\newcommand{\bx}{\boldsymbol{x}} % \newcommand{\xx}{\boldsymbol{x}}

\newcommand{\bY}{\boldsymbol{Y}}

\newcommand{\bZ}{\boldsymbol{Z}}

\newcommand{\bphi}{\boldsymbol{\phi}}
\newcommand{\bPhi}{\boldsymbol{\Phi}}

\newcommand{\bsigma}{\boldsymbol{\sigma}}
\newcommand{\bSigma}{\boldsymbol{\Sigma}}
\newcommand{\btau}{\boldsymbol{\tau}}

\def\Acal{\mathcal{A}}
\def\Bcal{\mathcal{B}}

\def\Dcal{\mathcal{D}} 			% \newcommand{\D}{\mathcal{D}}
\def\Fcal{\mathcal{F}}
\def\Gcal{\mathcal{G}}
\def\Hcal{\mathcal{H}}

\def\Pcal{\mathcal{P}}

\def\Vcal{\mathcal{V}}

\newcommand{\EE}{\mathbb{E}}

\newcommand{\RR}{\mathbb{R}} \newcommand{\R}{\mathbb{R}} %\def\R{\mathbb{R}}

\title{ 
 Viscoelastic flows of Maxwell fluids \\
 with conservation laws % using Maxwell fluids !!
}
\author{
S\'ebastien Boyaval, Laboratoire d'hydraulique Saint-Venant \\
Ecole des Ponts % ParisTech 
-- EDF R\&D -- CEREMA \\ % Universit\'e Paris-Est 
\& MATHERIALS, Inria Paris (sebastien.boyaval@enpc.fr) \\
EDF'lab 6 quai Watier 78401 Chatou Cedex France
}

\begin{document}

\maketitle

%\cite{Maxwell01011867}
% J.C. Maxwell, « On the dynamical theory of gases », Phil. Trans. Royal Soc., no 157,‎ 1867, p. 49-88 (DOI 10.1098/rstl.1867.0004) ; J.C. Maxwell, « On the dynamical theory of gases », Phil. Mag., vol. 35, no 235,‎ 1868, p. 129-145 et 185-217
% Le modèle de Kelvin-Voigt est un modèle unidimensionnel décrivant le comportement mécanique d'un solide visqueux. 
% En effet, lorsque l'on cesse d'appliquer un chargement à ce matériau, il recouvre toujours la même configuration (forme et dimensions). 
% L'existence d'une configuration stable et unique sous chargement nul est caractéristique des solides, et les différencie des fluides.
% Woldemar Voigt, « Über die innere Reibung der festen Körper, insbesondere der Krystalle – Erster Teil », Abhandlungen der Königlichen Gesellschaft von Wissenschaften zu Göttingen, vol. 36,‎ 1890, p. 3-47 http://gdz.sub.uni-goettingen.de/dms/load/img/?PID=GDZPPN002020785&physid=PHYS_0078
% William Thomson, Math. And Phys. Papers, vol. 3, Cambridge, 1890; également Encyclopedia Britannica, vol. 3, Londres, 1875, « Kelvin W. »

\begin{abstract}
We consider % The question how to write
multi-dimensional extensions of Maxwell's 
seminal % differential constitutive
rheological equation for % fluids with a viscoelastic behaviour
1D % \emph{one-dimensional} % (1D) 
viscoelastic % for fluid with stress relaxation (intuitive fluid), as opposed to creep (intuitive solid) by Kelvin(Thomson)-Voigt  
flows. % has been proposed by Maxwell % one and a half century ago % in 1867
%
% Using stress relaxation, Maxwell's model % it 
% encompasses the modelling of % actually a model for 1D flows: it has well-posed Cauchy problems for motions
% viscous (fluid) % -like : many equilibrium displacements ``instantaneously reached'', stress follows
% flows % intuitive fluids are viscous 
% and % large-deformation
% elastic (solid) % -like : assuming stress equilibrium ``instantaneously reached'', deformation follows
% deformations % though not with creep yet
% within a unified framework that satisfies % formulations of both
% principles of \emph{causality} and \emph{locality}. % unlike Navier-Stokes !!
% %
% It has many physical extensions. % to various (idealized) material properties
% %
% However, Maxwell's equation is not yet univoquely % obviously % clearly 
% used for % general (compressible !)
% multi-dimensional viscoelastic flows. % motions 
% % despite many efforts ! notably noll-1955 
% % from the mathematical viewpoint (well-posedness) as well as physical for compressibility
%
We aim at a causal model % that is both causal and local
for compressible % viscoelastic 
flows, defined by semi-group solutions % to well-posed Cauchy problems
given initial conditions,
and such that % smooth
perturbations % with a bounded support 
propagates at finite speed.

We propose % for the first time 
a \emph{symmetric hyperbolic} system of \emph{conservation laws} 
% endowed with a strictly convex "entropy" -- useful only in absence of source term ?
% to model % two and three
% (compressible) multi-dimensional viscoelastic flows of a fluid satisfying
that contains the \emph{Upper-Convected Maxwell (UCM)} equation as causal model.
% (i.e. elastic fluids with % memory characterized by 
% one % a single 
% relaxation time)
The % new 
system is an extension of % a spatial description of % ???
polyconvex % large-strain % 2D (thin-layer) 
elastodynamics, % for the large deformations of hyperelastic materials (with polyconvex stored energy)
% (usually for solids)
with an additional \emph{material metric} variable % ``distortion'' 
that relaxes to model viscous effects. % like anisotropy !!
% (a non-uniform metric attached to each material particle inspired by the so-called K-BKZ formulation of UCM % Maxwell fluids). % K-BKZ theory \cite{kaye-1962,kaye-1963,bernstein-kearsley-zapas-1963,BKZ1964}
%
% We show that 
% It % the new system 
% is \emph{symmetric hyperbolic} using adequate conservative variables and % a variant of 
% a convexity theorem by Lieb. % 1973
%
% Multi-dimensional % 2 and three : precise % compressible, isothermal and ...
% viscoelastic flows % as solutions to a well-posed Cauchy problem
% which are causal and local can thus be uniquely defined by that system,
% as smooth solutions to Cauchy problems. % like polyconvex elastodynamics
%
Interestingly, the framework could also cover other rheological equations, % visco-elasto-plastic models 
% => INCLUDE CREEP (which is lacking to initial Maxwell model)
% to % further 
% unify % the dynamics of
% for solid and fluid flows % states of various materials
depending on the chosen relaxation limit for the material metric variable. % not necessarily the inverse right Cauchy-Green deformation tensor. % inverse -> inverse in fact
% % NOTE:
% the new system is only one (viscoelastic) instance 
% how to extend polyconvex elastodynamics with an additional material metric variable
% and preserve the causality and locality principles.
% %
% % in this work, we also briefly explain how the relaxation equation % limit 
% % of the new variable % our new material ``order'' parameter
% % could be modified to let our system model a number of materials
% % in between elastic solids and viscous fluids. % therefore
% Indeed, the new system % of conservation laws % for multi-dimensional flows ( compressible) % using UCM fluids 
% is obtained after choosing a particular relaxation limit for the new metric variable.
% % (with compressible Newtonian fluid dynamics as fast-relaxation limit)
%
% We give tracks toward other (visco-elasto-plastic) extensions.

% % Last, we show how 
% The new system is useful % not only in computational rheology but also
% in % computational geophysics 
% geophysical applications, 
%
We propose to apply the new system to incompressible free-surface gravity flows in the shallow-water regime, % with a hydrostatic pressure
when causality % and locality are important properties to model large-scale flows
is important. % in large domains.
% For computational rheology, it offers a new approach to the High-Weissenberg Number Problem (HWNP). <<<< ????
% For incompressible free-surface gravity flows in the shallow-water regime, % with a hydrostatic pressure
The system reduces to a viscoelastic extension of Saint-Venant 2D shallow-water % equations
system % of conservation laws
that is symmetric-hyperbolic and that encompasses our previous % (non-conservative) 2D models
viscoelastic extensions of Saint-Venant proposed % since 2013 
with F.~Bouchut. 
\end{abstract}

%------------------------------------------------------------

\section{Introduction} % ** modelling issue **
\label{sec:intro}

In 1867, 
when % while % it was already clear to Maxwell that 
viscosity % had already become % emerged as
was already % thought 
an important % material (physical property of the matter) or kinemtical !!
concept to model % internal 
friction % batchelor
within % real as opposed to the ideal perfect fluids
fluid flows % turbulent, non-homogeneous \cite{LEIGHTON19861377,CHAPMAN1991469,Leighton-Acrivos-1987,1.5085363
at the human scale % \cite{maxwell-1874} % USEFUL ???
following Poisson's theory of % internal 
friction \cite{poisson-1831},
Maxwell introduced a seminal % stress
relaxation equation for the rheology of one-dimensional (1D) % motions
flows where viscosity % of the fluid  % ``viscoelastic''
is defined % \emph{asymptotically} % ONLY !
from % the product of
elasticity and a characteristic time
\cite{Maxwell01011867}.
The \emph{viscoelastic} model of Maxwell is long known as an interesting model for 1D flows: % with viscous effects
given initial conditions, fluid motions are well-defined % by solutions 
\cite{joseph-renardy-saut-1985} that are genuinely % truely 
causal, i.e. causal and local in particular.

\smallskip %%% On the contrary

By contrast, nowadays, viscosity is often % classically 
introduced in % modern 
continuum mechanics as a material parameter
% precisely 
% as a coefficient of % second-order i.e. 
% rate-dependent \emph{extra-stresses} % energy-dissipative
into the momentum balance of motions described in spatial % Eulerian 
coordinates \cite{Coleman-Noll1963}. % truesdell role ?? noll thesis is essential ; wang-truesdell-1973 is a ``nice'' summary (discussing symmetry groups)
% The definition of viscosity as a material parameter
It still allows to define % properly % from the mathematical viewpoint
\emph{causal} viscous flows as % nonlinear : kato, fisher-marsden, bnzonigavage-serre
semi-group solutions to Cauchy problems. % though 
However, it uses \emph{diffusive} Partial Differential Equations (PDEs) % for the velocity % of \emph{hyperbolic-parabolic} type 
like the celebrated Navier-Stokes equations \cite{lions-1996},
and the latter viscous flows % defined by hyperbolic-parabolic PDEs 
do not satisfy the % physical !!!
desirable % required 
principle of \emph{locality} % in time ad space \cite{muller-ruggeri-1998}
(i.e. % equiv.
motions are not genuinely causal)
because information % perturbations % of a fluid at rest => linearization !
% or the full initial data (when compactly supported)
propagates at \emph{infinite} speed. % throughout the material body
% burger's kernel (assume uniform density: liu % https://web.math.sinica.edu.tw/bulletin_ns/20171/2017103.pdf
%
Now, locality is important % in some applications, 
% even though the speed of propagation is very fast compared to the rate of change in most mechanical situations
in geophysics % liquid flows for instance
e.g. when unstationary % time-dependent 
processes associated with internal friction % in viscous flows
obviously % intuitively %  
have a local character (the migration of suspended particles, the production of turbulent energy\dots).

\smallskip %%% On the contrary

In this work, to model viscosity in % \emph{viscoelastic} continuum
fluid flows, % as an asymptotic concept following Poisson and Maxwell, 
we follow Maxwell's approach and we look for a good % useful 
(hyperbolic) viscoelastic model. % for viscoelastic flows
% which % that 
% satisfies the locality principle
% (see below).

\smallskip 

Many viscoelastic % multi-dimensional
models have % indeed already
been proposed % in the past
after Maxwell, % with a view to unify solid and fluid behaviour and also for multi-dimensional fluid flows 
in particular to explain % with viscoelastic % fluids % models (multi-dimensional) 
non-Newtonian % multi-dimensional !!
flows % observed experimentally % zaremba fromm weissenberg
of polymeric rubber-like liquids after \cite{dewitt-1955} 
that are mostly steady. % see also our section~\ref{sec:viscoelastic}
% \cite{joseph_narain_riccius_1986,joseph_riccius_arney_1986,bird-curtiss-armstrong-hassager-1987a}
% where the elasticity modulus is more easily determined than in water \cite{noirez-baroni-2012}.
% to that aim
For multi-dimensional flows, there is now a consensus about the need for a rheological equation % of state
with % the flawed concept of
\emph{objective derivatives}, like the famous \emph{Upper-Convected Maxwell} (UCM) equation.
%
% Note that in UCM fluids, the polymeric viscosity is an % purely 
% asymptotic concept, % like in Maxwell seminal 1867 model
% and % supposedly-causal 
% compressible motions are local. % satisfy the locality principle 
%
But flow models with UCM are usually formulated as % general 
quasilinear % hyperbolic 
systems % of first-order PDEs, possibly degenerate (with incompressibility)
without % much
more % mathematical 
structure; % in any case, for incompressible or more general compressible motions 
and solutions to Cauchy problems have % naturally !!
remained difficult to analyze % mathematically 
or simulate % numerically after discretization
beyond 1D. % dimension 1
% in 1D \cite{bouchut-boyaval-2013} % isentropic / isothermal  +  conservation laws !!
% or 2D \cite{PHELAN1989197,EDWARDS1990411}. % olsson-ystrom-1993,Guaily2010158 % \emph{slightly compressible} -- mass and energy are not conserved
% except at the price of destroying its mathematical nature
In practice, the % multi-dimensional 
UCM models -- mostly used % only 
for incompressible flows -- are often modified with an additional ``background viscosity'' (equiv. a % so-called 
retardation time) e.g. as in the Oldroyd-B model, which % cannot satisfy
spoils the local character of Maxwell's  model. % for well-defined motions
% See section~\ref{sec:viscoelastic} for a more detailed review after has been mathematically set 
%
See \textbf{Section~\ref{sec:continuum}} for % the standard modelling % problem % of compressible UCM fluids
more details about % continuum mechanics and 
standard viscoelastic %flow
models.

\smallskip %%%

In this work, we propose % develop 
the first formulation of the % non-isothermal
compressible % viscoelastic 
UCM model as a \emph{symmetric-hyperbolic system of conservation laws} % to our knwoledge
in \textbf{Section~\ref{sec:model}}.

\smallskip %%%

Starting with the % polyconvex 
elastodynamics system % of physically-meaningful conservation laws for polyconvex hyperelastic materials 
like the K-BKZ theory for viscoelastic % fluid 
models \cite{kaye-1962,kaye-1963,bernstein-kearsley-zapas-1963,BKZ1964},
a new system of physically-meaningful conservation laws is proposed % for the Eulerian description of viscoelastic UCM flows
for the compressible UCM model in section~\ref{sec:conservation}.

In section~\ref{sec:strictlyconvex}, it is then proved that the system is symmetric-hyperbolic, 
using conservative variables adequate for the application of 
Godunov-Mock theorem.
% We believe that 
Recall that % the formulation as
symmetric-hyperbolic systems of conservation laws % is 
are essential % useful 
to the analysis % definition of solutions within a ``stable'' class % cite benzoni gavage ??
and to the numerical simulation of % Cauchy
solutions to quasilinear systems \cite{benzonigavage-serre-2007},
and to polyconvex elastodynamics in particular \cite{dafermos-2000,wagner-2009,bonet-gil-ortigosa-2015}. % christoforou-galanopoulou-tzavaras-2019

The new system is not simply a sound mathematical % and numerical 
framework for the % compressible 
viscoelastic models % still 
under development \cite{mackay-philips-2019}. % Bollada2012
% we believe that 
It is also one particular \emph{viscoelastic} case % example 
in \emph{a class of % physically-interesting
mathematically-sound models that unifies the hyperelastic solids with viscous fluids}.
%
% although accurately matching various phenomenologies like compressibility % by the definition of pressure in our viscoelastic framework
% and distortion simultaneously in a continuum model remains a challenge for physicists developing meaningful stored energy functional
% for continuum mechanics. % like in thermoelasticity

In section~\ref{sec:physics}, we show that the new system 
% for the compressible % non-isothermal
% viscoelastic UCM model 
has not only a physical interpretation % inline with
as one extension of the polyconvex elastodynamics system (usually modelling \emph{solids}),
but also \emph{one particular extension towards fluids},  
that uses % the help of 
an additional % \emph{non-uniform} 
material metric variable like other well-known extensions (e.g. the elastoplastic systems). % a metric evolving in time and relaxing
That latter interpretation shows the potentialities of the new symmetric-hyperbolic system of conservation laws,
to soundly unify the solid and fluid dynamics of various materials. 
% in a % mathematically and physically 
% sound way. % well-posedness plus causality and locality
%
% In the new system for UCM fluids, one particular relaxation limit has been chosen for the new metric variable,
% so that the model matches Newtonian fluids in a formal asymptotic regime, like % similarly to 
% standard UCM formulations.
% %
% But in fact, other physically-interesting relaxation limits could be chosen, to match other materials with inelasticities.
% %
% For instance, insofar as polyconvex elastodynamics % thermoelasticity ?
% is also a similar basis for extensions toward standard elasto-plastic models 
% (with additional variables) see e.g. \cite{MR3633756} % cite PRESTON ?????,
% % (with an additional % inelastic 
% % deformation-gradient variable % though
% % that can be % equivalently ?
% % replaced by an additional metric variable) % with a relaxation limit different from UCM
% our % new 
% symmetric-hyperbolic system of conservation laws also seems interesting for elasto-plasticity. % in particular
% % as well as a connection between elasto-plasticity (rate-independant inelasticity) and viscoelasticity (rate-dependant inelasticity).
% %
% In section~\ref{sec:physics}, we give tracks for future investigations in that direction.
% However, a more precise formulation, with other inelasticities than viscoelasticity, is out of the scope of the present article.

Unifying fluid and solid dynamics % of various materials 
has of course been the goal of many previous works in the literature, and it is not the aim of the present work to review and compare them with our new system.
Here, unification is simply mentioned as % a quality
a potentiality of our new system.
Let us nevertheless mention the recent work \cite{PESHKOV2019481}.
As for unification, that work is the only one we are aware of which, like ours, 
first looks for \emph{a symmetric-hyperbolic system of conservation laws} extending polyconvex elastodynamics to viscoelastic Maxwell fluids.
In comparison with \cite{PESHKOV2019481}, we extend polyconvex elastodynamics 
to a \emph{hyperbolic quasilinear system with a different structure}, using a \emph{different additional variable}.

\smallskip

Last, we believe our new system will have very useful applications in the shallow-water regime, % in geophysics
% as a viscoelastic extension of Saint-Venant 2D models, 
to model free-surface % incompressible 
gravity flows % where large vortices develop and dissipate because of
with viscosity. % in presence of vortices ! 
% like many 2D shallow-water models proposed after Saint-Venant to that aim.

\smallskip

In \textbf{Section~\ref{sec:2D}}, we precisely show how our new system can be reduced to a symmetric-hyperbolic system of 2D conservation laws
that is a physically-meaningful viscoelastic extension of Saint-Venant models.
The % latter 
new 2D system encompasses % and improves 
our former viscoelastic extensions of Saint-Venant models with F.~Bouchut \cite{bouchut-boyaval-2013,bouchut-boyaval-2015,boyaval-hal-01661269},
without a conservative formulation in 2D.
%
% The new system % Saint-Venant-Maxwell (SVM) model
% contains the latter systems, as well as the standard shallow-water system of Saint-Venant \cite{saint-venant-1871} in the zero elasticity limit $\mu \to 0$.
% % The reason why it contains
% To encompass the limitations of our previous models, % \cite{bouchut-boyaval-2015,boyaval-hal-01661269} is that 
% we have % been able to 
% interpreted them as particular ``closed'' hyperbolic subsystems of the new model.
% %
% We have generalized the elastodynamics system of hyperelastic materials to Maxwell fluids thanks to a new state variable $\bA$ to that aim, 
% and our new model thus also contains standard elastodynamics equations when $\bA$ is uniform in space and time 
% (i.e. at large \emph{Deborah, or Weissenberg number} $\lambda\gg1$).

% Note that 
Developping 2D shallow-water models for free-surface flows with large vertical vortices and viscous dissipation
has also been the goal of many previous works in the literature, see \cite{ferrari-saleri-2004,bouchut-boyaval-2015}
and references therein.
Again, it is not the goal of the present article to review and compare those numerous 2D works with ours. % new system for viscoelastic flows.
Here, we simply mention an important application of our new 3D UCM system, which delivers a \emph{symmetric-hyperbolic} system of 2D conservation laws % after Saint-Venant reduction
in contrast to \cite{gavrilyuk-ivanova-favrie-2018} and our former works \cite{bouchut-boyaval-2015,boyaval-hal-01661269} e.g., 
see details in Section~\ref{sec:2D}.

\mycomment{ \smallskip %%%%%%%%%%%%%%%%%%%%%%%%%%%%%%%%%%%%%%%%%%%%%%%%%%%%%%%%%%%%%%%%%%%%%%%%%%%%%%%%%%%%%%%%%%%%%%%%%%%%%%%%%%%%%%%%%%%%%%%%%%%%%%%%%%

In Section~\ref{sec:continuum}, we mathematically set the mechanical framework
and we recall the standard % formulation of the 
compressible UCM model
for viscoelastic flows % along with terminology and notations
as it arises in continuum mechanics from particular % specific
constitutive assumptions.

In Section~\ref{sec:model}, we introduce our new formulation of the compressible UCM model with a system of conservation laws.
We prove that it is indeed symmetric-hyperbolic, using adequate conservative variables
and a variant of Lieb convexity theorem \cite{lieb-1973}.
We also discuss straightforward mathematical consequences (well-posedness of a Cauchy problem)
as well as some other properties.
In particular, the new formulation is interpreted physically in section~\ref{sec:physics},
which shows many potentialities. % in future works

In Section~\ref{sec:2D}, we detail the application in computational geophysics which extends the % predictivity
Saint-Venant equations to causal viscoelastic % energy-dissipative
flows.

} %%%%%%%%%%%%%%%%%%%%%%%%%%%%%%%%%%%%%%%%%%%%%%%%%%%%%%%%%%%%%%%%%%%%%%%%%%%%%%%%%%%%%%%%%%%%%%%%%%%%%%%%%%%%%%%%%%%%%%%%%%%%%%%%%%%%%%%%%%%%%  

\section{Viscoelastic flows in continuum mechanics} % constitutive assumptions in continuum mechanics ??
\label{sec:continuum}

First % We 
recall % the modelling issue and its answer as standard
\emph{standard} viscoelastic constitutive assumptions to model smooth % realistic, possibly
compressible material % behaviours
fluid motions (equiv. flows) % with stress relaxation
in % the 
continuum mechanics setting. % theory, by MAXWELL (flow) not KELVIN-VOIGT !

% Mathematical setting for modelling problem
\subsection{Continuum mechanics needs constitutive assumptions} % laws %%%%%%%%%%%%%%

Continuum mechanics aims at modelling the motions of ``matter'' % endowed with mass !
as flows of ``continuous bodies'' at the human scale % in time and space equipped with adequate mathematics 
(unlike % as opposed to 
``discrete particles'' at the molecular scale). % for instance
A prerequisite is the definition of material bodies % mathematical setting dates back to Lodge according to Truesdell Noll MR0193816 section 15 p.37 
and their flows.

The classical % Newtonian
theory considers % open, bounded, connected
bodies $\Bcal$ % \subset\RR^3
that are % $d$-dimensional with $d=3$
Riemannian manifolds, % (usually 3-dimensional), 
and flows that are collections of
``configurations'' i.e. mappings $\bphi_t(\Bcal)$ indexed by % functions of 
time $t\in\RR$ % absolute and given, independently of mass !
into the Euclidean % manifold 
ambiant space \cite{marsden-hughes-2012-mathematical}. % equipped with the usual Euclidean metrics % degenerate from 4D viewpoint
% $\Bcal$ can be a reference configuration at reference time $t=0$ say
For future reference, recall that 
on % the Riemannian manifold 
bodies $\Bcal$ % equipped
with a % local Cartesian 
coordinate system $\{a^\alpha\}$ % with basis vectors $\be_\alpha$ % , $\alpha=1\dots 3$ \cite[Box 2.1 p. 44 Chap 1]{marsden-hughes-2012-mathematical}
and % a family of positive-definite metric % symmetric % 2-covariant % 2-tensor $G^{\alpha\beta}$ 
a material (or body) metric defined by a positive symmetric 2-tensor $G_{\alpha\beta}\in S^+(\RR^{d\times d})$ ($d=2,3$),
% the divergence 
$\div_{\ba}\bv = % \frac1{\sqrt{|G|}}
\partial_\alpha(\sqrt{|G|}v^{\alpha\beta\ldots})/\sqrt{|G|}$ for $\bv(\ba)=v^{\alpha\beta\ldots}\be_\alpha\otimes\be_\beta\dots$ % divergence acts on first component
is well-defined % in the non-singular case
when $G_{\alpha\beta}\in S^{++}(\RR^{d\times d})$ % a symmetric positive \emph{definite} matrices of $\RR^{d\times d}$
i.e. the determinant is stricly positive $|G_{\alpha\beta}|>0$, 
and an inverse metric $G_{\alpha\beta}G^{\beta\gamma}=\delta_{\alpha\gamma}$ exists
-- $\delta_{\alpha\gamma}$ denoting Kronecker's symbol --. % \cite{marsden-hughes-2012-mathematical}
% Most often that latter metric is Euclidean, so $\div_{\ba} \bS = \partial_\alpha S^{i\alpha}$.
% $\div$ denotes the divergence on the manifold $\bB$ with material metrics $g$: \partial_\alpha (sqrt|g| S^{i\alpha}) / sqrt|g|

Next, one establishes % achieve 
a % quantitative 
precise description of bodies motions i.e. ``flows'' using % based on a number of 
axioms  % that make precise the physical principles (postulated), at the human scale and ``below''
and assumptions. % mainly about the ``chemical nature'' (the microscopic content below human scale)
Viscoelastic % fluid
flows arise % as a particular case well-suited for viscoelastic matter % ie well-suited for some specific materials
from particular \emph{constitutive assumptions}, % needed to close the theory
see section~\ref{sec:viscoelastic}.
But let us first recall % briefly % basics concepts of % classical
the continuum mechanics % theory
setting and simpler consitutive equations
(some % fundamental 
notions need to be assumed though, like those in quotes ``\dots'',
and we refer to % modern accounts of the theory in the monographs 
\cite{dafermos-2000,marsden-hughes-2012-mathematical,wang-truesdell-1973}
% MR0193816 Truesdell, C. and Noll, W. 1965 => article instead ?? % silhavy1997book
for % a more complete exposition
more details). 

\smallskip

% Classically, % to start with
Given a force field $\f$ in the Euclidean ambiant space with a coordinate system $\{x^i\}$,
one assumes % that
a Galilean frame-invariant balance of total % non-specific !! % per unit volume, not mass
energy $E\ge 0$ % a positive total extensive field % in Euclidean space, % !! not general manifolds
holds as follows for bodies, % (written with Einstein convention)
with $R$ the % specific 
heat supplied during the process:
\begin{equation}
\label{eq:totalenergy_balance}
\partial_t (E\circ\bphi_t) = \div_{\ba}( S^{i\alpha}\partial_t\phi^i_t ) + \partial_t\phi^i_t (f^i\circ\bphi_t) + R \,.
\end{equation}
Then, % one can show that 
bodies 
are % need be !!
characterized by a % positive Radon measure
% constant in time
mass-density % field 
$\hat\rho(\ba)\ge0$, and their motions % described by % (time-indexed) configurations 
$\bphi_t:\ba\in\Bcal\to\bx=\bphi_t(\ba)\in\RR^d$ ($d=2,3$)
satisfy the \emph{momentum balance}: % \cite{Silhavy1989,Fosdick-Serrin2014,fosdick-2019,DiCosmo2019}
\begin{equation}
\label{eq:momentum_material}
\hat\rho (\partial^2_{tt} \bphi_t) = \div_{\ba} \bS % \div_{\bA}
+ \hat\rho (\f\circ\bphi_t)
\end{equation}
where $\bS$ is the % so-called 
(first) Piola-Kirchoff stress tensor, $S^{i\alpha}$ in coordinates.
For non-polar bodies, % because of invariance first principles !!
it holds $S^{i\alpha}\partial_{\alpha}\phi^j=S^{j\alpha}\partial_{\alpha}\phi^i$ % balance of angular momentum for Euclidean material metrics
and, on introducing $r\circ\bphi_t = R/\hat\rho$, % the \emph{first law} (of % thermomechanics / thermodynamics):
\begin{equation}
\label{eq:firstlaw_material}
\hat\rho (\partial_t e\circ\bphi_t) - S^{i\alpha} \partial^2_{t\alpha} \phi^i_t % \partial_{\alpha} u^i in Eulerian description
= \hat\rho (r\circ\bphi_t)
\end{equation} % i.e. the \emph{first law} of thermodynamics 
where $ e\circ\bphi_t:= % \frac1{\hat\rho} 
E\circ\bphi_t / \hat\rho - \frac12 | \partial_{t}\bphi_t |^2 $ is the % (specific)
internal % (or stored) 
energy. % source from the environment, not to be confused with -D our supply to
% for the sake of invariance first principles !!
Note that % in \eqref{eq:firstlaw_material} and the sequel, 
we assume \emph{adiabatic processes} 
(i.e. no heat flux % due to conduction
within bodies, which are assumed % not to conduct 
heat insulators), 
and we use Einstein summation convention for repeated indices. % vector through the body

Next, if \emph{constitutive assumptions} specify $e$ as a function of $\partial_{\alpha} \phi^i_t$ -- thus also $\bS$ by \eqref{eq:firstlaw_material} --,
% to define a closed system of partial differential equations
% body 
motions $\bphi_t$ can be % fully described
defined as solutions to % the initial value problem for
\eqref{eq:momentum_material} for % small times 
$t\in[0,T)$ % at least 
given % a smooth initial condition 
$\phi_{t=0}=\phi_0$. % (at $t=0$)
Some constitutive assumptions % when wisely used with computer 
and well-defined % numerically computed 
motions have shown the practical interest of the theory for % specific 
applications % in a number of particular situations
to various materials, % solids or fluids 
% after comparing with experiments
see e.g. \cite{leveque-2002,bonet-gil-ortigosa-2015}. % Zienkiewicz2000,pironneau-1989
% However,
But specifying constitutive assumptions that are both mathematically and physically meaningful % even % for particular applications
is % has remained 
a difficult % challenging 
task since the beginning of the theory. %  in the nineteenh century: Cauchy, Lagrange, Euler
Despite many rationalization efforts % like \cite{wang-truesdell-1973} notably
guided by mathematical soundness, % MR0193816 Truesdell, C. and Noll, W. 1965 => artuicle instead ?? % silhavy1997book,muller-ruggeri-1998
we are not aware of a definitive approach % to constitutive modelling
to model particular real materials
(many practical
% A huge number of 
constitutive assumptions % derived following different approaches, with different principles
exist, scattered in a vast literature).
We recall standard constitutive assumptions 
% $$
% e(\partial_\alpha \phi^i) \quad D(\partial_\alpha \phi^i)
% $$
for \emph{viscoelastic fluids} % flows of ...
in section~\ref{sec:viscoelastic}. 

\smallskip 

In section~\ref{sec:constitutive}, we first recall % more 
fundamental % standard 
constitutive assumptions for elastic and viscous material bodies in the ``solid'' or ``fluid'' states,
when $e$ is function of $\partial_\alpha\phi^i_t$ or $|\partial_\alpha\phi^i_t|$.
Viscoelasticity arises as a % one possible
unifying concept in between. % \cite{noll-1955}
We consider \emph{smooth} motions %flows 
$\bphi_t$, % defined as C^k 
diffeomorphisms with inverse $\bphi_t^{-1}$, and we denote: 
\begin{itemize}
 \item $F^i_\alpha := \partial_\alpha \phi^i_t \circ \bphi_t^{-1}$ the deformation gradient in component form given two coordinates systems $\{x^i\}$ and $\{a^\alpha\}$,
 i.e. the matrix representation of the % the second-order 
 tensor $\bF=F^i_\alpha \be_i\otimes\be_\alpha$ % abraham-marsden-ratiu
 with rows labelled by a \emph{Roman} letter like $i,j,k\dots$ to precise % vector
 coordinates in the \emph{spatial} frame 
 % (contravariant, thus in superscript)
 and with columns labelled by a \emph{Greek} letter $\alpha,\beta,\gamma,\dots$ to precise % derivative 
 coordinates in the \emph{material} frame 
 % (covariant, thus in subscript)
 \item $|F^i_\alpha|$ the determinant of $F^i_\alpha$, also sometimes denoted $|\bF|$ % it is invariant by frame changes
 \item $C^\alpha_i$ the cofactor matrix (or transpose adjugate) of $F^i_\alpha$ % = $|F^i_\alpha|\bF^{-T}$ in matrix notation
 \item $u^i := \partial_t \phi^i_t \circ \bphi_t^{-1}$ the % material
velocity
 \item $D(u)^{ij} := \frac12\left(\partial_i u^j + \partial_j u^i\right)$ the strain-rate % or rate-of-strain (i.e. symmetrized velocity gradient)
tensor
 \item $\div\bu=\partial_iu^i$ the Euclidean divergence for a vector field $\bu$ and
 \item $\delta$ the identity tensor compatible with the Kronecker symbol notation in coordinates
so $\delta_{i}^j = |F^i_\alpha|^{-1}F^j_\alpha C^\alpha_i$ for instance.
\end{itemize}
We classically assume % as usual 
that \eqref{eq:totalenergy_balance}--\eqref{eq:momentum_material} % $\bphi_t$
are % obtained as 
the Euler-Lagrange equations of a % Hamiltonian
variational principle for a Lagrangian density $\hat\rho\left(\frac{1}2|\partial_t \bphi_t|^2- e\circ\bphi_t\right)$ with $e$ a function of $\partial_\alpha\phi^i_t$,
\cite{marsden-hughes-2012-mathematical}, % which also implies \eqref{eq:totalenergy_balance}.
and $G_{\alpha\beta}=\delta_{\alpha\beta}$ for simplicity.
Then, $\bS$ is a function of $F_\alpha^i \circ\bphi_t$ i.e. 
\begin{equation}
\label{eq:2Dpiolakirchhof}
S^{i\alpha}  = \hat\rho \partial_{ F^i_\alpha } e
\end{equation}
so \eqref{eq:firstlaw_material} holds with $r=0$. And \eqref{eq:momentum_material} rewrites % as part of a closed first-order 
within a
system of % first-order
conservation laws:
\beq 
\label{eq:firstorder_material} % {eq:SVUCM0detaillag}
\begin{aligned}
& \partial_t \left( \hat\rho \: u^i\circ\bphi_t \right) - \partial_\alpha S^{i\alpha} = \hat\rho f^i
\\
& \partial_t \left(F^i_\alpha\circ\bphi_t \right)- \partial_\alpha \left( u^i\circ\bphi_t \right)= 0
\\
& \partial_t \left(|F^i_\alpha|\circ\bphi_t \right)- \partial_{\alpha}\left( C^\alpha_j\circ\bphi_t \: u^j\circ\bphi_t \right) = 0
\end{aligned}
\eeq 
% in material coordinates  
that fully defines \emph{causal} motions in the so-called material (or Lagrangian) description 
% insofar as % it consists % e.g. in 
% the evolution system \eqref{eq:firstorder_material} % is closed and has
as semi-group solutions, % on small times
possibly after % \eqref{eq:firstorder_material} has been complemented by 
adding % the conservation law 
\eqref{eq:cofGevolution_material} to \eqref{eq:firstorder_material} when $d=3$
\beq
\label{eq:cofGevolution_material}
\partial_t \left( C^\alpha_i\circ\bphi_t \right) + \sigma_{ijk}\sigma_{\alpha\beta\gamma}\partial_\beta \left( F^j_\gamma\circ\bphi_t \: u^k\circ\bphi_t \right) = 0 %(1.9) dans wagner-2009
\eeq
where $C^\alpha_i = \sigma_{ijk}\sigma_{\alpha\beta\gamma}F^{j}_{\beta}F^{k}_{\gamma}$, and
$\sigma$ is Levi-Civita's symbol --  % for the sign of a permutation
so it holds e.g.
$$
|F^i_\alpha|= \sigma_{ij}\sigma_{\alpha\beta}F^{i}_{\alpha} F^{j}_{\beta}
\qquad
C^\alpha_i=|F^i_\alpha|\sigma_{ij}\sigma_{\alpha\beta}F^{j}_{\beta}
$$
when $d=2$.
Moreover, when $\hat\rho$ is constant, % for \emph{homogeneous} materials
smooth motions with a material (or Lagrangian) description % of motions % Hamiltonian with a Poisson map !!
also have a spatial (or Eulerian) description: % a ``reduced'' Hamiltonian given usual constitutive assumptions 
% which is also a system of % first-order
% conservation laws % for configuration fields functions of $t$ and $\{x^i\}$
% in spatial coordinates: 
\beq 
\label{eq:firstorder_spatial} 
\begin{aligned}
% & \rho ( \partial_t +  u^j\partial_j) u^i - \partial_j \sigma^{ij} = \rho f^i
& \partial_t \left( \rho u^i \right) +  \partial_j \left( \rho u^j u^i - \sigma^{ij} \right) = \rho f^i
\\
%& |F^i_\alpha|^{-1} ( \partial_t +  u^j\partial_j) F^i_\alpha - \partial_j \left( |F^i_\alpha|^{-1} F^j_\alpha u^i \right) = 0
& \partial_t \left( \rho F^i_\alpha \right) +  \partial_j \left( \rho u^j F^i_\alpha - \rho F^j_\alpha u^i \right) = 0
\\
% & \partial_t |F^i_\alpha|^{-1} + \partial_j\left( |F^i_\alpha|^{-1} u^j \right) = 0 
& \partial_t \rho + \partial_j\left( \rho u^j \right) = 0 
\end{aligned}
\eeq 
with Cauchy stress $\sigma^{ij} := |F^i_\alpha|^{-1}F^j_\alpha S^{i\alpha}\circ\bphi_t^{-1}$ function of $F^i_\alpha$,
and % mass density
$\rho := |F^i_\alpha|^{-1} \hat\rho$ % \circ\bphi_t^{-1}
\cite{wagner-2009},
possibly complemented when $d=3$ % in dimension 3 when $C^\alpha_i = \sigma_{ijk}\sigma_{\alpha\beta\gamma}F^{j}_{\beta}F^{k}_{\gamma}$ 
by % the conservation law 
\beq
\label{eq:cofGevolution_spatial}
% |F^i_\alpha|^{-1} ( \partial_t +  u^j\partial_j) C^\alpha_i 
% + \sigma_{ijk}\sigma_{\alpha\beta\gamma}\partial_l \left( |F^i_\alpha|^{-1} F^l_\beta F^j_\gamma u^k  \right) = 0 \,, %(1.9) dans wagner-2009
\partial_t \left( \rho C^\alpha_i  \right) +  \partial_j \left( \rho u^j C^\alpha_i  \right)
+ \sigma_{ijk}\sigma_{\alpha\beta\gamma}\partial_l \left( |F^i_\alpha|^{-1} F^l_\beta F^j_\gamma u^k  \right) = 0 \,. %(1.9) dans wagner-2009
\eeq
The Lagrangian and Eulerian descriptions of \emph{smooth} motions are equivalent as long as the following Piola's identities hold \cite{wagner-2009}: %  at least
\begin{equation}
\label{eq:2Dpiola_identity}
% \partial_{\alpha}( \sigma_{\alpha\beta} F^k_\beta ) = 0 \quad
\partial_{j}( |F^i_\alpha|^{-1} F^j_\alpha ) = 0 \quad \forall \: i=1\ldots d.
\end{equation}
% when $F^j_\alpha = \partial_\alpha \phi^j_t$ an invertible deformation gradient. % bi-Lipshitz diffeomorphisms : reciprocally \partial_{\alpha}( \sigma_{\alpha\beta} F^k_\beta ) = 0

% IT IS NOTEWORTHY THAT THE EULERIAN DESCRIPTION MOST OFTEN ALSO COINCIDES WITH A (REDUCED) VARIATIONAL PRINCIPLE

% \begin{itemize}
%  \item 
%  \item 
%  \item 
% \item $\tr(B)=B^{ii}$ the trace of a matrix $B^{ij}$.
%  which, unless specified otherwise, is % the inverse
%  Finger's % left Cauchy-Green pushed-forward
%  deformation tensor % in coordinates
%  $B^{ij} = F^i_\alpha G^{\alpha\beta} F^j_\beta$ given a 
% \end{itemize}

% A reminder of
\subsection{Constitutive assumptions for elastic bodies and fluids} 
\label{sec:constitutive}

\paragraph{Elastic motions} % also \emph{elastic} flows or deformations of material bodies
have been considered since the beginnings of continuum mechanics % Cauchy, having in mind applications to
for ``solids'' % materials
\cite{wang-truesdell-1973,maugin2015continuum}. % landau-lifschitz-7,maugin2014continuum maugin2013continuum
Some elastic constitutive assumptions % have been proposed to that aim which 
efficiently summarize the molecular structure of matter % for description 
at a human scale and are useful to predict real solid behaviours.
In particular, smooth motions of \emph{hyperelastic} % Green elastic according to \cite{ogden1984book}
materials with % a stored / internal 
an energy % function 
$e(F^i_\alpha\circ\bphi_t)$ % (recall $\partial_\alpha phi^i_t= F\circ\bphi_t$) that is \emph{polyconvex} in $F^i_\alpha$
can be well defined when $r=0$ % heat cannot be supplied to purely mechanical hyperelastic bodies without thermal properties, so it necessarily holds 
as % strong 
solutions to (a Cauchy problem for)
either the second-order equation \eqref{eq:momentum_material} 
\cite{john-1988}, % in the small, or local-in-time % ,sideris-2000,agemi-2000 
or a first-order system of conservation laws: 
\eqref{eq:firstorder_material} in material coordinates, % with a convex extension ! this is stronger than ellipticity required by Dafermos theory of involutions dafermos-2000
or \eqref{eq:firstorder_spatial} in spatial coordinates,
e.g. when $e$ is \emph{polyconvex} and both % systems 
are symmetric-hyperbolic \cite{wagner-2009}.
% {eq:firstorder_material} the \emph{elastodynamics} where many math questions remain though \cite{MR1919825} % ball 2002

Postulating indifference % of e
to Galilean changes of spatial frames as usual in classical physics % implies
requires that $e$ is function of $F^i_\alpha $ through the right Cauchy-Green deformation tensor 
$F^k_\alpha F^k_\beta$. % in a Euclidean space
% It also depends on the material body metric non-necessarily Euclidean !!
%
Then, for homogeneous isotropic bodies with % constant 
$G^{\alpha\beta}=\delta^{\alpha\beta}$ % uniform i.e. a 
Euclidean, % metric % $G_{\alpha\beta}=\delta_{\alpha\beta}$) 
a useful % example of 
polyconvex energy is the \emph{neo-Hookean} % model proposed historically by Rivlin, phenomenologically (isotropy = materially-covariant by SU)
\begin{equation}
\label{eq:neohookean}
e(F^k_\alpha F^k_\alpha):=\frac{\mu}2 ( F^k_\alpha F^k_\alpha % G^{\alpha\beta} F^i_\alpha F^i_\beta 
-d)
\end{equation}
% recall the left Cauchy-Green deformation tensor $B^{ij} = F^i_\alpha G^{\alpha\beta} F^j_\beta$ % in coordinates
% so = $e=C_1(\tr(C)-d)$ with $C^{\alpha\beta} = F^i_\alpha F^i_\beta$ the right Cauchy-Green deformation tensor 
% \cite{ogden1984book} a consequence of material covariance without more structure than the right Cauchy-Green deformation tensor
with both molecular and phenomenological justifications % derivations 
\cite{Gloria2014}. % justifications Blanc2002

The neo-Hookean model is simplistic, but it is already quantitativaly useful for practical applications.
% numerically % \cite{MR3283786,MR3452769} Bonet 2015
Moreover, it has many % physically-interesting 
refinements. % like all purely hyperelastic models
For instance, the neo-Hookean model % has only one coefficient $\mu$ and
cannot capture volumetric changes observed simultaneously with elongation.
% e.g. in traction tests, uniaxial or biaxial extensions, typically used to fix $\mu$ as Lam\'e first coefficient or as 2\mu  Young modulus ?
%  
% Saint-Venant-Kirchhoff \frac{\lambda}2 (E^{ii})^2 + \mu E^{ij}E^{ij} using $E^{ij} = (F^i_\alpha-\delta^i_\alpha)(F^j_\alpha-\delta^j_\alpha)$ and Lam\'e coefficients
%
But one can either use the model along with the % additional 
incompressibility % penalization
constraint $|F^i_\alpha|=1$ (if relevant) as a remedy.
Or one can add a \emph{compressible} term function of $|F^i_\alpha|$ in the energy \eqref{eq:neohookean} 
% with one more coefficient to calibrate ! % Lam\'e second coefficient or Poisson ratio
that preserves polyconvexity. % \cite{charrier-dacorogna-hanouzet-laborde-1988,murphy-rogerson-2018}. 
% flory-1961
% Yet, % an additive modification remains to be justified physically % since Flory ?? see e/g/ LION2014729
% it has remained difficult to derive a realistic stored energy function that can match many phenomenologies. % (more than two tests, say)
% %
% The volumetric changes for instance % in particular 
% are often linked with other molecular processes than those explaining % pure 
% elastic elongations.
% %
% And various molecular processes % occuring at the same time and same place
% are hardly well summarized by % a single model like the neo-Hookean model with
% one single coefficient $\mu$. % alone 
% % microscopic details at the molecular scale neglected in the continuum framework
% % incl consequences like possible \emph{non-reversible} deformations that ensue at the human scale
% %
% In particular, the heat exchanges $r\neq 0$ that occur simultaneously % as observed since Joule in rubber \cite{treloar2005physics}
% with non-reversible volumetric changes % even for adiabatic bodies as considered here
% are not accounted for. % so far by the neo-Hookean model.

%, as a correction, 
Non-reversible motions with $r\neq 0$ % as expected in non-reversible motions from the definition of a given $e$
can moreover % also % still % however 
be considered % in continuum mechanics
when % the internal energy 
$e$ is a function of $F^i_\alpha$ and % a \emph{thermodynamical potential} termed 
\emph{entropy} $\eta$ such that % additional constitutive assumption !!!!!!!!!
% at the same time as the first law a \emph{second law} with dissipation $D\ge0$ holds
% thermodynamics second principle, which provide a selection principle
it holds for some dissipation $D\ge0$:
\begin{equation}
\label{eq:secondlaw_material}
r\circ\bphi_t = (\theta\circ\bphi_t) \partial_t( \eta\circ\bphi_t ) - D\circ\bphi_t \,.
\end{equation}
% Note that
Usual elongations with % simultaneous 
volumetric changes are indeed non-reversible, % even for adiabatic bodies as considered here
with heat exchanges $r\neq 0$; % as observed since Joule in rubber \cite{treloar2005physics}
% The second law 
\eqref{eq:secondlaw_material} means that the heat supply may be 
either dissipated by irreversible processes % at an another scale
(``inelasticities'')
or compensated for by variations in the body state (through entropy).
% Clausius, Rudolf. The Mechanical Theory of Heat. London: Taylor and Francis, 1867. eBook
Using % the second law 
\eqref{eq:secondlaw_material} as an additional constitutive assumption 
leads one to introduce the \emph{temperature} % of the body
$\theta=-\partial_\eta e$ 
% for reversible processes $r = \theta\partial_t \eta$, $D=0$
\cite{Coleman-Noll1963}. % ``free variations in eta''
% which is consistent with ideal gas law and Maxwell-Boltzmann def of entropy ?? see below
Then, further constitutive assumptions % still needed to fully describe the (possibly non-reversible) \emph{thermo-elastic} motions % flows
about % the mathematical nature of admissible 
inelasticities and $D$ allow to close \eqref{eq:momentum_material} (or \eqref{eq:firstorder_material}, or \eqref{eq:firstorder_spatial})
complemented by \eqref{eq:firstlaw_material}--\eqref{eq:secondlaw_material} when $r\neq 0$.
For instance, 
smooth \emph{isentropic} motions such that $\partial_t(\eta \circ\bphi_t)=0$
can be defined for polyconvex hyperelastic bodies with $e$ jointly convex in $F^i_\alpha$ and $\eta$,
as well as non-smooth % isentropic 
motions like 1D % one-dimensional 
shocks % processes 
using the inequality associated with \eqref{eq:firstlaw_material}--\eqref{eq:secondlaw_material}
% obtained on combining \eqref{eq:firstlaw_material} and \eqref{eq:secondlaw_material}
% to select the physically admissible weak solutions unique when one-dimensional
\cite{dafermos-2000}. % MR1919825 ball review
% (one has to precise the dependence of $e$ as a function of $F^i_\alpha$ -- and $\eta$ !)

Thermo-elastic models % that take volumetric changes into account % as well as isochoric distortions
in fact use the \emph{Helmholtz free energy} $\psi = e - \theta \eta$ as a function of $\theta$ % and $F^i_\alpha$
more often than $e$ as a function of $\eta$, % \cite{chadwick-hill-1974}
% which also defines \eqref{eq:2Dpiolakirchhof} and next $\eta=-\partial_\theta \psi$ % for reversible processes in second law
with a constitutive assumption precising the temperature evolution rather than the entropy evolution.
Then
\begin{equation}
\label{eq:secondlaw_material_bis}
\hat\rho \left( (\eta\circ\bphi_t) \partial_t( \theta\circ\bphi_t ) + \partial_t ( \psi\circ\bphi_t ) \right) - S^{i\alpha} \partial_{\alpha} ( u^i \circ\bphi_t ) % \partial^2_{t\alpha} \phi^i_t 
= -\hat\rho D\circ\bphi_t 
\end{equation}
complements \eqref{eq:momentum_material} (or \eqref{eq:firstorder_material}, or \eqref{eq:firstorder_spatial})
rather than \eqref{eq:firstlaw_material}--\eqref{eq:secondlaw_material}.
% \partial_t e - S^{i\alpha} \partial^2_{t\alpha} \phi^i_t = r
It allows one to define (smooth and non-smooth) \emph{isothermal} motions 
for polyconvex hyperelastic bodies when $\psi$ is jointly convex in $F^i_\alpha$ and $\theta$ 
using $\eta=-\partial_\theta\psi$ and $S^{i\alpha}=\hat\rho\partial_{F^i_\alpha}\psi$. % as definitions 
% equivalent to \eqref{eq:2Dpiolakirchhof} using $e = \psi-\theta \partial_\theta\psi$.
% This is also the case for fluid flows, non necessarily elastic, see below.

% the accurate numerical dscription of
Non-reversible motions however need a more accurate description in many applications. % more than only by their lack of elasticity
And it remains an active research field how to specify % model 
inelasticities,
%
% One can decompose the (irreversible) motions into elastic and inelastic deformations
% with % a constraint , von Mises criterium etc for solids typically
% a % maximum 
% threshold on % admissible 
% reversible motions % (beyond, irreversible motions are typically thought as mediated by particular material defects)
% % such as dislocations, when justified from the molecular level % Wilkins, M.L., 1964. Calculation of elastic-plastic flow. Methods in Computational Physics 3, 211–263
% like in thermo-elasto-plastic models for solids % often rate-independent without history % and not necessarily with explicit temperature variations
% \cite{PESHKOV2019481}. % lee-1969,maugin2015continuum,Hashiguchi2019
% But in many existing (thermo-elasto-plastic) models, 
% the mathematical description of inelastic deformations in irreversible motions has yet to be clarified \cite{MR3633756} 
% and/or validated % for brittle cracks, damage / fatigue etc  specific to materials ??
% to tackle particular defects,
especially over a range of temperatures where the material properties change a lot (throughout phase transitions)
% from solid to fluid for instance 9.4.6 Changes in Phase and Symmetry
and for \emph{large deformations} of flowing bodies when the \emph{fluidity} concept enters \cite{bingham-1922}.
Viscoelasticity is one example of inelasticity.
This will be very clear in Section~\ref{sec:model} with our new UCM system.
We show in section~\ref{sec:physics} that % our formulation of 
the % compressible
UCM % viscoelastic 
model % formulated with a material metric 
is only one viscoelastic instance within a large class of mathematically-sound models with inelasticities.
But first, let us recall % now 
a standard introduction of \emph{viscosity} alone, without elasticity,
as a constitutive assumption for imperfections in irreversible % motions i.e. 
flows of \emph{fluids}. % flows
% Viscosity is indeed an obvious imperfection in real fluids flows,
% and it has been incorporated early in fluid models\dots as well as in some elastic models, 
% recall Maxwell's seminal work \cite{Maxwell01011867}.

% \bigskip %%%%%%%%%%%%

\paragraph{Fluid flows} % for liquids have however also
have long been considered in continuum mechanics. % since the beginnings
The molecular structure of fluids is more difficult to summarize % at human scale % within continuum mechanics
than that of solids, because they are much more deformable.
Useful constitutive assumptions for % the energy $e$ of a fluid  for 
simple enough fluid materials have been proposed
-- though usually without a clear link to solids, % assuming material symmetry with a different symmetry group
the fluid-solid transition % already mentionned above 
being a well-identified difficulty \cite{bingham-1922}. % Cauchy  \cite{euler-1757} 

A useful % purely volumetric 
constitutive law % with high-degree of symmetry ! invariant by the unimodular group !
for ``perfect'' % ideal ``simple'' % with maximal symmetry
fluids % isotropic materially-covariant like nondense``gas'' materials, at extreme opposite of elastoplastic solids that starts flowing
% also homogeneous  !! constant $\hat\rho$ ????
is the \emph{polytropic} law % for \emph{barotropic} fluids
% which is both isotropic % forces are the same whatever the spatial direction
% and indifferent to material frame %  materially-covariant
\beq
\label{eq:polytropic}
e(\rho) % (|F^i_\alpha|) 
:= % \frac{C_0}{\gamma-1} (|F^i_\alpha|\hat\rho)^{1-\gamma} \equiv 
\frac{C_0}{\gamma-1} \rho^{\gamma-1}.
\eeq
% with
%     n=0 for an isobaric process,
%     {\displaystyle n=+\infty } for an isochoric process.
% In addition, when the ideal gas law applies:
%     {\displaystyle n=1} for an isothermal process,
%     {\displaystyle n=\gamma } for an isentropic process. 
% ratio of the heat capacity at constant pressure ( C P {\displaystyle C_{P}} C_{P}) to heat capacity at constant volume ( C V {\displaystyle C_{V}} C_{V}). 
Smooth motions can be % well and conveniently 
defined % including in presence of \emph{vacuum} when $\rho\ge 0$ \cite{chemin-1990} 
with \eqref{eq:polytropic} % as Cauchy solutions to
in the \emph{reduced} spatial % (or Eulerian) 
description % in ambiant space
\beq
\label{eq:firstorder_spatial_reduced}
\begin{aligned}
%& \partial_t ( |F^i_\alpha|^{-1} \hat\rho ) + \partial_i( u^i |F^i_\alpha|^{-1} \hat\rho ) = 0 
%\\
%& |F^i_\alpha|^{-1} \hat\rho \left( \partial_t u^i + u^j \partial_j u^i \right) - \partial_j \left( \sigma^{ij} \right) = |F^i_\alpha|^{-1} \hat\rho f^i
& \partial_t \rho + \partial_i( u^i \rho ) = 0 
\\
& \rho \left( \partial_t u^i + u^j \partial_j u^i \right) - \partial_j \: \sigma^{ij} % \left( \sigma^{ij} \right)
= \rho f^i
\end{aligned}
\eeq
where the Cauchy stress tensor 
% using $S^{i\alpha}=\hat\rho  \partial_{F^i_\alpha} e$ from \eqref{eq:firstlaw_material} % Coleman - Noll theorem 1963 according to Marsden
% in % the first equality of 
% \eqref{eq:cauchystress_pressure}, % Jacobi's formula $d|F^i_\alpha| = C^\alpha_i d(F^i_\alpha)$
% or % more precisely the so-called Doyle-Ericksen formula
% $\sigma^{ij} := 2 \rho F^i_\alpha F^j_\beta \partial_{F^k_\alpha F^k_\beta} e$ % where  F^i_\alpha F^j_\beta = \partial_{\delta_{ij}}F^k_\alpha F^k_\beta 
% for isotropic bodies
reduces to a pure pressure $p\equiv-\partial_{\rho^{-1}}e=C_0 \rho^\gamma$
\beq
\label{eq:cauchystress_pressure}
\sigma^{ij} % := |F^i_\alpha|^{-1}F^j_\alpha S^{i\alpha}
% = |F^i_\alpha|^{-1}F^j_\alpha \hat\rho  \partial_{F^i_\alpha} e =  |F^i_\alpha|^{-1}F^j_\alpha C^\alpha_i  \hat\rho \partial_{|F^i_\alpha|} e
% = \hat\rho \delta_{ij} \partial_{|F^i_\alpha|} e = \delta_{ij} \partial_{\rho^{-1}} e
= -p\: \delta_{ij} \,.
\eeq 
% gas dynamics
The system \eqref{eq:firstorder_spatial_reduced} is indeed symmetric-hyperbolic,
and it is useful e.g. for the dynamics of simple (monoatomic) gases.
But note that \eqref{eq:firstorder_spatial_reduced} is strictly contained in the Eulerian system \eqref{eq:firstorder_spatial},
and motions are not equivalently described by the larger Lagrangian system \eqref{eq:firstorder_material}
which is \emph{not symmetric hyperbolic}. % \cite{DespresMazeran2005}.

Non-smooth irreversible motions can also be considered %  like in the elastodynamics case above
with \eqref{eq:firstorder_spatial_reduced} complemented by \eqref{eq:secondlaw_material} and an entropy variable $\eta$.
When % the internal energy 
$e$ % {eq:polytropic}
is (jointly) convex in $\rho$ and $\eta$, one can consider % more general non-smooth possibly non-reversible 
\emph{isentropic} motions % with possibly non-zero heat supply $r\neq0$
through weak solutions, and define univoque 1D shocks % in particular 
\cite{majda1984book}. % like in the polyconvex elastodynamics case
% full gas dynamics
Isomorphocally, one can define \emph{isothermal} motions % incl 1D shocks univoquely 
using Helmholtz free energy $\psi$, the spatial version of \eqref{eq:secondlaw_material_bis} 
\begin{equation}
\label{eq:secondlaw_spatial_bis}
\rho \left( \eta (\partial_t+u^i\partial_i) \theta + (\partial_t+u^i\partial_i) \psi \right) 
= -\rho D + \sigma^{ij} \partial_j u^i % \partial^2_{t\alpha} \phi^i_t 
\end{equation}
and a temperature variable $\theta$.
%
% But recall one has to precisely charaterize the (inelastic character of the) % thermodynamically-admissible 
% non-reversible motions considered % for fluid flows % more than only by their lack of elasticity
% in addition to precise $e$ or $\psi$ % as a function of $\eta$ or $\theta$ 
% and % the evolution of 
% $\eta$ or $\theta$. % for thermodynamics,
%
However, % similarly to the elastodynamics case
more constitutive assumptions % about the inelastic nature of the deformations 
are often needed to precisely describe % more general 
irreversible fluid motions, % neither isentropic nor isothermal,
like the \emph{vortices} % that are obvioulys inelastic defects insofar as they cannot be described in Lagrangian variables
observed in many \emph{viscous} real fluid flows.
% , with intensity % amplitude 
% controlled by \emph{viscosity} coefficients. % in laminar regime, not turbulent 
%
% Euler equations for ideal fluids \cite{euler-1757} do not % correctly model well, in a quantitative way, with a parameter
% provide control on the vorticity % (vortices appear in flows with shear)
% through materials parameters,
% while the dependence of viscous stresses on % the nature  -- and structure ?? of
% the flowing fluid is easily observed in reality. % REF ?
%
% Recall viscosity has many important consequences on real flows.
% For instance, in discharge predictions, engineers would like to know the % -- often disadvantageous -- 
% amount of % linear
% momentum that is % contained in
% ``absorbed" by vortices. 
% also, some kinetic energy is transformed into heat by friction -- but is it really important ?
%
To that aim, viscous stresses % with a viscosity coefficient, and which produce vortices
have been introduced in % Eulerian flow models like 
\eqref{eq:firstorder_spatial_reduced} % to control (energy dissipation by) vortices !!
by adding an \emph{extra-stress} $\btau$ % of entropic nature
as $\bsigma=-p\delta+\btau$ in % the right-hand side (RHS) 
\eqref{eq:cauchystress_pressure} i.e. 
\beq
\label{eq:extra}
\sigma^{ij} =-p\:\delta_{ij}+ \tau^{ij} % 2\mu D(u)^{ij} + \ell \: D(u)^{kk} \: \delta_{ij}
\eeq
% An admissible extra-stress also needs to be
provided it is ``objective'' (invariant to Galilean change of spatial frames) % in the sense of
and ``dissipative'' i.e. $D:=\tau^{ij} \partial_j u^i % D(u)^{ij} % IN USUAL SYMMETRIC CASE
\ge0$ % to satisfy second law (injected in first)
% using a ``dissipation potential'' in variational derivations !?
\cite{Coleman-Noll1963}. 
The Newtonian % linear  ! viscous 
extra-stress e.g.
\beq
\label{eq:extranewt}
\tau^{ij} = 2\dot{\mu} D(u)^{ij} + \ell \: D(u)^{kk} \: \delta_{ij}
\eeq
% added to \eqref{eq:cauchystress_pressure} 
% yields the celebrated Navier-Stokes (NS) equations in \eqref{eq:firstorder_spatial_reduced}, along with
is admissible with $D= 2\dot{\mu} D(u)^{ij}D(u)^{ij} + \ell |\partial_iu^i|^2 % |D(u)^{kk}|^2
\ge0$, in
$
\partial_t \eta + ( u^j \partial_j ) \eta = {D}/\theta % \frac{D}\theta % (1.24) p6 or (1.34) p8 \cite{lions-1996}
$
when the entropy $\eta$ is chosen as additional state variable \cite{lions-1996}, or in
\begin{equation}
\label{eq:temperature}
\partial_t \theta + ( u^j \partial_j ) \theta = -{D}/\eta % -\frac{D}\eta % (1.24) p6 or (1.34) p8 \cite{lions-1996}
\end{equation}
when the temperature $\theta$ is the additional state variable. % mention incompressible version ?
%
%The latter equations clearly show that the viscous extra-stress in NS equations are of entropic nature,
%as oppposed to the energetic nature of the pressure $p$.
%
The NS equations have an interpretation at the same molecular level % Boltzmann Maxwell
as the polytropic law, % though both are not rigorously defined at the same time
with $\dot{\mu}>0$ \& $\ell>0$ % Lam\'e coeff
the shear \& bulk % dynamic
viscosities typically measured for a fluid % liquid 
close to its rest-state % say in Couette flow 
at given pressure and temperature.
But % this is not fully coherent with the definition of stresses from the first law and
although useful in many % some 
cases, % \emph{causal}
the flows defined by NS % Navier-Stokes equations 
or any % similar linearly-viscous
momentum balance with diffusion % such as developped similarly on pure objectivty arguments for nonlinear rheology
are \emph{not local} % (in time and space)
unlike the % elastic 
motions defined by \eqref{eq:firstorder_material} for polyconvex hyperelastic bodies. 
%
% That breaks the finite-propagation speed of information initially inherited from the causality principle requirement.
% It also makes it difficult the extension of the model to more complex physics,
% where the viscosity is naturally coupled e.g. to particle migration: 
% we would like viscous behaviour to come from the constitutive law, where physics is usually added.
%
% There also remains many nonlinear (non-Newtonian) rate-dependent stresses to model in \emph{viscoelastic} flows. % De witt
% The so-called \emph{materials with fading memory} have been conceived to that aim \emph{Coleman1964}.
% But they seem completely useless -- numericall at least.

To describe % causal
local viscous motions, we next follow Maxwell % recall Sec. \ref{sec:intro}
and consider \emph{viscoelastic} fluids relaxing to an elastic equilibrium,
where viscosity arises asymptotically only -- just like the steady flows where it is actually measured !
For fast relaxing fluid flows, one may prefer the standard extra-stress approach, leading to the ``simple'' NS % Navier-Stokes 
equations,
at the price of losing %full causality  ie
locality. But that preference % approximation which amounts to time renormalization
depends % has to be evaluated depending 
on what ``fast'' means in comparison with the physically-relevant % wave 
speeds. For applications when time-dependence is particularly important,
one should % therefore 
prefer the viscoelastic models below to the viscous fluid model above.

\subsection{Standard viscoelastic flow models with Maxwell fluids} % constitutive assumptions for 
\label{sec:viscoelastic}

Standard viscoelastic constitutive assumptions for the extra-stress are formulated as extensions of viscous fluids, 
first constrained by ``objectivity'' like in \cite{Coleman-Noll1963}. % Coleman_et_al-1961 without consideration of thermodynamics !!
Viscoelastic fluids of Maxwell type \cite{Maxwell01011867} thus use differential equations like
\begin{equation}
\label{eq:UCM} 
\lambda \stackrel{\Diamond}{\btau} + \btau = 2 \dot{\mu}  \bD
\end{equation}
for the extra-stress in \eqref{eq:extra}, with $\lambda>0$  a relaxation time scale % $\bD$ is the strain-rate (i.e. symmetrized velocity gradient) 
and $\stackrel{\Diamond}{\btau}$ an \emph{objective} time-rate % derivative 
\cite{oldroyd-1950,bird-curtiss-armstrong-hassager-1987a,renardy-2000}.
% 
% Our viscoelasticity is different from the % usual % for mathematicians 
% diffusive viscoelasticity concept developped % beyond elastodynamics for solids 
% e.g. in~\cite{Barker2011,christoforou-galanopoulou-tzavaras-2018}. % huo-yong-CMS-2017-0015-0004-a010
% 
% Generalization of hypo-elastic framework along Bolztmann integral formulation of viscoelasticity is fading-memory fluids \cite{Coleman_et_al-1961,Coleman-Gurtin1965}.
%
The extra-stress governed by \eqref{eq:UCM} % induced by sollicitations on a time range small compared with $\lambda>0$
is well understood in small deformations % (typically oscillatory shear)
when % the linearization 
$\stackrel{\Diamond}{\btau} \approx \partial_t \btau$:
high-frequency motions are elastic % small deformations around a preferential configuration like a solid 
with % shear elastic
modulus $\mu :=\dot{\mu} /\lambda$ (in stress units), and low-frequency motions are viscous % with a shear stress proportional to \emph{strain time-rate i.e. velocity gradient} (like any viscous fluid)
with % dynamic 
viscosity $\dot{\mu} $.
% See for instance the linearized model in small oscillatory shear motions in bird-curtiss-armstrong-hassager-1987a ex. 5.3.1 p. 264
More generally, it evolves nonlinearly, % it is nonlinearly coupled with u 
using as time-rate in \eqref{eq:UCM} % Gordon-Schowalter derivatives
% hence a Johnson-Segalman model included in Oldroyd 8-constant model -- without diffusion (retardation) here
\begin{equation}
\label{eq:gordon} 
% \triangledow \triangle % diffusion (retardation, or background viscosity) arise as a generalization of Jeffreys extension of Maxwell model
\stackrel{\Diamond}{\tau^{ij}} 
= \partial_t \tau^{ij} + u^k\partial_k \tau^{ij} - \partial_k u^i \tau^{kj} - \tau^{ik} \partial_k u^j + \zeta (D(u)^{ik} \tau^{kj} + \tau^{ik} D(u)^{kj})
\end{equation}
for some $\zeta\in[0,2]$.
The nonlinear terms in \eqref{eq:gordon} are believed responsible for non-Newtonian motions observed experimentally,
like rod-climbing (equiv. Weissenberg effect) with polymeric % rubber-like 
liquids \cite{dewitt-1955}.
% But motions have hardly been computed with the system \eqref{eq:firstorder_spatial}--\eqref{eq:UCM}, as detailed below.
%
% at that stage, since second law is a constitutive assumption, it remains to check
Moreover, the ``dissipativity'' of the extra-stress $\btau$ % governed by \eqref{eq:UCM} 
% this seems to have been ``neglected'' ? (or not understood) for a while  Coleman1964, MR0214318 -- MR1096233  !!
% i.e. the thermodynamical compatibility of Maxwell fluids 
is % now 
standardly 
analyzed on introducing a \emph{conformation tensor} $\bc$ 
\cite{GRMELA1987271} % ref[20,21,22] in the latter and Leonov1976,giesekus-1982,GRMELA1987276,
% with contravariant coordinates $c^{ij}$ 
interpreted % in a kinetic theory
as % the second moment 
$\EE(\bR\otimes\bR)$ where $\bR(t,\bx)$ is the end-to-end % material
vector of % elastic 
``dumbbells'' modelling statistically macromolecules suspended in the % continuous 
fluid % liquid body % a mechanistic model
\cite{bird-curtiss-armstrong-hassager-1987b,renardy-2000}.

Assume dumbbells are governed by the (overdamped) Langevin equation 
\begin{equation}
\label{eq:langevin0} 
d R^i = \left( - (u^j\partial_j) R^i + (\partial_j u^i) R^j - \frac{2K}\xi F^i(\bR) \right) dt + \sqrt{\frac{4 k_B \theta}\xi} dW^i(t)
\end{equation} % i.e. a stochastic differential equation
given friction $\xi$ % for the beads % fixed by Stokes approximation at low relative motion (overdamped ?)
and spring factor $K(\theta)$ % given 
at % temperature 
$\theta$. % isothermal motions
Using \eqref{eq:gordon} % as time rate % upper-convected Lie Derivative
with % a ``slip-parameter'' 
$\zeta=0$, it leads % one 
to % the evolution equation \eqref{eq:conformation} 
\begin{equation}
\label{eq:conformation} 
\stackrel{\Diamond}{c^{ij}} = - \frac{4K \Hcal'}\xi  c^{ij} + \frac{4 k_B \theta}\xi \delta_{ij}
% \frac1\lambda ( \delta_{ij}-c^{ij} )
\end{equation} % in space equipped with Euclidean geometry !!
for $\bc$. %at least approximately.
% after Ito calculus and Peterlin approximation in the nonlinear case
%
Precisely, when $F^i$ in \eqref{eq:langevin0} is non-linear, a good % reasonable 
approximation \eqref{eq:conformation} should postulate a non-linear potential $\Hcal\left(tr(c)\right)$ % peterlin approximation
% which preserves the physically-meaningful domain of
i.e. $\Hcal'\left(tr(c)\right)$ non-constant so that $\bc$ % (symmetric positive definite) 
remains strictly positive, see e.g. \cite{hulsen-1990}.
The particular case when $F^i(\bR)=\Hcal' R^i$ with $\Hcal'$ constant does not need approximation:
the random vector $\bR$ is Gaussian and \eqref{eq:conformation} is exact.
It is the consitutive assumption for \emph{Upper-Convected Maxwell} (UCM) fluids.
The motions defined with smooth solutions to \eqref{eq:conformation} indeed satisfy %  the evolution equation of $F$ straightforwardly computed
\beq
\label{eq:freeenergyevolution}
(\partial_t + u^j\partial_j) \Fcal(c) = 2(K \Hcal' c^{ij} - k_B \theta \delta_{ij}) \partial_i u^j - \frac{4}\xi \Dcal
\eeq
on denoting $[c^{-1}]_{kl}$ the matrix inverse of $c^{ij}$ symmetric positive definite, with
\beq
\label{eq:dumbbellfreeenergy}
\Fcal = K \Hcal\left(tr(c)\right) - k_B \theta \log|c| \,,
\eeq
\beq
\label{eq:dumbbelldissipation}
\Dcal = (K \Hcal' c^{ij} - k_B \theta \delta_{ij}) [c^{-1}]_{jk} (K \Hcal' c^{ik} - k_B \theta \delta^{ik}) \ge 0 \,.
\eeq
So $D \equiv \frac{4}\xi \Dcal$ can be % interpreted as
a dissipation in % the second law
\eqref{eq:secondlaw_spatial_bis} for isothermal flows,
% at fixed temperature $\theta$ in isothermal motions with Helmholtz free energy
and $\mathcal F$ a dumbbell contribution to the Helmholtz free energy % already in 1.447210 cited by bird-b
$\psi = e_0(\rho,\theta)+\mathcal F(\bc,\theta)$ % in the compressible UCM model for isothermal motions 
where $e_0(\rho,\theta)$ is a % supposedly additive
solvent contribution like % compressible 
the % barotropic -- temp independant
polytropic law \eqref{eq:polytropic}, % for the (compressible) suspension 
while the % definition for a viscoelastic admissible
extra-stress
\begin{equation}
\label{eq:extrastress}
\tau^{ij} = 2\rho (K \Hcal' c^{ij} - k_B \theta \delta_{ij})
\end{equation}
is admissible in \eqref{eq:extra} and has a molecular interpretation through % a Gaussian vector 
$\bR$, $k_B$ being the same 
Boltzmann constant as in \eqref{eq:langevin0} % we are not at canonical equil. where free-energy F = -\beta log( Z = \int p(R)dR exp ( -\beta H(R) ) )
\cite{doi-edwards-1998,bird-curtiss-armstrong-hassager-1987b}. % initially kramers 1944, see eg Bird1971  ottinger-1996-Book
% out of equilibrium statistical physics (langevin) !! free-energy not obviously defined from a "canonical" ensemble
% but should be compatible with dF = -PdV+SdT : 1° stress function of c (rather than P) -> "Kramers" expression and
% 2° the entropy $-\partial_\theta F$ coincides with the % (average of) 
%  Gibbs statistical entropy, i.e. or rather a relative entropy
%
For incompressible isothermal flows ($\partial_i u^i \equiv 0$)
with $\rho$ constant, the evolution of $\btau$ % defined by \eqref{eq:extrastress}
satisfies exactly Maxwell upper-convected % viscoelastic model
equation \eqref{eq:UCM}
with $\lambda = \frac\xi{4 K \Hcal'}$ and $\dot{\mu} =2\lambda\rho k_B\theta$.
% the definition \eqref{eq:extrastress} of extra-stress requires one to add 
For general % compressible viscoelastic 
flows, % or non isothermal like (66) - (67) in Dressler
$\btau$ satisfies \eqref{eq:UCM} with additional terms % $ \lambda \btau (\partial_t + u^j\partial_j) \log\rho $
in RHS, see \eqref{eq:viscoelasticcomponent} in section~\ref{sec:physics}.

Multi-dimensional models that are extensions of Maxwell seminal ideas
often use the UCM model \eqref{eq:firstorder_spatial_reduced}--\eqref{eq:extra}--\eqref{eq:extrastress}--\eqref{eq:conformation} 
as a starting point, up to the recent efforts \cite{DresslerEdwardsOttinger1999,mackay-philips-2019} % Hutter2009,GUAILY2015305,
toward % (compressible) viscoelastic models for 
non-isothermal % experimentally observed
flows, % (with \eqref{eq:temperature} or a diffusion equation for temperature) % (in adiabatic processes)
% bae-trivisa-2012,MR3668885
% See \cite{1.2798235} for experimental validation
or some variations of UCM \cite{larson1988constitutive,renardy-2000},
% there is still a debate about which are useful (mathematically at least), starting with stability questions (leonov vs renardy, eg Leonov_1995)
using for instance another % (nonlinear) 
force $F^i$ in \eqref{eq:langevin0} than linear
(which leads to a different viscoelastic flow model with a different free energy), % FENE-P
or another Langevin equation % langevin in material coordinates !
(which could lead to an evolution of conformation \eqref{eq:conformation} % like the \emph{Lower-Convected Maxwell} (LCM) model 
using $\zeta=2$ rather than $\zeta=0$). % in \eqref{eq:UCM}
% JS ?? thermodynamically admissible ??
%%%%%%%%%%% % eg kwon and leonov ``poor'' studies of stabilty issues
% Only $\zeta=0$ % only ??
% yields a hyperbolic % strongly ?
% system % cite previous 
% ... PROOF ...
% a minimal requirement to define 1D motions by 1D solutions to Cauchy problems % smooth -- and weak ?
% %
% it is not enough to go beyond even for smooth solutions
% insofar as the quasilinear system lacks a structure with conservation laws \cite{benzonigavage-serre-2007}. 
% % so we do not know about strong Leray solutions
%%%%%%%%%%% % our model even resolve a false discrepancy between differential and integral models kwon-1996
%
% Oldroyd considered only eq linear in stress, Giesekus added nonlinear terms ! molecular interpretation too
%
% The UCM model above is one % very particular
% model at the intersection of the differential class and of
% the K-BKZ integral class \cite{bernstein-kearsley-zapas-1963,BKZ1964}, that is useful to compute motions \cite{}. % DETAILS ???????????????? K BKZ
% what does the kernel/integral approach subsume ?
% % 1.1724109 % green-tobolsky foundation of K-BKZ kayes bernstein kearskey zapas etc
% % Doi_Edwards_3_JCS-FT78 % reptation => constitutive
% % LARSON1983279, Narain1986, 1.549986, mitsoulis-2013 % reviews of constitutive equations, see inside
%
% % is fading memory an interesting generalization ? 
%
General compressible viscoelastic motions % (isothermal or not)
have however hardly been analyzed % mathematically
or simulated so far, with the full compressible UCM system or any other similar % differential
viscoelastic model.
% numerically computed => both are ``general'' in promise
% \footnote{We avoid the terminology thermoviscoelasticity because it has another meaning in e.g. \cite{Christoforou2018}.}
We are aware of a 2D hyperbolic quasilinear % system
UCM model, % but mass and energy are not conserved !!
but it is not a system of conservation laws,
and its numerical simulation relies on some \emph{empirical} % possibly too much !!
diffusion % artificial ad-hoc stabilization
\cite{PHELAN1989197,EDWARDS1990411,OLSSON1994309}. % olsson-ystrom-1993,ystrom-olsson-kreiss-1993
% Besides the slightly compressible non-conservative UCM proved only \emph{weakly hyperbolic} i.e. the eigenstructure of the system can degenerate as time evolves and solutions can become unstable (oscillatory)
% although some other computations have sometimes been made differently \cite{Guaily2010158} ?? + tempretaure ??
%
% This could also explain why model sare still being discussed
%
One difficulty with the (multi-dimensional, compressible) viscoelastic models proposed so far 
might be the lack of a % standard 
mathematical structure to properly define motions through Cauchy problems, % for instance the lack of
such as a symmetric hyperbolic system of conservation laws % which is interesting 
\cite{Kato1975,majda1984book}. 
% analytically : the Cauchy problem admits a unique strong classical solution for small times and smooth initial conditions
% numerically ? !! weak solutions

Viscoelastic motions have mostly been studied under the % additional 
incompressibility assumption % In common flow situations, liquids are close to incompressible.
and with additional diffusion so far, whether for UCM or other fluids \cite{owens-philips-2002}.
% This leads to simplifications.
Indeed, incompressible viscoelastic motions with % finite speed ?? no !!
$\partial_i u^i=0$ and $\rho$ constant 
have been well defined as solutions % on the full ambiant space $\R^d$ 
to % well-posed then !!
Cauchy problems for the UCM model \eqref{eq:firstorder_spatial}--\eqref{eq:extra}--\eqref{eq:extrastress}--\eqref{eq:conformation},
as well as other quasilinear systems % of the K-BKZ class 
provided they are % very !!
regular enough \cite{renardy-1985}. % following Kato's theory for hyperbolic quasilinear systems despite discrepancy
% \footnote{In fact, for the K-BKZ viscoelastic fluids which include UCM,LCM... but not JS !} dupret1986signe DUPRET1986143
%
% MR1174943 existence uniqueness of solution to initial-boundary value problem (weak) modelling the velocity field of linear, isothermal, compressible viscoelastic fluid of fading memory 
%
Still, numerical simulations of the incompressible UCM system have shown unstable % for realistic large values of $\lambda>0$ 
in applications % to polymer flows !! % first papers are by MARCHAL and JOSEPH ??
\cite{joseph-saut-1986,joseph-renardy-saut-1985} % kwon-leonov-1995 % series of paper about stability criterium % like Euler !!
and most viscoelastic flows have in fact been % simulated
computed for incompressible fluids of Jeffrey type % \cite{joseph-renardy-saut-1985} bird-curtiss-armstrong-hassager-1987a
with an additional retardation time % parameter in comparison with Maxwell type
(i.e. a rate-dependent term in \eqref{eq:UCM} which induces velocity diffusion with a ``background viscosity'')
\cite{owens-philips-2002}. % MARCHAL198777,keunings-1989,RAJAGOPALAN1990159, mostly in steady case 
% renardy-1985a Existence of slow steady flows of viscoelastic fluids with differential constitutive equations -- girault-scott-2017
%
In any case, assuming incompressibility prevents % full causality
locality %  motions are not causal because of the global constraint whale_fowkes_hocking_hill_1995
and limits applications to non-isothermal flows. %  with mass transfers anyway \cite{mackay-philips-2019} % also ?? Bollada2012,flaman1988injection, 
% incompressible with temperature like is not very interesting, though it allows to compute thermal sensitivity  PARK1999197 CHANG19941 DACOSTAMATTOS20122134
%
% Compressibility (with a correct pressure law) is % probably 
% important for % with a view to capturing 
% correct mass and energy transfers. % \cite{ma00115a017} % see also flow past cylinder in DELVAUX1990297
%
% (the need of) ??
Diffusion % a somewhat artificial dissipation
% stabilizes but
% cannot propagate discontinuities like % weak solutions of
% a hyperbolic system % to model shear-bands (``slip surfaces'') for instance !! see plohr nohel \cite{joseph-renardy-saut-1985}
% and
does not restore the locality % full causality 
of motions, on the contrary. % of course % especially when it is somewhat numerical / artificial

So the question thus remains % so far 
how to usefully extend Maxwell's seminal viscoelastic model to general (compressible, multi-dimensional) % causal 
motions.

% Hyperbolic 
\section{Symmetrizing % Reformulating the 
Upper-Convected Maxwell} % model
\label{sec:model}

% Having set the % compressible 
% UCM model as a starter for viscoelastic flows, % a bridge between fluid and solid
% see section above
We now propose to rewrite the UCM model as a useful symmetric-hyperbolic system of conservation laws
% Its standard formulation % as a non-conservative system, mostly for incompressible flows 
% has not been used without difficulty so far.
which extends the % symmetric-hyperbolic system of conservation laws for
elastodynamics of polyconvex hyperelastic materials using an additional material metric variable.
% that evolves locally in time and relaxes (to the inverse right Cauchy-Green deformation tensor) % inverse -> inverse  in fact
% as additional % dependent
% variable. % metrics on the body different from the native Riemmanian metrics !!!!!
%
The new system of conservation laws % which remainds of the so-called K-BKZ formulation of the UCM model
is introduced in section~\ref{sec:conservation}.
It is shown symmetric-hyperbolic in section~\ref{sec:strictlyconvex}.
% using a strictly convex mathematical entropy for an adequate conservative formulation.
%
Finally, the physics of UCM is discussed using that new system in section~\ref{sec:physics}.
It allows to interpret UCM as one particular extension of elastodynamics using an additional material metric variable,
with much more potentialities (beyond fluid viscoelasticity) to be discussed in future works.
%
% In \cite{Peshkov2016,PESHKOV2019481}, a new variable is also introduced which leads to a hyperbolic system with a non-essential non-conservative product ;
% the main difference is how to interpret the new variable.
%
% In \cite{bouchut-boyaval-2013,bouchut-boyaval-2015,boyaval-hal-01661269} we could restore mass and energy conservation
% in comparison with \cite{PHELAN1989197,EDWARDS1990411} in 2D, but not more without adding the new metrics variable.
%
% the question of using a quasilinear system of conservation laws (with a convex extension) for fluid flows is not new \cite{MR1999156} godunov
% In the past, one has tried \cite{zbMATH06380509,zhu2015JNET,huo-yong-CMS-2017-0015-0004-a010} % HUO20161264

The present new % formulation of the compressible UCM model as a viscoelastic extension of the elastodynamics 
system already has interesting applications, % in computational rheology % for non-isothermal flows -- discussed elseweher and 
% in geophysics. They are discussed in
see Section~\ref{sec:2D}. % SV !! TURBULENCE => \cite{teshukov-2007,Teshukov2007} 

\subsection{Conservation Laws for UCM}
\label{sec:conservation}

A reformulation of the standard UCM model was already proposed by the K-BKZ theory \cite{kaye-1962,kaye-1963,bernstein-kearsley-zapas-1963,BKZ1964},
to establish a clear link between the viscoelastic UCM fluids and (elastic) solids, % rubber elasticity, in particular 
and to next improve the UCM model. % for more complex fluids -- and CREEP !!! ???
% % 1.1724109 % green-tobolsky foundation of K-BKZ kayes bernstein kearskey zapas etc
% % Doi_Edwards_3_JCS-FT78 % reptation => constitutive
% % LARSON1983279, Narain1986, 1.549986, mitsoulis-2013 % reviews of constitutive equations, see inside
% we note that 
% the left Cauchy-Green % strain
% deformation tensor $B^{ij}$ satisfies $\stackrel{\Diamond}{\bB} = 0$ with $\zeta=0$, 
% like the conformation tensor $c^{ij}$ in the UCM model % \eqref{eq:UCM} 
But it leads to an integro-differential systems that is not much more easily used for general flows than standard UCM.
Still, to get a useful formulation, we can follow K-BKZ theory and first interpret 
the UCM model % (eq.~\eqref{eq:conformation} for $c^{ij}$) in particular
with the help of the \emph{full} Eulerian description \eqref{eq:firstorder_spatial} of (smooth) motions 
% smooth UCM motions, with equivalent spatial and material descriptions,
for continuous bodies
% and not simply its \emph{reduced} form \eqref{eq:firstorder_spatial_reduced} for (simple) fluids
as follows.
\begin{proposition}
\label{prop:ode}
Consider smooth motions of UCM fluids such that
$c^{ij}$ satisfies \eqref{eq:conformation} with $\zeta=0$ in % the time-rate 
\eqref{eq:gordon},
% \begin{equation} \label{eq:conformation} 
% \stackrel{\Diamond}{c^{ij}} = - \frac{4K H }\xi  c^{ij} + \frac{4 k_B \theta}\xi \delta_{ij} % \frac1\lambda ( \delta_{ij}-c^{ij} )
% \end{equation} % in space equipped with Euclidean geometry !!
and $[F^{-1}]^{\alpha}_i$ denotes the inverse % matrix with coefficients
of the deformation-gradient % matrix with coefficients 
$F^i_\alpha$. % with an inverse !
It holds for $A^{\alpha\beta} = [F^{-1}]^{\alpha}_i c^{ij} [F^{-1}]^{\beta}_j$ in the material description:
\beq
\label{eq:ode}
\partial_t \left( A^{\alpha\beta}\circ\bphi_t \right)
= \frac{4 k_B \theta}\xi \left( [F^{-1}\circ\bphi_t]^{\alpha}_i [F^{-1}\circ\bphi_t]^{\beta}_i \right) 
- \frac{4K \Hcal'}\xi A^{\alpha\beta}\circ\bphi_t \,. % \circ\bphi_t
\eeq 
\end{proposition}
\begin{proof}
Recalling \eqref{eq:firstorder_material}, 
% \beq % the \emph{elastodynamics} where many math questions remain though \cite{MR1919825} % ball 2002
% \label
% \begin{aligned}
% & \partial_t |F^i_\alpha| - \partial_{\alpha}\left( C^\alpha_j u^j \right) = 0 
% \\
% & \partial_t F^i_\alpha - \partial_\alpha u^i = 0
% \\
% & \hat\rho \partial_t u^i - \partial_\alpha S^{i\alpha} = \hat\rho f^i
% \end{aligned}
% \eeq % is sometimes termed the \emph{elastodynamics} system and
the deformation gradient $F^i_\alpha$ % of the current spatial description with respect to an arbitrary fixed material description
satisfies % the following evolution 
\beq 
\label{eq:deformationgradient}
(\partial_t + u^i\partial_i) F^i_\alpha - (\partial_j u^i) F^j_\alpha = 0
\eeq
in spatial description. Then, the inverse satisfies
\beq 
\label{eq:deformationgradientinverse}
(\partial_t + u^i\partial_i) [F^{-1}]^{\alpha}_i - [F^{-1}]^{\alpha}_i (\partial_j u^i) = 0
\eeq
which can be combined with \eqref{eq:conformation} to yield
\beq
\label{eq:ode_spatial}
(\partial_t + u^i\partial_i) A^{\alpha\beta} = - \frac{4K \Hcal'}\xi A^{\alpha\beta} + \frac{4 k_B \theta}\xi \left( [F^{-1}]^{\alpha}_i [F^{-1}]^{\beta}_i \right) \,.
\eeq
It follows \eqref{eq:ode} in the material description.
\end{proof}
\begin{corollary}
\label{cor:ode}
Consider smooth motions of UCM fluids like in Prop.~\ref{prop:ode}
given positive constants  $K$, $\Hcal'$, % in addition to 
$\xi$, $\theta$.
Then, denoting $\lambda = \frac\xi{4 K H}$, it holds
% for the % smooth 
% solutions to \eqref{eq:conformation} % has the representation of Duhamel type
% in smooth motions with material and equivalent spatial descriptions
% satisfy 
for $t \ge t_0$:
\begin{multline}
\label{sol:ode}
c^{ij}(t)\circ\bphi_t = e^{\frac{t_0-t}\lambda} F^i_\alpha(t)\circ\bphi_t   % A^{\alpha\beta}(t_0) = 
[F^{-1}]^{\alpha}_k(t_0)\circ\bphi_{t_0}  [F^{-1}]^{\beta}_k(t_0)\circ\bphi_{t_0}  F^j_\beta(t)\circ\bphi_{t} \\
 + \frac{k_B\theta}{K \Hcal'} \int_{t_0}^t ds \frac1\lambda e^{\frac{s-t}\lambda} F^i_\alpha(t)\circ\bphi_{t} [F^{-1}]^{\alpha}_k(s)\circ\bphi_{s}  [F^{-1}]^{\beta}_k(s)\circ\bphi_{s}  F^j_\beta(t)\circ\bphi_{t} \,.
\end{multline}
\end{corollary}
\begin{proof}
One straightforwardly obtains \eqref{sol:ode} on injecting % the Duhamel representation of 
the exact solution to the linear first-order
differential equation \eqref{eq:ode} in $c^{ij} = F^i_\alpha A^{\alpha\beta} F^j_\beta$.
\end{proof}

Next, 
% to model viscoelastic motions 
% without more variable than the % (relative) 
% deformation gradient like in % already used for 
% solids elastodynamics, 
K-BKZ theory assumes
\beq
\label{eq:bkzassumption}
F^i_\alpha(t)\circ\bphi_{t} [F^{-1}]^{\alpha}_k(t_0)\circ\bphi_{t_0} [F^{-1}]^{\beta}_k(t_0)\circ\bphi_{t_0} F^j_\beta(t)\circ\bphi_{t} \to \delta_{ij} \text{ as } t_0 \to -\infty \,,
\eeq
and rewrites the free energy \eqref{eq:dumbbellfreeenergy} % $F$ in $\psi$ 
and the extra-stress \eqref{eq:extrastress} of UCM fluids 
with the (history of) \emph{relative} deformation gradients $F^i_\alpha(t)[F^{-1}]^{\alpha}_k (s)$, $t\ge s$ only,
i.e. without using explicitly material coordinates \cite{kaye-1962,kaye-1963,bernstein-kearsley-zapas-1963,BKZ1964}.
The resulting \emph{integro-differential} % evolutionary
system % which is not easy to analyze % no math structure
% and simulate % only very smooth motions seem well-posed
has allowed one % K-BKZ theory 
to compute viscoelastic UCM motions % for the UCM model, 
and also other viscoelastic motions % with other models % than UCM incl LCM `cite{LUO1996173}
after generalizing \eqref{sol:ode} to other ``kernels'' than $\frac1\lambda e^{\frac{s-t}\lambda}$,
% and other free energies (thus other extra-stress) !! thermo of K-BKZ\dots
when \emph{incompressible} (therefore not local) \cite{renardy-1985}.

Here, to define % causal and 
local UCM motions,
we propose a new purely differential approach to compute multi-dimensional (compressible) % UCM 
flows with a \emph{symmetric-hyperbolic system of conservation laws} inspired by polyconvex elastodynamics.
Unlike K-BKZ theory, we do not avoid % eliminate reference to
material coordinates.
We propose to use $A^{\alpha\beta}$ as a % new dependent
variable of the system % governed by equation \eqref{eq:ode} 
and to write $c^{ij}$ as a function of $F^i_\alpha$ and $A^{\alpha\beta}$:
\begin{proposition} 
\label{prop:newucm_spatial}
The smooth isothermal viscoelastic motions % described in spatial coordinates by 
solutions to the Eulerian model \eqref{eq:firstorder_spatial}--\eqref{eq:extra}--\eqref{eq:extrastress}--\eqref{eq:conformation}
for compressible UCM fluids
are equivalently solutions to the system of conservation laws with algebraic source terms \eqref{eq:newucm_spatial}:
\beq
\label{eq:newucm_spatial}
\begin{aligned}
& \partial_t ( \rho u^i ) + \partial_j\left( \rho u^j u^i \right) - \partial_j \: \left( -p \delta_{ij} + % \tau^{ij}  {eq:extrastress}
 2\rho (K \Hcal' F^i_\alpha A^{\alpha\beta} F^j_\beta % c^{ij} 
 - k_B \theta \delta_{ij}) \right) % \sigma^{ij}
= \rho f^i
\\
& \partial_t ( \rho F^i_\alpha ) + \partial_j\left( \rho u^j F^i_\alpha - \rho u^i F^j_\alpha \right) = 0
%& \partial_t ( |F^i_\alpha|^{-1} \hat\rho ) + \partial_i( u^i |F^i_\alpha|^{-1} \hat\rho ) = 0 
\\
%& |F^i_\alpha|^{-1} \hat\rho \left( \partial_t u^i + u^j \partial_j u^i \right) - \partial_j \left( \sigma^{ij} \right) = |F^i_\alpha|^{-1} \hat\rho f^i
& \partial_t \rho + \partial_i( u^i \rho ) = 0 
\\
& \partial_t ( \rho A^{\alpha\beta} ) + \partial_j\left( \rho u^j A^{\alpha\beta} \right) 
= \frac{4\rho}\xi \left( k_B \theta \left( [F^{-1}]^{\alpha}_i [F^{-1}]^{\beta}_i \right) - K \Hcal' A^{\alpha\beta} \right) %{eq:ode}
\end{aligned}
\eeq
with % insofar as 
$A^{\alpha\beta} = [F^{-1}]^{\alpha}_i c^{ij} [F^{-1}]^{\beta}_j \in  S^{++}(\R^{d\times d})$. Furthermore, they satisfy
\beq
\label{eq:energy_spatial}
% \rho (\partial_t + u^j\partial_j) \left( \frac{|u|^2}2 + \psi \right) 
\partial_t E + \partial_j\left( u^j E \right)
- \partial_j \left( u^i \sigma^{ij} \right) = f^iu^i - \frac{4 \rho}\xi \Dcal
\eeq
with % an energy 
$E=\rho\left(\frac{|u|^2}2 + \psi\right)$,
% and a stress % Cauchy
$\sigma^{ij}=-p\delta_{ij}+ 2 \rho F^i_\alpha F^j_\beta \partial_{F^k_\alpha F^k_\beta} \psi$, % Doyle Ericksen 
$p=\partial_{\rho^{-1}}e_0(\rho,\theta)$, % so $p=C_0 (\hat\rho|F^i_\alpha|)^\gamma$ when $e_0(|F^i_\alpha|)$ is the polytropic law \eqref{eq:polytropic} 
%\beq
\begin{align}
\psi(\rho,F^i_\alpha,A^{\alpha\beta}) 
& = e_0(\rho) + K \Hcal' F^i_\alpha F^i_\beta A^{\alpha\beta}   % \Hcal\left(F^i_\alpha A^{\alpha\beta} F^i_\beta\right) 
- k_B \theta \log|F^i_\alpha A^{\alpha\beta} F^i_\beta| % {eq:dumbbellfreeenergy}
\label{eq:Helmholtzfreeenergy}
\\
& = e_0(\rho) + K \Hcal' F^i_\alpha F^i_\beta A^{\alpha\beta}   % \Hcal\left(F^i_\alpha A^{\alpha\beta} F^i_\beta\right) 
+ 2 k_B \theta (\log\rho/\hat\rho-\log|A^{\alpha\beta}|) % ( \hat\rho/\rho) % \rho^{-1} adimensional % |F^i_\alpha F^i_\beta| 
\label{eq:Helmholtzfreeenergy2}
\end{align}
%\eeq
% the Helmholtz free energy % already in 1.447210 cited by bird-b
% say for suspensions where the solvent contribution $e_0(\rho,\theta)$ can be the polytropic law \eqref{eq:polytropic}
and $\Dcal\ge0$ the same \emph{dissipation} as given by \eqref{eq:dumbbelldissipation}. % with $c^{ij} = F^i_\alpha A^{\alpha\beta} F^j_\beta$. 
% without last term : $\Dcal = - k_B \theta d + K \Hcal' F^i_\alpha F^i_\beta A^{\alpha\beta}$ (without sign a priori). 
\end{proposition}
\begin{proof}
We have already shown that the \emph{smooth} isothermal viscoelastic motions described in spatial coordinates by 
the compressible UCM model \eqref{eq:firstorder_spatial}--\eqref{eq:extra}--\eqref{eq:extrastress}--\eqref{eq:conformation} satisfy \eqref{eq:ode}.
Now, smooth motions also satisfy \eqref{eq:deformationgradient} by definition, thus the last line of \eqref{eq:firstorder_spatial}
using the Piola identities \eqref{eq:2Dpiola_identity} for smooth motions like in elastodynamics \cite{wagner-1994}.
So finally, the full system \eqref{eq:newucm_spatial} is satisfied.

Reciprocally, the standard formulation of UCM is recovered from \eqref{eq:newucm_spatial}
using $c^{ij} = F^i_\alpha A^{\alpha\beta} F^j_\beta$ and Piola identities for smooth motions $\phi^i_t$ such that 
$u^i\circ\bphi_t=\partial_t \phi^i_t$, $F^i_\alpha\circ\bphi_t=\partial_\alpha \phi^i_t$, $|F^i_\alpha|=\rho^{-1} \hat\rho>0$ with a constant $\hat\rho>0$.

Last, one can % easily
check \eqref{eq:energy_spatial} with \eqref{eq:Helmholtzfreeenergy} or \eqref{eq:Helmholtzfreeenergy2} directly for smooth motions,
on recalling $\rho^{-1} = |F^i_\alpha|\hat\rho^{-1}$.
The total energy balance \eqref{eq:energy_spatial} is also exactly that satisfied by UCM % \eqref{eq:firstorder_spatial}-- etc
using \eqref{eq:secondlaw_spatial_bis}
with \eqref{eq:dumbbellfreeenergy}, % F = K \Hcal\left(tr(c)\right) - k_B \theta \log|c| \,,
$c^{ij} = F^i_\alpha A^{\alpha\beta} F^j_\beta$
and % the dissipation 
\eqref{eq:dumbbelldissipation}.
\end{proof}

The UCM reformulation \eqref{eq:newucm_spatial} is an interesting system of \emph{conservation laws}.
When $\xi\to\infty$ and $A^{\alpha\beta}$ is constant in time,
the system \eqref{eq:newucm_spatial} coincides with a % well-known 
spatial description for compressible motions of homogeneous % uniform ????
neo-Hookean materials, % the Euclidean metrics
see \cite{dan44217} or \cite[(3.7)]{wagner-2009} when $A^{\alpha\beta}\equiv\delta^{\alpha\beta}$.
% Moreover, when $\xi\to0$ but $A^{\alpha\beta}\neq\delta^{\alpha\beta}$,
Inspired by the latter, we show that a further reformulation of \eqref{eq:newucm_spatial} allows one to define % multi-dimensional
flows of compressible UCM fluids
as solutions to a \emph{symmetric-hyperbolic} system of conservation laws.
% \footnote{
%   Note % however
%   that in \eqref{eq:newucm_spatial}, we do not need the first equation of % the latter 
%   system (3.7) in \cite{wagner-2009}, 
%   i.e. we do not need the spatial equivalent of the evolution \eqref{eq:cofGevolution_material} for $C^\alpha_i$,
%   % the cofactor matrix % transpose adjugate of $F^i_\alpha$  % = $|F^i_\alpha|\bF^{-T}$ in matrix notation
%   % $C^\alpha_i=|F^i_\alpha|\sigma_{ij}\sigma_{\alpha\beta}F^{j}_{\beta}$ in dimension 2
%   % $C^\alpha_i = \sigma_{ijk}\sigma_{\alpha\beta\gamma}F^{j}_{\beta}F^{k}_{\gamma}$ in dimension 3
%   % $$
%   % \partial_t C^\alpha_i + \sigma_{ijk}\sigma_{\alpha\beta\gamma}\partial_\beta \left( F^j_\gamma u^k \right) = 0 \,. %(1.9) dans wagner-2009
%   % $$
%   since the Helmholtz free energy \eqref{eq:Helmholtzfreeenergy},
%   which is the convex extension for \eqref{eq:newucm_spatial} solution to (3.20) in \cite{wagner-2009} when $\xi\to0$ and $\delta^{\alpha\beta}$ (see below),
%   is already strictly convex in $(u^i,|F^i_\alpha|,F^i_\alpha)$
%   for well-chosen $e_0$.
% } % the latter 

\subsection{A strictly convex extension % mathematical entropy 
for UCM}
\label{sec:strictlyconvex}

\begin{proposition}
\label{prop:symmetrichyperbolic}
The % smooth 
isothermal viscoelastic motions of compressible UCM fluids defined by smooth solutions to the % homogeneous !! quasilinear 
system \eqref{eq:newucm_spatial} with $\bA\in S^{++}(\R^{d\times d})$ 
are also equivalently defined by smooth solutions to % with the conservation laws 
\beq
\begin{aligned}
\label{eq:newucm_spatial_y}
& \partial_t ( \rho u^i ) + \partial_j\left( \rho u^j u^i \right) - \partial_j \: \left( -p \delta_{ij} + % \tau^{ij}  {eq:extrastress}
 2\rho (K \Hcal' F^i_\alpha A^{\alpha\beta} F^j_\beta % c^{ij} 
 - k_B \theta \delta_{ij}) \right) % \sigma^{ij}
= \rho f^i
\\
& \partial_t ( \rho F^i_\alpha ) + \partial_j\left( \rho u^j F^i_\alpha - \rho u^i F^j_\alpha \right) = 0
\\
& \partial_t \rho + \partial_i( u^i \rho ) = 0 
\\
& \partial_t ( \rho Y^{\alpha\beta} ) + \partial_j\left( \rho u^j Y^{\alpha\beta} \right) 
= - \frac{4\rho}\xi 
Y^{\alpha\gamma} \left( k_B \theta Z^{\gamma\delta} - 2 K \Hcal' \delta^{\gamma\delta} \right) Y^{\delta\beta} %{eq:ode}
\end{aligned}
\eeq
where $\bA=\bY^{-\frac12}$ is defined componentwise by identification with
the square-root matrix-inverse of $\bY=Y^{\alpha\beta}\be_\alpha\otimes\be_\beta\in S^{++}(\R^{d\times d})$
and $\bZ=\bF^{-T}\bF^{-1}\bA^{-1}+\bA^{-1}\bF^{-1}\bF^{-T}$. % in matrix notations a tensor such that
% $Z^{\gamma\delta} =  [F^{-1}]^{\gamma}_i [F^{-1}]^{\beta}_i [A^{-1}]^{\beta\delta} + [A^{-1}]^{\gamma\alpha} [F^{-1}]^{\alpha}_i [F^{-1}]^{\delta}_i$
Furthermore, if $p=-\partial_{\rho^{-1}}e_0$ % (\rho,\theta)
is given by $e_0$ strictly convex in $\rho^{-1}$, then the following additional conservation law is also satisfied
\beq
\label{eq:energy_spatial2}
\partial_t \tilde E + \partial_j\left( u^j \tilde E \right)
- \partial_j \left( u^i \sigma^{ij} \right) = f^iu^i - \frac{4 \rho}\xi \tilde \Dcal
\eeq
with $\tilde E=\rho\left(\frac{|u|^2}2 + e_0(\rho) + K \Hcal' F^i_\alpha F^i_\beta A^{\alpha\beta} + Y^{\alpha\beta}Y^{\alpha\beta} \right)$,
$\sigma^{ij}=-p\delta_{ij}+ 2 \rho F^i_\alpha F^j_\beta A^{\alpha\beta}$ % Doyle Ericksen 
and an algebraic source term $\tilde\Dcal$ without sign a priori.
So the \emph{strictly convex} function $\tilde E(  \rho ,\rho u^i,\rho F^i_\alpha, \rho Y^{\alpha\beta} )$ defines a mathematical entropy for % the (homogeneous) system
\eqref{eq:newucm_spatial_y}, 
\eqref{eq:energy_spatial2} defines a \emph{strictly convex extension} for % the (homogeneous) system
\eqref{eq:newucm_spatial_y},
and \eqref{eq:newucm_spatial_y} is % thus therefore 
a \emph{symmetric-hyperbolic} system of conservation laws on the % domain i.e. a convex
open set $\Acal^{+} := \{\rho>0,\ \bY=\bY^{T}>0 \} % \equiv \{\rho>0,\ \bY=\bY^{T}>0 \}
$. 
\end{proposition}
\begin{proof} First, recalling % $\partial_t \bA^{-1} = -  \bA^{-1} (\partial_t \bA) \bA^{-1}$ and
$\partial_t \bA^{-2} = -  \bA^{-2} (\partial_t \bA) \bA^{-1} - \bA^{-1} (\partial_t \bA) \bA^{-2}$
for smooth matrix-valued functions $\bA(t)$ % in matrix notation : multiplication is the matrix multiplication
one straightforwardly establishes the equivalence between formulations \eqref{eq:newucm_spatial_y} and \eqref{eq:newucm_spatial}
when $\bY,\bA\in S^{++}(\R^{d\times d})$. % on the domain $\Acal$, both with smooth source terms.
Note that $A^{\alpha\beta}A^{\beta\gamma}Y^{\gamma\delta}=\delta^{\alpha\delta}$ defines a bi-univoque relationship
on the % domain i.e. a convex
open set $\Acal^{+}$.
Next, one shows directly % that $\tilde E$ satisfies 
\eqref{eq:energy_spatial2}: % an additional conservation law
the computation is similar to that for $E$ in Prop.~\ref{prop:newucm_spatial}.

Then, Godunov-Mock theorem \cite[Chapter 3]{godlewski-raviart-1996} % (for the system without source, when $\xi\to0$)
implies that $\tilde E$ is a mathematical entropy % in the sense of e.g. 
and \eqref{eq:energy_spatial2} a strictly convex extension
for the \emph{symmetric-hyperbolic} system \eqref{eq:newucm_spatial_y} % therefore
provided $\tilde E(\rho,\rho u^i,\rho F^i_\alpha,\rho Y^{\alpha\beta})$ is strictly convex % function of the conservative variable
on the % open
convex set $\Acal^{+}\subset\RR^{1+ d + d(d+1)/2 + d^2}$. % := \{\rho>0,\ \bA=\bA^{T}>0 \}$. 
% to apply Godunov-Mock theorem, 
% or \cite{friedrichs-lax-1971} which introduced the terminology \emph{convex extension}.
%
We recall that the (strict) convexity of $\tilde E$ % lsc
function of $(\rho,\rho u^i,\rho F^i_\alpha,\rho Y^{\alpha\beta})$ on $\Acal^{+}$
is equivalent to the (strict) convexity of $E/\rho$ % lsc
function of $(\rho^{-1},u^i,F^i_\alpha,Y^{\alpha\beta})$ on $\Acal^{+}$, see \cite[Th. 3.1]{wagner-2009}
% in fact a rephrasing of (2.2.) and (2.3) p.126 of \cite{wagner-1987}
% though it uses the characterizatino of _lower semi-continuous_ convex functions (on Banach spaces) -- and we should prove the latter !! --
% see e.g. ekeland-temam-1976-convex Prop 3.1
or \cite[Lemma 1.4]{bouchut-2004}. 
As a matter of fact, $E/\rho$ is a mathematical entropy % an energy functional 
for an equivalent system of conservation laws % rewritten 
in material coordinates which we detail later, see \eqref{eq:energy_material}.

Now, $e_0$ % -2 k_B\theta \log\rho^{-1} possible ``for free''
and $\frac{|u|^2}2$
are strictly convex % functions of
in $\rho^{-1}>0$ and $u^i$, respectively.
Then, $E/\rho$ is a strictly convex function of $(\rho^{-1},u^i,Y^{\alpha\beta},F^i_\alpha)$ on $\Acal^{+}$
if % (and only if) 
$ F^i_\alpha F^i_\beta A^{\alpha\beta} + Y^{\alpha\beta}Y^{\alpha\beta}$ is a strictly convex function of $(Y^{\alpha\beta},F^i_\alpha)$ on $\Acal^{+}$.
% see e.g. 3.2.1 in \cite{boyd_vandenberghe_2004}
We conclude in two steps.
On the one hand, 
% denoting $\bY^{-\frac12}$ the inverse square-root of $\bY$, $\bF^T$ the transpose of $\bF$, 
$
(\bF,\bY) \in \RR^{d\times d} \times S^{++}(\RR^{d\times d}) \to tr(\bF\bY^{-\frac12}\bF^T)
$
is a (jointly) convex function of its $d^2+d(d+1)/2$ arguments 
by % Lieb's 
Theorem 2 in \cite[p.276]{lieb-1973} with $r=\frac12$ and $p=0$.
On the other hand, \emph{strict} convexity holds since $Y^{\alpha\beta}Y^{\alpha\beta}$ is strictly convex in $Y^{\alpha\beta}$,
and $ F^i_\alpha F^i_\beta A^{\alpha\beta}$ is strictly convex in $F^i_\alpha$. 
% to conclude when $Y^{\alpha\beta}_1=Y^{\alpha\beta}_2$ and the previous terms is useless for the strict inequality
% see 10.1.1.522.9564 for operator monotone functions with application to matrix inverse (convex)
% I believe the operator equality can read tr(X..X) > ... and the case of spd matrices can be treated through eigenvalues like with JWB
\end{proof}

\begin{corollary}
\label{cor:strongsol}
Consider the UCM formulation \eqref{eq:newucm_spatial_y} i.e. the % system of
conservation laws 
\begin{equation}
\label{eq:quasilinear}
% \partial_t q + A_x(q)\partial_x q + A_y(q)\partial_y q = B(q)
\partial_t q + \grad_qF_i(q)\partial_i q = B(q)
\end{equation}
with % a source term 
$F_i,B$ % analytic
$C^\infty$ in $q=(\rho,\rho u^i,\rho Y^{\alpha\beta},\rho F^i_\alpha)$
when $q\in \Acal^{+}$ lies in an open convex set, and \eqref{eq:quasilinear} is symmetric-hyperbolic, 
recall Prop.~\ref{prop:symmetrichyperbolic}. % for instance.
% non standard: not dissipative like $D(q)\equiv -\nabla_q S(q)\cdot B(q) \ge 0$
% % like BGK $B_i(q)=(q_i^\infty(q)-q_i)/\lambda_i$ where $q^\infty(q)$ lies in the convex domain for $q$, and $\lambda_i(q)>0$ \cite{chen-levermore-liu-1994}.
% \begin{equation}
% \label{eq:secondprinc}
% \partial_t S(q) + \div % \partial_x 
% \bG(q) = -D(q)
% \end{equation}
% holds for a true \emph{mathematical entropy} $S(q)$ % that is % a scalar function that depends only on the \emph{rotation-invariants} and is
% convex in the Galilean-invariants of the state $q$ 
For all state $q_0 \in \Acal^{+}$, % := \{\rho>0,\ \bA=\bA^{T}>0 \}$
and for all $\left(1+ d + d(d+1)/2 + d^2\right)$-dimensional % initial 
perturbation $\tilde q_0 \in H^s(\RR^d)$ in Sobolev space $H^s$ with $s>1+d/2$ 
such that $q_0 + \tilde q_0$ is compactly supported in $\Acal^{+}$,
there exists $T>0$ and a unique classical solution $q\in C^1([0,T)\times\RR^d)$ to \eqref{eq:quasilinear} 
such that $q(t=0)=q_0 + \tilde q_0$. 
Furthermore, $q - q_0 \in C^0([0,T),H^s)\cap C^1([0,T),H^{s-1})$. % = \tilde q
\end{corollary}
\begin{proof}
When the UCM reformulation \eqref{eq:newucm_spatial_y} is a \emph{symmetric-hyperbolic} system of conservation laws
with a smooth source term as in Prop.~\ref{prop:symmetrichyperbolic},
the small-time existence of smooth classical solutions is straightforward,
see e.g. Theorem 10.1 in \cite[Chapter 10]{benzonigavage-serre-2007}. % p.293
% To be precise, 
In Corr.~\ref{cor:strongsol}, one should however take care of the domain $\Acal^{+}$.
Now, it is open, convex and can be treated similarly to $\{\rho>0\}$ for the Euler equations of gas dynamics like in Theorem 13.1 of \cite[Chapter 13]{benzonigavage-serre-2007}. % p.399
\end{proof}

To our knowledge, 
Cor.~\ref{cor:strongsol} is the first well-posedness result % existence of unique solutions to
for the Cauchy problem 
of the compressible % fully-3D
multi-dimensional UCM model without background viscosity, % unlike xianpeng-guochun-2013
i.e. the first well-posedness result for a model of \emph{genuinely} causal viscoelastic flows (of Maxwell fluids) 
satisfying the locality principle.
Similarly to elastodynamics \cite{wagner-2009}, that latter result straightforwardly extends to the 
% (currently topical and much investigated) %  \cite{Bollada2012}
\emph{non-isothermal} compressible UCM models where \eqref{eq:newucm_spatial_y} is complemented with \eqref{eq:temperature}
when $\tilde E$ 
remains a convex extension, i.e. is strictly convex jointly for $q$ and $\theta$.

\smallskip

The  ``relaxation'' form % or BGK
of source terms in % the otherwise symmetric-hyperbolic system
\eqref{eq:newucm_spatial_y} also suggests the possibility of damping, 
and the existence of global (strong) solutions for sufficiently small initial data 
close to an equilibrium $q^\infty$ such that $B(q^\infty)=0$.
However, we leave this question for future works. % investigations.
Note that our symmetrizer has been obtained with the % strictly 
convex extension \eqref{eq:energy_spatial2}, which is not dissipative like \eqref{eq:energy_spatial}. % with the free energy !
But physically, dissipativity should be required, for instance the inequality $\le$ in \eqref{eq:energy_spatial}. % for  $E$
%
% Now $E$ cannot be jointly convex in $\rho F^i_\alpha$ and $\rho A^{\alpha\beta}$ or Y. % for isothermal flows
% On the other hand, $E$ in \eqref{eq:energy_spatial} is only (jointly) convex in $(\rho,\rho u^i,\rho F^i_\alpha)$
% and not with the full conservative variable $q$ like $\tilde E$.
%
So the setting is non-standard % for damping and % the study of 
% weak % measure-valued 
% solutions as well 
\cite{dafermos-2000}. % compare e.g. with \cite{christoforou-galanopoulou-tzavaras-2019} for polyconvex adiabatic thermoelasticity.
In particular, difficulties are also to be expected for numerical simulations by the standard discretization of symmetric-hyperbolic systems.
% if the physical dissipativity \eqref{eq:energy_spatial} is required, % as a stabilizing tool
Thus discretization will also be the object of future specialized works.
%
% pb with piola also like in elasto \cite{MR2337740,DespresMazeran2005} hui
% \cite{kluth-despres-2010}. % whatever the choice of variable

\smallskip

In any case, our new UCM formulation has promising applications % to viscoelastic flows 
in % computational rheology or 
geophysics % our primary goal, 
that can already be discussed % at a theoretical level 
here, see Section \ref{sec:2D}.
To that aim, % for further extension of the compressible UCM model too
let us first interpret physically the new variable $A^{\alpha\beta}$ in % the next 
section~\ref{sec:physics} below, which also shows the many potentialities of our new % symmetric-hyperbolic 
system as an extension of elastodynamics.

\subsection{UCM as extended elastodynamics % with additional variables
and beyond} %  NS etc.
% A physical interpretation of the new system of UCM conservation laws
% Elastodynamics with additional variables: UCM, NS\dots
\label{sec:physics}

Let us recall that the system % of conservation laws 
\eqref{eq:newucm_spatial} % equiv \eqref{eq:newucm_spatial_y} for smooth motions % actually 
models \emph{viscoelastic ``fluid''} flows (with stress relaxation)
insofar as the stress component $\btau$ defined in \eqref{eq:extrastress}
satisfies an equation of the type \eqref{eq:UCM}, % proposed Maxwell upper-convected 
with $\xi=0$ and additional terms due to compressibility.
\begin{proposition}
\label{prop:ucm}
In smooth motions defined by \eqref{eq:newucm_spatial} % or \eqref{eq:newucm_spatial_y},
the Cauchy % ``true'' 
stress $\bsigma$ in the spatial momentum balance has a viscoelastic component
\begin{equation}
\label{eq:viscoelasticcomponent}
\btau = 2 \rho \left( K \Hcal' F^i_\alpha A^{\alpha\beta} F^j_\beta - k_B \theta \delta_{ij} \right)
\end{equation}
solution to the modified Maxwell equation \eqref{eq:UCM} with one additional term
\begin{equation}
\label{eq:UCMmodified} 
\lambda \stackrel{\Diamond}{\btau} + \div\bu \: \lambda \: \btau  + \btau = 2 \dot{\mu}  \bD \,,
\end{equation}
an upper-convected time-rate with $\xi=0$ in \eqref{eq:gordon} and $\lambda=4K \Hcal'/\zeta$, $\dot{\mu} =2 \rho k_B \theta \lambda$.
\end{proposition}
\begin{proof}
This is a direct computation on recalling $\theta,K,H$ are constants (in the considered isothermal motions) so
$\bt:=\frac{\btau}{2 \rho k_B \theta} $ is solution to $\lambda \stackrel{\Diamond}{\bt} + \bt = 2\lambda \bD$.
\end{proof}
So formally, the stress in the compressible UCM model % defined by 
\eqref{eq:newucm_spatial} is then either ``elastic'' (like the stress in hyperelastic solids) or ``viscous'' (like extra-stress in Newtonian fluids) asymptotically as expected. % for smooth motions 
This is usual for a ``Maxwell % viscoelastic
fluid'': it tends to a Newtonian fluid at a characteristic time-scale $\lambda>0$,
and it is elastic at shorter times (recall the K-BKZ theory).

But % note that 
our UCM system \eqref{eq:newucm_spatial} can be \emph{precisely} interpreted 
as an extension of the elastodynamics of hyperelastic solids
using an additional ``material'' metric variable $\bA$ (attached to matter, $A^{\alpha\beta}$ in coordinates)
that describes locally % properties of matter
the physical state of the % continuous
material body, and our UCM fluid becomes Newtonian with viscosity $\dot{\mu} >0$ at ``large-time'' equilibrium 
thanks to a specific form of the relaxation limit for $A^{\alpha\beta}$.
On the one hand, other relaxation limits for $A^{\alpha\beta}$ are possible, which are physically meaningful and reminiscent of complex materials in the literature.
% This shows that 
The system  \eqref{eq:newucm_spatial} with one particular \emph{viscoelastic} relaxation limit for $A^{\alpha\beta}$
is only one instance % within
in a class that extends elastodynamics % of hyperelastic solids 
to complex materials with inelasticities, see subsection~\ref{sec:hookean}.
On the other hand, it % the system  \eqref{eq:newucm_spatial} 
suggests a new understanding of the Newtonian fluid, as explained below. % in subsection \ref{sec:newtonian}.

\subsubsection{The Newtonian viscous limit regime}
\label{sec:newtonian}

In smooth motions 
% such that \eqref{eq:deformationgradient} holds as well as Piola identities \eqref{eq:2Dpiola_identity} % \cite{wagner-1994}
% (recall section~\ref{sec:conservation}),
our formulation \eqref{eq:newucm_spatial} % or \eqref{eq:newucm_spatial_y}
of UCM contains standard formulations % like e.g. \eqref{eq:UCM} in the incompressible or ... in the comressible case !!! TODO
of section~\ref{sec:viscoelastic},
and the viscoelastic stress component % in \eqref{eq:newucm_spatial} and \eqref{eq:newucm_spatial_y}
then formally converges to $\btau\approx 2 \dot{\mu}  \bD$ % i.e. the same viscous contribution as in % compressible NS equations
 % at least
when $\lambda\ll1$, $\rho\theta\gg1$ and $\dot{\mu} =2 \rho k_B \theta \lambda$ is fixed, % bounded above and below
like in standard cases.

But moreover, unlike standard cases, % with the standard formulation of UCM,
our \emph{symmetric-hyperbolic} system % of conservation laws
allows % one to perform
the (first) proof % (to our knowledge) 
that the compressible UCM mode is mathematically sensible,
with univoque % well-defined and unique 
(strong) solutions % on small times 
to the Cauchy problem given smooth initial values.
So our system is also a new starting point to establish \emph{mathematically} the NS equations as a precise % rigorous 
limit of viscoelastic equations, of UCM in particular.
We will elaborate on this elsewhere,
see \cite{zbMATH06380509} for a recent mathematical justification of NS % from a similar viewpoint 
starting from a slightly modified % (linear) 
UCM model.

Furthermore, our formulation with conservation laws % naturally 
suggests one to study the formation and stability of shocks,
i.e. weak solutions with jumps across a % smooth
discontinuity surface, 
which are physically relevant for fluids. % with isolated vortices
Some conservation laws % (for non-physical quantities unlike mass and energy) 
could be irrelevant, % of course RENARDY198969  BASOMBRIO19931 Renardy-Rogers1993
but our new formulation % yet still nevertheless
at least suggests one an % standard 
approach how to perform shock computations % that are 
inline with seminal studies using \eqref{eq:firstorder_spatial_reduced} for gases, % majda klainerman christodoulou kevlahan_1997 PhysRevE.90.023002
and inline with more recent studies using %  elastodynamics
\eqref{eq:firstorder_spatial} for solids \cite{Morando2019}. % cite Kulikovskii too ??
% It would cover a fluid regime with ``viscosity''. % interesting to study vorticity formation across a shock -- like experimental vortex rings ?
This will be the subject of future works, as well as other quantitative studies discretizing our conservation laws with standard techniques.

\subsubsection{The Hookean elastic limit regime and its inelastic extensions}
\label{sec:hookean}

Unlike the standard UCM systems % like e.g. \eqref{eq:UCM} in the incompressible or ... in the comressible case !!! TODO
of section~\ref{sec:viscoelastic}, the formal limit of the viscoelastic stress 
$\btau\approx \mu \left( \frac{K \Hcal'}{k_B \theta} F^i_\alpha A^{\alpha\beta} F^j_\beta - \delta_{ij} \right)$ in \eqref{eq:newucm_spatial} % and \eqref{eq:newucm_spatial_y}
when $\lambda\gg1$, $\dot{\mu}\gg1$ and $\mu=\dot{\mu}/\lambda$ % 2 \rho k_B \theta
is clearly the same neo-Hookean elastic contribution 
%with % shear elastic
%modulus $\mu $ 
as in % a standard spatial description of 
elastodynamics % (see \cite{wagner-2009}) 
for a Riemannian body with inverse % body
metric $G^{\alpha\beta}=\frac{K \Hcal'}{k_B \theta} A^{\alpha\beta}$.
%that is %necessarily 
%time-independent in the material description.
%
Indeed, in the % ``time-indifference'' 
limit $\lambda\gg1$ where $A^{\alpha\beta}$ becomes time-independent in the material description,
$A^{\alpha\beta}$ % the latter 
can indeed % actually 
be interpreted as % compared with
the inner metric (inverse) $G^{\alpha\beta}$ of a Riemannian body, 
possibly non-Euclidean $G^{\alpha\beta}\neq\delta^{\alpha\beta}$ when the body is pre-stressed \cite{Kupferman2017}. % \cite{S0219199713500521}
But note that in general, the % tensor of metric type % metric 
variable $A^{\alpha\beta}$ solution to \eqref{eq:newucm_spatial} is not time-independent in the material description.
So it cannot be a material % inner 
metric like $G^{\alpha\beta}$ for the Riemannian flowing body 
as long as the mass balance % in Euclidean ambiant space  % unlike ?? \cite{yavari-golgoon-2019}
in \eqref{eq:newucm_spatial} % and \eqref{eq:newucm_spatial_y}
reads as usual for $\rho = |F^i_\alpha|^{-1} \sqrt{|G|} \hat\rho$. % without a source term !!  $\partial_t rho ... = 0$
In an evolution problem, the initial value of $A^{\alpha\beta}$ could nevertheless model pre-stress similarly to $G^{\alpha\beta}$ when non-Euclidean. % \cite{Kupferman2017}

In general, the new metric variable $A^{\alpha\beta}$
% which cannot be % is not % different from 
% the % inner
% Riemannian metric (inverse) $G^{\alpha\beta}$ of the body
should rather be compared with the % Riemannian 
metric $K^\alpha_kK^\beta_k$ that arises in elasto-plasticity, after adding a plastic deformation % variable 
$\bK^{-1}$ and a ``flow rule'' governing its evolution like in e.g. \cite{Krishnan-Steigmann2014} and many references therein, to extend % standard (``perfect'') 
elastodynamics with some inelasticities.
% smooth motions are equivalently formulated with material coordinates

This may be seen more easily in the material (or Lagrangian) description. % more natural for elasto-plastic solids % For smooth motions, \eqref{eq:newucm_spatial} is equivalent to :
\begin{proposition} 
\label{prop:newucm_material}
When $\hat\rho$ is constant, % homogeneous materials !!
the smooth isothermal (viscoelastic, compressible) UCM motions are equivalently described 
in spatial coordinates, 
by either \eqref{eq:firstorder_spatial}--\eqref{eq:extra}--\eqref{eq:extrastress}--\eqref{eq:conformation}, or \eqref{eq:newucm_spatial}, or \eqref{eq:newucm_spatial_y}, 
and in material coordinates, 
by
\beq
\label{eq:newucm_material}
\begin{aligned}
& \partial_t ( u^i\circ\bphi_t ) - \partial_\alpha \: \left( 
\left[ - (p/\rho) [F^{-1}]^\alpha_i + % \tau^{ij}  {eq:extrastress}
 2 (K \Hcal' A^{\alpha\beta} F^i_\beta % c^{ij} 
 - k_B \theta [F^{-1}]^\alpha_i \right]\circ\bphi_t ) \right) % \sigma^{ij}
= \rho f^i\circ\bphi_t
\\
& \partial_t ( F^i_\alpha\circ\bphi_t ) - \partial_\alpha ( u^i\circ\bphi_t ) = 0
\\
& \partial_t ( \rho^{-1}\circ\bphi_t ) - \partial_\alpha ( \hat\rho C^\alpha_i\circ\bphi_t u^i\circ\bphi_t ) = 0
\\
& \partial_t ( A^{\alpha\beta}\circ\bphi_t ) = \frac{4}\xi \left( k_B \theta \left( [F^{-1}]^{\alpha}_i [F^{-1}]^{\beta}_i \right) - K H A^{\alpha\beta} \right)\circ\bphi_t %{eq:ode}
\end{aligned}
\eeq
with $A^{\alpha\beta} = [F^{-1}]^{\alpha}_i c^{ij} [F^{-1}]^{\beta}_j \in S^{++}(\R^{d\times d})$. 
Furthermore, if $p=\partial_{\rho^{-1}}e_0(\rho,\theta)$, % so $p=C_0 (\hat\rho|F^i_\alpha|)^\gamma$ when $e_0(|F^i_\alpha|)$ is the polytropic law \eqref{eq:polytropic} 
\beq
\label{eq:energy_material}
\partial_t ( [E/\rho]\circ\bphi_t ) - \partial_\alpha \left( [u^i \sigma^{ij}]\circ\bphi_t \right) = - \frac{4}\xi \Dcal\circ\bphi_t
\eeq
then holds with % an energy 
$E/\rho=\frac{|u|^2}2 + \psi$,
% and a stress % Cauchy
$\sigma^{ij}=\partial_{F^i_\alpha} \psi $, % Doyle Ericksen 
$\psi$ as in \eqref{eq:Helmholtzfreeenergy}
% \beq
% \label{eq:Helmholtzfreeenergy}
% \psi = e_0 + K H F^i_\alpha F^i_\beta A^{\alpha\beta}   % \Hcal\left(F^i_\alpha A^{\alpha\beta} F^i_\beta\right) 
% - k_B \theta \log|F^i_\alpha A^{\alpha\beta} F^i_\beta| % {eq:dumbbellfreeenergy}
% \eeq
% the Helmholtz free energy % already in 1.447210 cited by bird-b
% say for suspensions where the solvent contribution $e_0(\rho,\theta)$ can be the polytropic law \eqref{eq:polytropic}
and % the \emph{dissipation} 
$\Dcal\ge0$ as in \eqref{eq:dumbbelldissipation}.
\end{proposition}
\begin{proof}
Recalling Piola identities \eqref{eq:2Dpiola_identity}, the system \eqref{eq:newucm_material} and the additional law \eqref{eq:energy_material}
for $u^i,A^{\alpha\beta},F^i_\alpha$ and $E /\rho$ as functions of $t,\ba$ are straightfrowardly derived from
\eqref{eq:newucm_spatial} and \eqref{eq:energy_spatial}
for $u^i,A^{\alpha\beta},F^i_\alpha$ and $E /\rho$ as functions of $t,\bx=\bphi_t(\ba)$.
% \cite{wagner-1994}.
\end{proof}
%The material description \eqref{eq:newucm_material} can be % more easily 
%compared with the ``perfect'' elastodynamics system \eqref{eq:firstorder_material} for \emph{isotropic} bodies 
%(bodies indifferent to rotations of the material coordinates system).
% 
When $A^{\alpha\beta}$ is time-independent ($\lambda\gg1$), the stress in \eqref{eq:newucm_material} is the sum of
$$ 
\left[ - (p/\rho + 2 k_B \theta) [F^{-1}]^\alpha_i + 2 K \Hcal' A^{\alpha\beta} F^i_\beta \right]\circ\bphi_t
$$
i.e. Piola-Kirchhoff stress $S^i_\alpha$ for neo-Hookean materials, plus an additional % compressible 
term % strain energy
to account for volumetric changes, recall section~\ref{sec:constitutive}.
% Otherwise, 
More generally, % when $A^{\alpha\beta}$ evolves, 
the viscoelastic stress \eqref{eq:viscoelasticcomponent} % can be interpreted as
can be interpreted as the mean-field approximation \def\Ksf{\mathsf{K}} 
$A^{\alpha\beta}=\EE\left(\Ksf^\alpha_k\Ksf^\beta_k\right)$
of an % stochastic \emph{rate-dependent} 
elastoplastic model \cite{Krishnan-Steigmann2014} 
% in $KH A^{\alpha\beta}- k_B \theta [F^{-1}]^\alpha_k[F^{-1}]^\beta_k=KH \EE\left(K^\alpha_kK^\beta_k\right)-k_B \theta [F^{-1}]^\alpha_k[F^{-1}]^\beta_k$ % a Mandel stress \cite{Gurtin_Fried_Anand_2010-book}
with stochastic flow rule % semi Martingale
\begin{equation}
\label{eq:ito}
[\Ksf^{-1}]^i_\alpha \left( d\Ksf^\alpha_j + u^k \partial_k \Ksf^\alpha_j \right)= - \delta^i_j \frac{2 K \Hcal'}\xi dt + [\Ksf^{-1}]^i_\alpha [F^{-1}]^\alpha_k \sqrt{\frac{2 k_B \theta}{\xi\: d}} d W^k_j(t)
\end{equation}
in Ito notation, using a probability space with expectation % denoted 
$\EE$, and $d^2$ Wiener processes % modelling Brownian motions
denoted $W^k_j(t)$, $k,j=1\dots d$.

The % physical 
interpretation of white noise in \eqref{eq:ito} is left to future works,
as well as the comparison with the kinetic theory of dumbbells % usually invoked 
in rheology to establish viscoelastic models \cite{bird-curtiss-armstrong-hassager-1987b}, recall section~\ref{sec:viscoelastic}.
But one can already note here the potential of the new system, with a new metric variable $A^{\alpha\beta}$ % in \eqref{eq:newucm_spatial} and \eqref{eq:newucm_spatial_y} 
to unify various physically-relevant extensions of elastodynamics towards inelastic bodies. % INCL WITH CREEP !!
In particular, UCM can be interpreted from the elastoplastic viewpoint % with a distortion metric
with \eqref{eq:ito}, as a \emph{rate-dependent} flow rule which models Newtonian viscous ``fluid inelasticities'' when $\lambda\ll1$.
%
%Note that when $\lambda\ll1$,
%$\partial_t A^{\alpha\beta} \approx -\partial_t( [F^{-1}]^\alpha_k[F^{-1}]^\beta_k )$ % given the source term in $A^{\alpha\beta}$ evolution
%and the momentum balance in the spatial description \eqref{eq:newucm_spatial} % and \eqref{eq:newucm_spatial_y}
%can also be seen compatible with Newtonian viscosity insofar as the evolution of the inverse right Cauchy-Green deformation tensor reads
%% @book{chaves2013notes,
%%   title={Notes on Continuum Mechanics},
%%   author={Chaves, E.W.V.},
%%   isbn={9789400759862},
%%   series={Lecture Notes on Numerical Methods in Engineering and Sciences},
%%   url={https://books.google.fr/books?id=2Q1HAAAAQBAJ},
%%   year={2013},
%%   publisher={Springer Netherlands}
%% } possible explicit ref
%\beq
%\label{eq:inverserightCauchyGreen}
%F^i_\alpha \partial_t( [F^{-1}]^\alpha_k[F^{-1}]^\beta_k ) F^j_\beta = - 2 D(u)^{ij} \,.
%\eeq
Reciprocally, the standard rate-independent elastoplastic flow rules can be interpreted from the viscoelastic viewpoint
as yielding materials with permanently-fading memory.
Variations of the relaxation limit of $\bA$ to model various % (fluid) 
inelasticities (i.e. rheologies) as extensions of polyconvex elastodynamics will be investigated in future works.
Here, we focus on viscoelasticity.
% and we would now like to illustrate the fact that the new UCM system is already interesting 
% % not only for straightforward application to (isothermal, compressible) flows in computational rheology, but also
% for applications in % computational 
% geophysics.

\section{Application to geophysical % hydrostatic incompressible 2D shallow free-surface gravity
water flows} % of Saint-Venant type
\label{sec:2D}

Numerous geophysical flows are hardly-compressible \emph{shallow} gravity flows with a free surface, well % sufficiently-accurately 
described by the % inviscid hydrostatic  
two-dimensional (2D) shallow-water equations % often 
attributed to Saint-Venant \cite{saint-venant-1871} for many purposes. % hydraulic engineering
For instance, the Saint-Venant systems usually forecast well % the fast 
river floods, in particular when the equations % satisfy causality and locality principles
are non-diffusive and local \cite{chow-1959}. % for a long time
However, for some % large-time 
hydraulic applications, % for hydraulic engineers who favour better predictivity at a moderately increased computational price
it is still unsure how to % improve energy dissipative terms and
account for % some 
viscous effects like large % vertical 
vortices and recirculation zones. %  manning-1891   <== due to viscosity !!
We next show that, in the frame of free-boundary flows, our compressible UCM formulation can serve such a purpose
without introducing diffusion and losing the local character of useful Saint-Venant equations,
after reduction \emph{\`a la Saint-Venant} to a \emph{viscoelastic} shallow-water system % for 2D shallow flows with a hydrostatic pressure component
that generalizes the usual % inviscid and viscous 
shallow-water systems. 
But first, let us recall in subsection \ref{sec:saintvenant} the standard Saint-Venant equations with and without diffusion (of velocity). % (of momentum).

% some usual, standard
\subsection{Standard Saint-Venant % 2D 
models % derivation
for shallow % free-surface
water flows}
\label{sec:saintvenant}

Let us equip the Euclidean ambiant space % with a Galilean frame 
with 
% origin $\bzero$ as a reference point % for Galilean transformations within Galilean relativity theory (corrdinates + point define a Galilean inertial frame of reference)
% and a 
Cartesian % system of 
coordinates $(\be_x,\be_y,\be_z)$ % \bzero has coordinates (0,0,0)
so % $\bf=-g\be_z$ constant % and uniform
$(f^x,f^y,f^z):=(0,0,-g)$ is a constant gravity field with magnitude $g$.

We consider a fluid % continuum in
filling % a moving domain 
$\Dcal_t:=\{z^b(x,y)<z<z^b(x,y)+H(t,x,y)\}$ % which we 
supposedly % assumed to be 
a smooth % infinite 
layer with surface of outward unit normal 
$$
\bn_{z^b+H} = \frac1{\sqrt{1+|\grad_H (z^b+H)|^2}} \left(-\partial_x(z^b+H),-\partial_y(z^b+H),1\right) 
$$
where $\grad_H = (\partial_x,\partial_y)$ is the gradient associated with horizontal divergence $\div_H$. % for horizontal fields.

The fluid flow is assumed governed by the \emph{reduced} spatial % (Eulerian) 
description \eqref{eq:firstorder_spatial_reduced} in the moving layer $\Dcal_t$ (as usual for fluids) and by the so-called kinematic condition
\begin{equation}
\label{eq:kinematic}
\partial_t H + u^x \partial_x (z^b+H) + u^y \partial_y(z^b+H) = u^z \sqrt{1+|\grad_H (z^b+H)|^2}
\end{equation}
at $z=z^b+H$. % assuming particles have velocity $\bu$,
Then, along with the free-surface % normal, dynamic 
condition
\begin{equation}
\label{eq:dynamic}
%\sigma^{zz} = \sigma^{zx}\partial_x (z^b+H) + \sigma^{zy}\partial_y (z^b+H) \,,
% \forall i \in \{x,y,z\} 
\quad \sigma^{ij}n_{z^b+H}^jn_{z^b+H}^i = 0 \,,
\end{equation}
some % well-chosen !! (enforcing impermeability is not obvious with hyperbolic systems)
constitutive assumptions for 
%(the Cauchy stress % formula 
%of) 
the fluid % can usefully (with well-posed Cauchy solutions)
are known to close the % full 
3D evolution system \eqref{eq:kinematic}--\eqref{eq:firstorder_spatial_reduced}.
% If one can choose a perfect fluid with \eqref{eq:cauchystress_pressure} and the free-surface condition $p|_{z=H}=0$
% THEN THIS IS ENOUGH only if b\to-\infty \eqre{lindblad-luo-2018} otherwise it is not clear how an impermeable wall can be ensured !!
For instance, if the fluid is incompressible -- which gives a special meaning to $p$ in \eqref{eq:extra} --, with Newtonian extra-stress \eqref{eq:extranewt}, 
then one can define unique solutions to Cauchy % initial-boundary value 
problems for % the 3D system 
\eqref{eq:kinematic}--\eqref{eq:firstorder_spatial_reduced}
on requiring impermeability % condition 
$\bu\cdot\bn_{z^b}=0$ at $z=z^b$, % no-penetration !!
%i.e.
%\begin{equation}
%\label{eq:impermeability}
%u^x \partial_x b + u^y \partial_y b = u^z \sqrt{1+|\grad_H b|^2} \text{ at } z=b
%\end{equation}
% along with 
% a free-surface % dynamic 
% condition \eqref{eq:dynamic} % typically usually
% $$ %\forall i \in \{x,y,z\} \quad 
% - p + \tau^{ij}n_{z^b+h}^jn_{z^b+h}^i = 0
% $$ % p|_{z=H} \sqrt{1+|\partial_x (z^b+H)|^2 + |\partial_y (z^b+H)|^2} + \tau^{zz} = \tau^{zx}\partial_x (z^b+H) + \tau^{zy}\partial_y (z^b+H)=0
% for the normal stress at $z=z^b+H$, 
plus Navier friction conditions % for the tangential stress 
at $z=z^b,z^b+H$
\begin{equation}
\label{eq:navier}
% at z \in \{z^b+H,z^b\} % \sigma^{ij}n^j_z n^i_z - \sigma^{ix}\partial_x (z^b+H) + \sigma^{iy}\partial_y (z^b+H) 
% \quad 
\sigma^{ij}n^j_z - (\sigma^{kj}n^j_z n^k_z) n^i_z 
% \equiv \tau^{ij}n^j_z - (\tau^{kj}n^j_z n^k_z) n^i_z 
= - k_z \left( u^i - (u^jn_z^j)n_z^i \right)
\end{equation}
% \begin{multline} \left( \sigma^{iz} - \sigma^{ix}\partial_x (z^b+H) + \sigma^{iy}\partial_y (z^b+H) \right)/\sqrt{1+|\partial_x (z^b+H)|^2 + |\partial_y (z^b+H)|^2} 
with $k_{z^b+H}=0$ % at $z=z^b+H$ 
(pure slip at free surface) % WITHOUT WIND !! otherwise tension conditions
and $k_b\ge 0$ % at $z=b$ 
(dissipation at bottom). % \cite{Masmoudi-Rousset2012}
But the % Newtonian 
incompressible Navier-Stokes free-surface model % has not only the drawback of being
is barely tractable for numerical applications, % theoretically or numerically
let alone the propagation of information at infinite speed.
% In numerical practice, 
For the computation of % numerous
hardly-compressible thin-layer (i.e. shallow) geophysical water flows % gravity flows with a free surface used in hydraulics 
with uniform mass density $\rho>0$, % constant or stratified
one often prefers % to the 3D system % for water flows (rivers, oceans) at constant temperature and pression...
a % more tractable 
2D model reduced after Saint-Venant \cite{saint-venant-1871}, moreover local. Indeed, let us recall:
\begin{proposition}
\label{prop:depthaverage}
Given a family $z^b_\epsilon$, $\epsilon\to0^+$ of smooth topographies,
assume there exist % a \emph{bounded} family of regular
bounded regular solutions $H_\epsilon$, $\rho_\epsilon$, $\bu_\epsilon$, $p_\epsilon$ % extra-stress 
$\tau_\epsilon$ % \dots 
to \eqref{eq:kinematic}--\eqref{eq:firstorder_spatial_reduced}--\eqref{eq:extra}--\eqref{eq:dynamic}--\eqref{eq:navier}
for $(t,x,y)\in[0,T)\times\R\times\R$, $z\in(z^b_\epsilon,z^b_\epsilon+H_\epsilon)$
such that $X_\epsilon=X_0+\epsilon X_1+O(\epsilon^2)$ holds pointwise % strongly
for % all % the field variables,
$X\in\{z^b,H,\rho,\bu,p,\tau\}$ as well as % when $\epsilon\to0$
\begin{itemize}
 \item $\grad_H z^b_\epsilon = O(\epsilon) = H_\epsilon$, i.e. $X_\epsilon=\epsilon X_1+O(\epsilon^2)$ for $X\in\{b,H\}$ 
 % in fact % topo at constant pres
 \item $\rho_0$ is constant 
 for all $t,x,y$ and $z^b_\epsilon<z<z^b_\epsilon+H_\epsilon$, hence % $\rho_0(z)$ stratification => barocline
\begin{equation}
\label{eq:nearincompressibility}
\div\bu_\epsilon = O(\epsilon) \text{ as } \epsilon\to0
\end{equation}
 \item at $z=z^b_\epsilon$, $\bu_\epsilon\cdot\bn_{z^b_\epsilon}=0$ and \eqref{eq:navier} with $k_{z^b_\epsilon}=O(\epsilon)$ 
 \item at $z=z^b_\epsilon+H_\epsilon$, \eqref{eq:dynamic} and  \eqref{eq:navier} with $k_{z^b_\epsilon+H_\epsilon}=O(\epsilon^2)$.
\end{itemize}

Then, denoting $\bu^H=(u^x,u^y)$, it holds 
\begin{align}
\label{eq:massconservationeps}
& \partial_t H_\epsilon + \div_H( H_\epsilon \bU_\epsilon ) = O(\epsilon^2) 
\,,
\\
\nonumber 
& \partial_t(H_\epsilon\bU_\epsilon) + \div_H\left( \int_{z^b_\epsilon}^{z^b_\epsilon+H_\epsilon}dz\; \bu^H_\epsilon \otimes \bu^H_\epsilon \right) 
= O(\epsilon^3)
\\
\label{eq:momentumconservationeps}
& \quad + \div_H(H_\epsilon (\Sigma^{zz}_\epsilon\bI -\bSigma^H_\epsilon) )
- g H_\epsilon \grad_H (z^b_\epsilon+H_\epsilon) - \bu^H_\epsilon\:k_{z^b_\epsilon}/\rho_0
\,,
\end{align}
where % the 2D velocity 
$ %\beq \label{eq:horizontalstress}
 \bSigma^H_\epsilon = \Sigma^{xx}_\epsilon\be_x\otimes\be_x + \Sigma^{yy}_\epsilon \be_y\otimes\be_y + \Sigma^{xy}_\epsilon \be_x\otimes\be_y + \Sigma^{yx}_\epsilon \be_y\otimes\be_x$, % \eeq
$\Sigma^{zz}_\epsilon$ and % the horizontal stress tensor
$\bU_\epsilon(t,x,y)=U^x_\epsilon(t,x,y)\be_x+U^y_\epsilon(t,x,y)\be_y$ % the vertical stress 
are defined by
\begin{equation}
\label{eq:depthaveragedeps}
U^i_\epsilon = \frac1{H_\epsilon}\int_{z^b_\epsilon}^{z^b_\epsilon+H_\epsilon} dz\; u^i_\epsilon 
\qquad % U^y_\epsilon = \frac1{H_\epsilon}\int_{z^b_\epsilon}^{z^b_\epsilon+H_\epsilon} dz\; u^y \,,
\Sigma^{ij}_\epsilon = \frac1{\rho_0} \frac1{H_\epsilon}\int_{z^b_\epsilon}^{z^b_\epsilon+H_\epsilon} dz\; \tau^{ij}_\epsilon \,.% \qquad 
\end{equation}
\end{proposition}

Prop.~\ref{prop:depthaverage} rephrases a result that can be found in many places, see \cite{ferrari-saleri-2004,bouchut-boyaval-2015}
and references therein.
% articles and textbooks. % though not completely in that form
But we % nevertheless 
briefly recall its proof below % for the sake of completeness and
for future reference. % using our notations !

\begin{proof} The proof classically consists in three main steps: 
\begin{enumerate}
 \item 
 $\grad_H z^b_\epsilon = O(\epsilon) = H_\epsilon$ 
 first imply $u^z_\epsilon=O(\epsilon)$, % R1 ***  either
 at $z=z^b_\epsilon+H_\epsilon$ with \eqref{eq:kinematic} or
 at $z=z^b_\epsilon$ with impermeability and % next 
 in the whole layer by \eqref{eq:nearincompressibility},
 then \eqref{eq:massconservationeps}, % R2 ***
% $$
% u^z_\epsilon = (u^x_\epsilon\partial_x+u^y_\epsilon\partial_y)(z^b_\epsilon+H_\epsilon) - \div_H(H_\epsilon \bU_\epsilon) + O(\epsilon^2)
% $$
 \item \eqref{eq:navier} 
 first imply $\tau^{xz}_\epsilon,\tau^{yz}_\epsilon=O(\epsilon)$,
 at $z=z^b_\epsilon+H_\epsilon$ with $k_{z^b_\epsilon+H_\epsilon}=O(\epsilon^2)$
 and % next
 in the whole layer by the horizontal momentum balance % \eqref{eq:momentum} with $i\in\{x,y\}$
\beq
\label{eq:momentum}
\rho_0(\partial_t+u^j_\epsilon\partial_j)u^i_\epsilon+\partial_ip_\epsilon-\partial_j\tau^{ij}_\epsilon=\rho_0 f^i
\eeq
 with $i\in\{x,y\}$, % R3 ***
 then $p_\epsilon-\tau_\epsilon=\rho_0 g(z^b_\epsilon+H_\epsilon-z)+O(\epsilon)$ % R4 ***
 with \eqref{eq:momentum} for $i=z$ 
 and \eqref{eq:dynamic} i.e. $p_\epsilon-\tau_\epsilon^{zz}=O(\epsilon^2)$ at $z=z^b_\epsilon+H_\epsilon$
 \item  Depth-averageing the % latter 
horizontal momentum balance \eqref{eq:momentum} for $i\in\{x,y\}$ with % Leibniz rule and boundary conditions
\eqref{eq:navier} and \eqref{eq:massconservationeps} 
yields \eqref{eq:momentumconservationeps}.
\end{enumerate}
\end{proof}

Given Prop.~\ref{prop:depthaverage}, the next step is to infer a 2D % reduced 
model % for motions
of evolution form % i.e. with well-posed Cauchy problems
\begin{align}
\label{eq:massconservation}
& \partial_t H + \div_H( H \bU ) = 0
\\
& \partial_t(H \bU) + \div_H\left( H \bU \otimes \bU \right) = \nonumber 
\\
\label{eq:momentumconservation}
& \quad + \div_H(H (\Sigma^{zz}\bI -\bSigma^H) ) - g H \grad_H (z_b + H) - k H \bU
\end{align}
% \begin{equation}
% \label{eq:2Dmomentumbalance} 
%  (\partial_t \bU + (\bU\cdot\grad_H) \bU) + g \grad_H (H+b) + \frac1H \div_H( H\Sigma^{zz}\bI - H\bSigma^H ) = - K \bU
% \end{equation}
% where $\grad_H=(\partial_x,\partial_y)$ is the ``horizontal gradient'', 
that is \emph{closed} (with equations for % $\bI$ the identity % second-order tensor 
the friction parameter $k>0$ and stresses) 
% \begin{align}
% \label{eq:SVH}
% & \partial_t H + \div( H \bU ) = 0 
% \,,
% \\
% \label{eq:SVHUV}
% & \partial_t(H\bU) + \div(H\bU\otimes\bU + (gH^2/2)\bI - H\bSigma^H ) = - g H \grad_H b - K\bU
% \,,
% \end{align}
so that it can be used for fast and simple predictions of % incompressible 
free-surface gravity flows. % in environmental hydraulics \cite{hervouet-2007}
% with error bound
% after non-dimensionalizing \eqref{eq:kinematic}--\eqref{eq:firstorder_spatial}--\eqref{eq:extra} by a length $L$ 
% and an associated % characteristic time $\sqrt{L/g}$ % for gravity waves 

\begin{corollary}
\label{cor:depthaverage}
Assume the family of solutions % considered 
in Prop.~\ref{prop:depthaverage} also satisfies
\begin{equation}
\label{eq:dispersioneps} 
\int_{z^b_\epsilon}^{z^b_\epsilon+H_\epsilon}dz\; \bu^H_\epsilon = H_\epsilon \bu^H_\epsilon + O(\epsilon^2)
\end{equation}
and ${\bu^H_\epsilon k_{z^b_\epsilon}}/{\rho_0} = k H_\epsilon \bU_\epsilon + O(\epsilon^2)$
then, for small $\epsilon$, % the first-order approximation of $H_\epsilon$
$H_1$ can be approximated % itself as $\epsilon\to0$
by a solution $H\approx H_1$, $\bU\approx \bu^H_0$ to the % closed 
Saint-Venant system \eqref{eq:massconservation}--\eqref{eq:momentumconservation} where 
\begin{itemize}
 \item[a)] $\Sigma^{zz}\bI = \bSigma^H$,\ % (the ``inviscid'' -- non-diffusive -- Saint-Venant system)
  if $\tau_\epsilon^{iz}=O(\epsilon^2)=\tau_\epsilon^{ii}-\tau_\epsilon^{zz}$ for $i\in\{x,y\}$;
 \item[b)] $\Sigma^{zz}=-\nu_\epsilon(\partial_xU^x+\partial_yU^y)$,
  $\Sigma^{ij}=\nu_\epsilon(\partial_iU^j+\partial_jU^i)$ for $i,j\in\{x,y\}$,\\ % (the ``viscous'' -- diffusive -- Saint-Venant system)
  \smallskip
  if $\tau_\epsilon^{ij}=\nu_\epsilon\partial_j u_\epsilon^i$ with $\nu_\epsilon=O(\epsilon)$ for $i,j\in\{x,y,z\}$.
% \begin{multline}
% \label{eq:viscoussteadystate}
% \Sigma^{zz}=-\nu(\partial_xU^x+\partial_yU^y) \quad
% \Sigma^{ii}=2\nu\partial_iU^i=:2\nu D_H^{ii}  % \qquad \Sigma^{yy}=2\nu\partial_yU^y=:2\nu D_H^{yy} 
% \\ 
% \Sigma^{xy}=\Sigma^{yx}=\nu(\partial_xU^y+\partial_yU^x)=:2\nu D_H^{xy} \,.
% \end{multline} % i.e. the viscous shallow-water system $\bD_H=\frac12(\grad_H\bU+\grad_H\bU^T)$
\end{itemize}
% corrector !! ??
\end{corollary}

\begin{proof} To show Cor.~\ref{cor:depthaverage} starting from Prop.~\ref{cor:depthaverage}, % it suffices to 
see that \eqref{eq:dispersioneps} implies
$$
\int_{z^b_\epsilon}^{z^b_\epsilon+H_\epsilon}dz\; \bu^H_\epsilon \otimes \bu^H_\epsilon 
= H_\epsilon \bU_\epsilon\otimes \bU_\epsilon + O(\epsilon^3)
$$
and note that $\bu^H_0=0$ if $\lim_{\epsilon\to0}\bu^H_\epsilon=0$ at $z=z^b_\epsilon$ under \eqref{eq:dispersioneps},
so the value of $K$ is undetermined but useless then. 
\end{proof}

% \begin{align*}
% % \label{eq:SVHH}
% & \partial_t H + \partial_x( H U^x ) + \partial_y(H U^y) = 0
% \\
% % \label{eq:SVU}
% % & \partial_t U + (U \partial_x + V \partial_y) U + \partial_x (P+\Sigma^{zz}-\Sigma^{xx}) - \partial_y \Sigma^{xy} = -K U
% % & \partial_t(HU) + \partial_x( HUU ) + \partial_y( HUV ) + \partial_x (HP+H\Sigma^{zz}-H\Sigma^{xx}) - \partial_y (H\Sigma^{xy}) = -KHU
% % \label{eq:SVHU}
% & \partial_t(HU^x) + \partial_x( HU^xU^x + HP+H\Sigma^{zz}-H\Sigma^{xx} ) + \partial_y( HU^xU^y - H\Sigma^{xy} ) = -KHU^x
% \\
% % \label{eq:SVV}
% % & \partial_t V + (U \partial_x + V \partial_y) V - \partial_x \Sigma^{yx} + \partial_y (P+\Sigma^{zz}-\Sigma^{yy}) = -K V
% % & \partial_t(HV) + \partial_x( HUV ) + \partial_y( HVV ) - \partial_x (H\Sigma^{yx}) + \partial_y (HP+H\Sigma^{zz}-H\Sigma^{yy}) = -KHV
% % \label{eq:SVHV}
% & \partial_t(HU^y) + \partial_x( HU^xU^y - H\Sigma^{yx} ) + \partial_y (  HU^yU^y + HP+H\Sigma^{zz}-H\Sigma^{yy} ) = -KHU^y
% \end{align*}

So the flows of \emph{slightly viscous} fluids with (quasi-)Newtonian extra-stress (case b) % can be computed with
could be approximated through % by the solutions to 
a \emph{diffusive} Saint-Venant system \eqref{eq:massconservation}--\eqref{eq:momentumconservation} % as an asymptotic limit of the Navier-Stokes equations
where $k=k_{z^b_\epsilon}(1-Hk_{z^b_\epsilon}/3\nu_\epsilon)$,
or % simply 
the \emph{non-diffusive} % asymptotic % singular !
limit system with $k=k_{z^b_\epsilon}$ when the extra-stress is negligible (case a). % in the fluid see \cite{bouchut-boyaval-2015} and references therein.
In any case, the 2D system admits smooth causal solutions to % well-posed
Cauchy problems % which model reversible flows given smooth initial conditions and
that preserve $H\ge0$ and satisfy 
\begin{multline}
\label{eq:energyconservation} 
% \partial_t(HE) + \div(H(E+z^b+H+\Sigma^{zz})\bU-H\bSigma^H\cdot\bU) = -KH |\bU|^2 - HD for $E = \frac{1}2 \left( |\bU|^2 + gH/2 \right)$ when b=cst
\partial_t\left( H\left(\frac12|\bU|^2 + \frac12gH + g z_b \right) \right) 
\\
+ \div_H\left( H\left(\frac12|\bU|^2 + gH + g z_b +\Sigma^{zz} \right)\bU - H\bSigma^H\cdot\bU \right) 
\\ = -kH |\bU|^2 - HD 
\end{multline}
% as a formulation of the thermodynamics second-principle given a compatible evolution of the specific internal energy $e$
% \beq
% \label{eq:internalenergydissipation} 
% \partial_t e + (\bU\cdot\grad)e + (P+\Sigma^{zz})\div\bU - \bSigma^H:\grad\bU = -D \,.
% \eeq
with % a dissipation 
$D=2\nu_\epsilon\left(|\grad_H\bU+\grad_H\bU^T|^2/4+2|\div_H\bU|^2\right)\ge0$, % non-zero when $\bU\neqcst$ % the velocity $\bU$ is non-constant
indeed.

But the latter 2D flows suffer the same problems as their 3D counterparts.
The diffusive shallow-water system is a 2D version of damped Navier-Stokes equations,
with a tensor % bulk and shear 
viscosity as diffusion coefficients: it does not produce \emph{local} causal motions.
And the non-diffusive shallow-water system exactly coincides with a 2D version of the Euler equations with damping $K$,
for polytropic fluids with $e_H=gH/2$ and % mechanical 
energy $E \equiv \frac{H}2 \left( |\bU|^2 + e_H \right)$: % strictly convex in $(H,H\bU)$ on the domain $H>0$
it lacks viscosity to control % e.g. the size of 
vortices. % (with vertical axis) 
% which is important for (moderately) long-time predictions.
% And finally
Then, the same question arises as in the full 3D framework: could causal motions also be \emph{local} using a viscoelastic 2D flow model ? % to define them

\subsection{Viscoelastic Saint-Venant models with UCM fluids}
\label{sec:maxwell}

Viscoelastic shallow-water models have been proposed in the literature, but we are not aware of % good viscoelastic shallow-water 
2D models with well-posed Cauchy problems.
For instance,
to close % complement the Saint-Venant system 
\eqref{eq:massconservation}--\eqref{eq:momentumconservation}
with Maxwell equations for Cauchy stress, we proposed in \cite{bouchut-boyaval-2015}:
\begin{align}
\nonumber
& (\partial_t + U^x\partial_x + U^y\partial_y) \bSigma^H - (\grad_H\bU) \bSigma^H - \bSigma^H (\grad_H\bU)^T \\
& \quad = ( \nu_\epsilon (\grad_H\bU + \grad_H\bU^T)-\bSigma^H)/\lambda
\label{eq:Sigmah}
\\
\label{eq:Sigmazz}
& (\partial_t + U^x\partial_x + U^y\partial_y) \Sigma^{zz} + 2 (\div_H\bU) \Sigma^{zz} = (-2\nu_\epsilon\div_H\bU-\Sigma^{zz})/ \lambda
\end{align}
for small $\epsilon$ and a % characteristic time 
$\lambda>0$ given,
% with $\grad_H\bU\: \bSigma^H $ a matrix product and $\bD_H=\frac12(\grad_H\bU+\grad_H\bU^T)$
which for UCM % a Maxwell fluid
is the natural 2D generalization of our 1D viscoelastic Saint-Venant model \cite{bouchut-boyaval-2013}.
But similarly to the standard system for 3D flows of % (compressible) 
UCM fluids, % and unlike the 1D model \cite{bouchut-boyaval-2013},
the quasilinear 2D system
\eqref{eq:massconservation}--\eqref{eq:momentumconservation}--\eqref{eq:Sigmah}--\eqref{eq:Sigmazz}
lacks additional structure such as symmetric hyperbolicity
and we do not know how to define % viscoelastic flows % motions as 
solutions to Cauchy problems with that system.

On the contrary, we show in the sequel that the \emph{new} 3D (compressible) UCM model % with conservation laws 
of Section~\ref{sec:model} can be used to derive a \emph{symmetric-hyperbolic} % system of conservation laws that defines an interesting 
viscoelastic 2D Saint-Venant model with UCM fluids,
having \eqref{eq:massconservation}--\eqref{eq:momentumconservation}--\eqref{eq:Sigmah}--\eqref{eq:Sigmazz}
% \eqref{eq:SVH}--\eqref{eq:SVHUV}--\eqref{eq:Sigmah}--\eqref{eq:Sigmazz} 
as a % (non-conservative) 
subsystem.
To that aim, we first revise Prop.~\ref{prop:depthaverage} with assumptions allowing to depth-average all equations in \eqref{eq:firstorder_spatial},
% since the starting point is the \emph{full} Eulerian description \eqref{eq:firstorder_spatial} and not simply \eqref{eq:firstorder_spatial_reduced}.
guided by the interest % need 
of~\eqref{eq:dispersioneps} for closure.

\begin{proposition}
\label{prop:depthaverage2}
Given a family $z^b_\epsilon$, $\epsilon\to0^+$ of smooth topographies,
assume there exist % a \emph{bounded} family of regular
bounded regular solutions $H_\epsilon$, $\bF_\epsilon$, %$\rho_\epsilon$, 
$\bu_\epsilon$, $\tau_\epsilon$ % no extra-stress 
to % the Eulerian description 
\eqref{eq:kinematic}--\eqref{eq:firstorder_spatial}--\eqref{eq:dynamic}--\eqref{eq:navier}
in $(t,x,y)\in[0,T)\times\R\times\R$, $z\in(z^b_\epsilon,z^b_\epsilon+H_\epsilon)$
that define the motions of fluid layers with reference configurations $\{c\in(0,\epsilon)\}$ in a Cartesian frame $(\be_a,\be_b,\be_c)$,
such that $X_\epsilon=X_0+\epsilon X_1+O(\epsilon^2)$ holds pointwise % strongly
for % all % the field variables,
$X\in\{b,H,\bF,\bu,p,\tau\}$ as well as % when $\epsilon\to0$
\begin{itemize}
 \item $\grad_H z^b_\epsilon = O(\epsilon) = H_\epsilon$, which means $X_\epsilon=\epsilon X_1+O(\epsilon^2)$ for $X\in\{b,H\}$ 
 % in fact % topo at constant pres
 \item $|\bF_0|=1$ % is constant 
 for all $t,x,y$ and $z^b_\epsilon<z<z^b_\epsilon+H_\epsilon$, hence % $\rho_0(z)$ stratification => barocline
it holds \eqref{eq:nearincompressibility}
% \begin{equation}
% \label{eq:nearincompressibility}
% \div\bu_\epsilon = O(\epsilon) \text{ as } \epsilon\to0
% \end{equation}
 \item at $z=z^b_\epsilon$, $\bu_\epsilon\cdot\bn_{z^b_\epsilon}=0$ and \eqref{eq:navier} with $k_{z^b_\epsilon}=O(\epsilon)$ 
 \item at $z=z^b_\epsilon+H_\epsilon$, \eqref{eq:dynamic} and  \eqref{eq:navier} with $k_{z^b_\epsilon+H_\epsilon}=O(\epsilon^2)$
 \item $\partial_z \bu^H=O(\epsilon)$ % ou $\partial_z \bu^H_0=0$
 hence~\eqref{eq:dispersioneps} and
 \item at $t=0$, % it holds 
 $F^i_{\epsilon,c} =O(\epsilon)= F^z_{\epsilon,\alpha}$ for $i\in\{x,y\},\alpha\in\{a,b\}$, and for % some constant 
 ${\hat H}>0$
\begin{equation}
\label{eq:initialcompat} 
H_1/{\hat H}=F^z_{0,c} \equiv |\bF^H_0|^{-1} \,.
\end{equation}
\end{itemize}

Then, as $\epsilon\to0$, $H_\epsilon,\bU_\epsilon,\bSigma^H_\epsilon,\Sigma^{zz}_\epsilon$ defined as in \eqref{eq:depthaveragedeps} and
$$ 
\bF^H_\epsilon = \sum_{i\in\{x,y\},\alpha\in\{a,b\}} \left( \frac1{H_\epsilon} \int_{b_\epsilon}^{b_\epsilon+H_\epsilon} dz F^i_{\alpha} \right) \be_i\otimes\be_\alpha
$$
% converge to % first-order approximations
can be approximated by $H=H_1\equiv{\hat H}F^z_{0,c}$, $\bU=\bu^H_0$, $\bSigma^H=\bSigma^H_0$, $\Sigma^{zz}=\Sigma^{zz}_0$ and $\bF^H = \bF^H_0$ %\lim_{\epsilon\to0} \bF^H_\epsilon$ % supposedly smooth 
solution to \eqref{eq:massconservation}--\eqref{eq:momentumconservation}--\eqref{eq:strain} such that $H={\hat H}|\bF^H|^{-1}$ % {|F^{H,i}_\alpha|}$
and
\beq
\label{eq:strain}
\partial_t(H\bF^H) + \div_H(H\bU\otimes\bF^H-H\bF^H\otimes\bU) = 0 \,. % O(\epsilon^2)
\eeq

Moreover, if the 2D Piola identities \eqref{eq:2Dpiola} hold at $t=0$ % and therefore for all $t\ge0$ by \eqref{eq:strain}
\begin{equation}
\label{eq:2Dpiola}
{\hat H}\div_H\left(|\bF^H|^{-1}\bF^H\right) \equiv \div_H(H\bF^H) \equiv \partial_i( H F^i_\alpha ) = 0 % div S'INDICE toujours à la première composante tensorielle ici
\end{equation}
the motions defined by solutions to \eqref{eq:massconservation}--\eqref{eq:momentumconservation}--\eqref{eq:strain} 
have an equivalent 2D Lagrangian description using $\bPhi_t^H$ such that $\partial_t\bPhi_t^H=\bU\circ\bPhi_t^H$, $\grad_H\bPhi_t^H=\bF^H\circ\bPhi_t^H$.
\end{proposition}

% \beq 
% \label{eq:firstorder_spatial} 
% \begin{aligned}
% % & \rho ( \partial_t +  u^j\partial_j) u^i - \partial_j \sigma^{ij} = \rho f^i
% & \partial_t \left( \rho u^i \right) +  \partial_j \left( \rho u^j u^i - \sigma^{ij} \right) = \rho f^i
% \\
% %& |F^i_\alpha|^{-1} ( \partial_t +  u^j\partial_j) F^i_\alpha - \partial_j \left( |F^i_\alpha|^{-1} F^j_\alpha u^i \right) = 0
% & \partial_t \left( \rho F^i_\alpha \right) +  \partial_j \left( \rho u^j F^i_\alpha - \rho F^j_\alpha u^i \right) = 0
% \\
% % & \partial_t |F^i_\alpha|^{-1} + \partial_j\left( |F^i_\alpha|^{-1} u^j \right) = 0 
% & \partial_t \rho + \partial_j\left( \rho u^j \right) = 0 
% \end{aligned}
% \eeq 

\begin{proof} It suffices to complement the proof of Prop.~\ref{prop:depthaverage} and Cor.\ref{cor:depthaverage} as follows.
\begin{enumerate}
 \item The hypothesis $\partial_z \bu^H=O(\epsilon)$, and the intermediary result $\partial_H \bu^z=O(\epsilon)$ in the proof of Prop.~\ref{prop:depthaverage},
 imply that $F^i_{\epsilon,c} =O(\epsilon)= F^z_{\epsilon,\alpha}$ hold for $i\in\{x,y\},\alpha\in\{a,b\}$ and all $t\ge0$ if they hold at $t=0$,
 recall \eqref{eq:firstorder_spatial}.
 \item Then, by \eqref{eq:firstorder_spatial}, the hypothesis $H_1/{\hat H}=F^z_{0,c}$ in \eqref{eq:initialcompat} is also preserved 
 for all $t\ge0$ if it holds at $t=0$. 
 \item Last, \eqref{eq:firstorder_spatial} yields \eqref{eq:strain} for % the first-order approximation
 $\bF^H = \lim_{\epsilon\to0} \bF^H_\epsilon$ % = \bF^H_0
 insofar as $F^z_{0,c}=1/|\bF^H_0|$ holds for all $t\ge 0$ by assumption % on $\rho_0$
 $|\bF_0|=|\bF^H_0|F^z_{0,c}=1$.
\end{enumerate}
The equivalence of a Eulerian description with a Lagrangian description when the Piola identities \eqref{eq:2Dpiola} hold for all $t\ge 0$ is classical, see e.g.~\cite{wagner-2009}. Now, note that by \eqref{eq:strain}, \eqref{eq:2Dpiola} hold for all $t\ge0$ if they hold at $t=0$.
\end{proof}

\begin{corollary}
\label{cor:depthaverage2}
Assume that the % family of 
solutions considered in Prop.~\ref{prop:depthaverage2} also satisfy
${\bu^H_\epsilon k_{z^b_\epsilon}}/{\rho_0} = k H_\epsilon \bU_\epsilon + O(\epsilon^2)$.
Assume moreover that $\tau_\epsilon^{ij} = O(\epsilon)$ satisfy
\begin{equation}
\label{eq:taueps} 
\tau_\epsilon^{ij} =  % K \Hcal' 
\Gcal_\epsilon F^i_{\epsilon,\alpha} A^{\alpha\beta}_\epsilon F^j_{\epsilon,\beta} + O(\epsilon^2) 
\end{equation}
%% AJOUTER EQ POUR A_epsilon
with $\Gcal_\epsilon = O(\epsilon)$ and % symmetric positive definite tensors 
$\bA_\epsilon=A^{\alpha\beta}_\epsilon \be_\alpha\otimes\be_\beta \in S^{++}(\RR^{d\times d})$ such that 
\beq
\label{eq:odeeps}
\lambda % \frac{\xi}{4K \Hcal'}
(\partial_t + u^i\partial_i)  A^{\alpha\beta}_\epsilon % \left( A^{\alpha\beta}\circ\bphi_t \right)
=  
A^{\alpha\beta}_\epsilon % \circ\bphi_t 
+ \Theta % \frac{k_B \theta}{K \Hcal'} 
\left( [F^{-1}_\epsilon]^{\alpha}_i [F^{-1}_\epsilon]^{\beta}_i \right) + O(\epsilon) % \circ\bphi_t
\eeq 
for some $\lambda>0$ while at $t=0$, it holds
\begin{equation}
\label{eq:Aeps} 
A^{\alpha\beta}_\epsilon = O(\epsilon) \text{ if either $\alpha$ or $\beta$ equals $c$.}
\end{equation}

Then, for small $\epsilon$, % the first-order approximation of $H_\epsilon$
$H_1$ can be approximated % itself as $\epsilon\to0$
by a solution $H\equiv \hat H |\bF^H|^{-1}\approx H_1 $, $\bU\approx \bu^H_0$, $\bF^H\approx \bF^H_0$, $ A^{cc}\approx A^{cc}_0 >0$,
$ \bA^H = A^{\alpha\beta}\be_\alpha\otimes\be_\beta \in S^{++}(\R^{d \times d})$, % A^{aa}\be_a\otimes\be_a + A^{ab} ( \be_a\otimes\be_b + \be_b\otimes\be_a ) + A^{bb} \be_b\otimes\be_b$ where
$A^{\alpha\beta}\approx A^{\alpha\beta}_0$, $\alpha,\beta\in\{a,b\}$
to % the % closed Saint-Venant system 
\eqref{eq:massconservation}--\eqref{eq:momentumconservation}--\eqref{eq:strain}--\eqref{eq:stresseps}--\eqref{eq:AH}--\eqref{eq:Acc}
\begin{align}
\label{eq:stresseps}
& \Sigma^{zz} = \Gcal_\epsilon A^{cc} H^2
\qquad
\bSigma^H = \Gcal_\epsilon \bF^H \bA^H (\bF^H)^T
\\
\label{eq:AH}
& \lambda (\partial_t + \bU\cdot\grad_H) \bA^H = ((\bF^H)^T\bF^H)^{-1}-\bA^H 
\\
\label{eq:Acc}
&  \lambda (\partial_t + \bU\cdot\grad_H) A^{cc} = H^{-2}-A^{cc}
\end{align}
where the source terms for $\bA^H=(\bA^H)^T>0$ are defined using matrix products.

Moreover, the full Saint-Venant system \eqref{eq:massconservation}--\eqref{eq:momentumconservation}--\eqref{eq:strain}--\eqref{eq:AH}--\eqref{eq:Acc}
has an equivalent in material coordinates 
for smooth motions $\bPhi_t^H$ such that $\partial_t\bPhi_t^H=\bU\circ\bPhi_t^H$, $\grad_H\bPhi_t^H=\bF^H\circ\bPhi_t^H$
if % 2D 
Piola's identities \eqref{eq:2Dpiola} hold at $t=0$.
\end{corollary}

\begin{proof} 
\begin{enumerate}
 \item First observe $\bF^{-1} = \bF_0^{-1} + O(\epsilon)$ after using e.g. the Neumann series expansion of $(\bI + \epsilon \bF_0^{-1}\bR)^{-1}:=\bF_\epsilon^{-1}\bF_0$.
 \item Then recall from the proof of Prop.~\ref{prop:depthaverage2} that
 $F^i_{\epsilon,c} =O(\epsilon)= F^z_{\epsilon,\alpha}$ hold for $i\in\{x,y\},\alpha\in\{a,b\}$ and all $t\ge0$ in so far it holds at $t=0$, % \eqref{eq:firstorder_spatial}.
 so \eqref{eq:taueps}  % $A^{\alpha\beta}_\epsilon = O(\epsilon)$ 
 is preserved for all $t\ge0$ if it holds at $t=0$.
 \item With $H_1 = \hat H F^z_{0,c}=\hat H |\bF^H_0|^{-1}>0$, the first result above yields \eqref{eq:AH}--\eqref{eq:Acc}.
 Moreover, with the second result above, \eqref{eq:taueps} yields \eqref{eq:stresseps}. 
\end{enumerate}
Last, motions remain sufficiently smooth for changing to material coordinates without more constraint than in Prop.~\ref{prop:depthaverage2}.
\end{proof}

We have thus obtained a 2D system for the shallow % free-surface gravity
flows of UCM fluids which is a natural viscoelastic extension of the standard Saint-Venant system 
(for the shallow % free-surface gravity 
flows of % perfect or 
Newtonian fluids), and which we term Saint-Venant Maxwell (SVM in short).
Let us now show that the system of equations is a useful symmetric-hyperbolic system of conservation laws. % which we call SVUCM

\begin{proposition}\label{prop:svmconservation}
Smooth solutions to \eqref{eq:massconservation}--\eqref{eq:momentumconservation}--\eqref{eq:strain}--\eqref{eq:AH}--\eqref{eq:Acc}--\eqref{eq:stresseps} i.e. to % the system of conservation laws 
\beq\label{SVM}
\begin{aligned}
& \partial_t(H \bU) + \div_H\left( H \bU \otimes \bU + \left(\frac{g}2H^2 + \Gcal_\epsilon A^{cc} H^3\right)\bI - \Gcal_\epsilon H \bF^H \bA^H (\bF^H)^T \right)
\\
& \quad = - g H \grad_H z_b - k H \bU % \label{eq:momentumconservation}
\\
& \partial_t H + \div_H( H \bU ) = 0 % \label{eq:massconservation}
\\
& \partial_t(H\bF^H) + \div_H(H\bU\otimes\bF^H-H\bF^H\otimes\bU) = 0 % O(\epsilon^2) \label{eq:strain}
\\
& \partial_t( H  \bA^H) + \div_H( H \bU \bA^H) = H\left( ((\bF^H)^T\bF^H)^{-1}-\bA^H \right)/\lambda % \label{eq:AH}
\\
& \partial_t( H  A^{cc}) + \div_H( H \bU  A^{cc}) = H\left( H^{-2}-A^{cc}\right)/\lambda % \label{eq:Acc}
\end{aligned}
\eeq
are equivalently solutions to % the reduced system 
\eqref{eq:massconservation}--\eqref{eq:momentumconservation}--\eqref{eq:Sigmah}--\eqref{eq:Sigmazz} 
(our former formulation in \cite{bouchut-boyaval-2015} of a viscoelastic Saint-Venant system for Maxwell fluids)
when $\nu_\epsilon=\Gcal_\epsilon\lambda>0$ % \mu_\epsilon better notation ??
and the latter is complemented by \eqref{eq:strain}--\eqref{eq:stresseps},
or to a Lagrangian description in material coordinates  
% for smooth motions $\bPhi_t^H$ such that $\partial_t\bPhi_t^H=\bU\circ\bPhi_t^H$, $\grad_H\bPhi_t^H=\bF^H\circ\bPhi_t^H$
using $\bU\circ\bPhi_t^H=\partial_t\bPhi_t^H$, $\bF^H\circ\bPhi_t^H=\grad_H\bPhi_t^H$ when moreover \eqref{eq:2Dpiola} holds with $H|\bF^H|=\hat H>0$ constant.
% which % in absence of source terms  ??
% actually characterizes the critical points of the action functional 
% \begin{multline}
% \int_{t_0}^{t_1}\int_{\R^2} \frac12 |\partial_t\bPhi_t^H|^2
% - \frac12 \left( g |\grad_H\bPhi_t^H|^{-1} % H 
% + \Gcal_\epsilon A^{cc} |\grad_H\bPhi_t^H|^{-2}  % H^2 A_{cc}
% + \Gcal_\epsilon \grad_H\bPhi_t^H \bA^H (\grad_H\bPhi_t^H)^T % F^i_\alpha A_{\alpha\beta} F^i_\beta % - \log(A_{cc}|\bA_h|) \right) 
% \right)
% \end{multline}
%
Furthermore, % in smooth motions, 
they satisfy 
\begin{multline}
\label{eq:SVHEeul}
\partial_t E + 
\div_H\left(\bU (E + \frac{g}2H^2) + \Gcal_\epsilon H (\Sigma^{zz}-\bSigma^H) % (H^2 A_{cc}\bI-\bF^H \bA^H (\bF^H)^T)
\cdot\bU \right)
\\ 
= 
-kH|\bU|^2 -gH\bU\cdot\grad_Hz_b % U^i\partial_iz_b
-HD 
\end{multline}
% in material description
with % $\lambda D = \tr \bF^H \bA^H (\bF^H)^T + \tr (\bF^H \bA^H (\bF^H)^T)^{-1} + H^2A^{cc}+(H^2A^{cc})^{-1} -6 \ge 0$, 
$\Gcal_\epsilon \lambda D = \tr \bSigma^H + \tr (\bSigma^H)^{-1} + \Sigma^{zz}+(\Sigma^{zz})^{-1} -6 \ge 0$
and % for the energy functional 
\beq%gin{multline}
\label{energy}
E =\frac{H}2\Big( |\bU|^2 + gH +
%  \\ \Gcal_\epsilon \left( \tr (\bF^H \bA^H (\bF^H)^T) + H^2 A^{cc} - \log(H^2 A^{cc}|\bF^H \bA^H (\bF^H)^T|) \right)
\left( \tr \bSigma^H + \Sigma^{zz} - \log(\Sigma^{zz}|\bSigma^H|) \right) 
\Big).
% E =\frac{H}2\left( % \sum_i|U^i|^2 + gH + \mu  \left( F^i_\alpha A_{\alpha\beta} F^i_\beta + H^2 A_{cc} - \log(A_{cc}|\bA_h|) \right) \right)
\eeq%nd{multline} 
% $P = \frac{g}2H^2 + \Gcal_\epsilon H^3 A_{cc} =: -\partial_{H^{-1}} e_0$, $e_0 = \frac12(gH + \Gcal_\epsilon H^2 A_{cc}) $ 
\end{proposition}

\begin{proof}
The equivalence between the Eulerian descriptions of 
viscoelastic 2D Saint-Venant flows for Maxwell fluids 
% here and the more standard formulation in our former work \cite{bouchut-boyaval-2015}
can be seen e.g. on introducing 
\beq
\label{eq:BhandBzz}
\bc_H=\bF^H\bA^H(\bF^H)^T
\quad
c_{zz}=H^2 A^{cc}>0
\eeq
which can be thought as the first-order approximation (i.e. the depth-average) of the conformation tensor $\bc$
classically used for the viscoelastic modelling of polymeric flows, recall section~\ref{sec:viscoelastic}.
On noting $\bc_H=\lambda\bSigma^H/\nu_\epsilon+\bI$, $c_{zz}=\lambda\Sigma^{zz}/\nu_\epsilon+1$,
and starting from the system of conservation laws, one obtains 
\beq
\label{eq:SVUCM}
\begin{aligned}
    & \partial_t H + \div_H( H \bU ) = 0
    \\
    & \partial_t(H\bU) + \div_H(H\bU\otimes\bU + (gH^2/2 + \Gcal_\epsilon % \mu 
    H c_{zz})\bI - \Gcal_\epsilon %\mu 
    H\bc_H ) = -kH\bU -gH\grad_Hz_b
    \\
    & \partial_t \bc_H + \bU\cdot\grad_H \bc_H % D_t \bc_H 
    - (\grad_H\bU)\bc_H - \bc_H (\grad_H\bU)^T % - \zeta (\bD_h\bc_H + \bc_H \bD_h^T) 
    = (\bI-\bc_H)/\lambda
    \\ 
    & \partial_t c_{zz} + \bU\cdot\grad_H c_{zz} % D_t c_{zz} + 
    + 2 % (\zeta-1)
    c_{zz} \div_H\bU = (1-c_{zz}) / \lambda
\end{aligned}
\eeq
which is obviously equivalent to the 2D system proposed in \cite{bouchut-boyaval-2015}
as an extension of the 1D viscoelastic system in \cite{bouchut-boyaval-2013}
when it is complemented by \eqref{eq:strain}--\eqref{eq:BhandBzz}.

Reciprocally, the % reduced 
quasilinear system \eqref{eq:SVUCM} complemented by \eqref{eq:BhandBzz}--\eqref{eq:strain} rewrites as the % meaningful 
system of conservation laws \eqref{SVM} i.e.
\beq\label{eq:SVUCM0detail}
\hspace{-.3mm}\begin{aligned}
    & \partial_t (H U^i) + \partial_j( H U^jU^i + (g\tfrac{H^2}2 + \Gcal_\epsilon % \mu 
    H^3 A^{cc})\delta_{i=j} - \Gcal_\epsilon % \mu 
    H F^i_\alpha A^{\alpha\beta} F^j_\beta) = -HkU^i-gH\partial_iz_b
    \\
    & \partial_t (HF^i_\alpha) + \partial_j( H U^j F^i_\alpha - H F^j_\alpha U^i ) = 0
    \\
    & \partial_t H + \partial_{j}( H U^j ) = 0
    \\
    & \partial_t (H A^{\alpha\beta}) + \partial_j( H U^j A^{\alpha\beta} ) 
    = H ( |\bF_h|^{-2} \sigma_{\alpha\alpha'}\sigma_{\beta\beta'} % \delta_{\alpha\neq\alpha'}\delta_{\beta\neq\beta'}
    F^k_{\alpha'} F^k_{\beta'} -A^{\alpha\beta})/\lambda
    \\ 
    & \partial_t (HA^{cc}) + \partial_{j}( H U^j A^{cc} ) =  H(H^{-2}-A^{cc})/\lambda
\end{aligned}
\eeq 
in coordinates using $\alpha\in\{a,b\}$, $i\in\{x,y\}$
(note that adding \eqref{eq:BhandBzz},\eqref{eq:strain} was not necessary in 1D \cite{bouchut-boyaval-2013}).
Moreover, if Piola's identities \eqref{eq:2Dpiola} hold and $H|\bF^H|=\hat H$
% \begin{equation}
% \label{eq:2Dpiola}
% \partial_{\alpha}( \sigma_{\alpha\beta} F^k_\beta ) = 0
% \quad
% \forall \gamma
% \quad 
% \text{ or }
% \quad
% \partial_{j}( H F^j_\gamma ) = 0 
% \quad 
% \forall k 
% \,,
% \end{equation}
then one has the equivalent Lagrangian description
% using a bijective map $\phi:(t,a,b)\to(t,x,y)$ % between the current and reference configurations
% such that $\bU\equiv\partial_t\phi$, $\bF\equiv\grad\phi$
\beq
\label{eq:SVUCM0detaillag}
\begin{aligned}
    & \partial_t U^i + \partial_\alpha \left( (gH^2/2 + \Gcal_\epsilon % \mu 
    H^3 A^{cc}) \sigma_{ij}\sigma_{\alpha\beta} %\delta_{j\neq k}\delta_{\alpha\neq\beta} 
    F^j_\beta - \Gcal_\epsilon % \mu 
    F^i_\beta A_{\beta\alpha} \right) = -kU^i - g \partial_i z_b % (\partial_\alpha ( \sigma_{ij}\sigma_{\alpha\beta}F^j_\beta z_b)
    \\
    & \partial_t F^i_\alpha - \partial_\alpha U^i = 0
    \\
    & \partial_t H^{-1} - \partial_{\alpha}( U^j \sigma_{jk}\sigma_{\alpha\beta} % \delta_{j\neq k}\delta_{\alpha\neq\beta} 
    F^k_\beta ) = 0 
    \\
    & \partial_t A^{\alpha\beta} = ( |\bF^H|^{-2}\sigma_{\alpha\alpha'}\sigma_{\beta\beta'} %\delta_{\alpha\neq\alpha'}\delta_{\beta\neq\beta'}
    F^k_{\alpha'} F^k_{\beta'} -A^{\alpha\beta})/\lambda
    \\ 
    & \partial_t A^{cc} = (H^{-2}-A^{cc})/\lambda
\end{aligned}
\eeq
using fields functions of material coordinates (defined in a reference configuratio of the body)
-- i.e. for the sake of clarity we abusively used the same notation in \eqref{eq:SVUCM0detaillag} 
as in the Eulerian description \eqref{eq:SVUCM0detail}, omitting $\circ\bPhi_t^H$.

Last, one easily computes the following balance % conservation law
in the Lagrangian description % classically obtained by variational principles % in absence of source terms, and then also obviously 
\begin{multline}
\label{eq:SVHElag}
\partial_t 
\left( \frac12 \sum_i|U^i|^2 + \frac12 g H + \frac12 \Gcal_\epsilon \left( F^i_\alpha A^{\alpha\beta} F^i_\beta + H^2 A^{cc} % - \log(A_{cc}|\bA_h|) 
\right) \right)
\\
+ \partial_\alpha\left(  U^i \left( (\frac{g}2H^2 + \Gcal_\epsilon H^3 A^{cc}) \sigma_{ij}\sigma_{\alpha\beta} F^j_\beta 
- \Gcal_\epsilon  F^i_\alpha A^{\alpha\beta} \right) \right)
\\
= -k|\bU|^2 - U^i(\partial_iz_b) + ( \delta_{\alpha\beta} - F^k_{\alpha} F^k_{\beta} A^{\alpha\beta})/\lambda + (1 -H^2 A^{cc})/\lambda
\end{multline}
hence \eqref{eq:SVHEeul} in spatial coordinates on noting
\begin{multline}
\label{eq:SVHElagbis}
(\partial_t+\bU^H\cdot\grad_H) \log|F^i_\alpha A^{\alpha\beta} F^i_\beta| = (\partial_t+\bU^H\cdot\grad_H) \log|\bc_H| 
\\
= \tr\left((\bc_H)^{-1} (\partial_t+\bU^H\cdot\grad_H) \bc_H \right)
= 2(\div_H\bU) + (\tr(\bc_H)^{-1}-2)/\lambda
\end{multline}
% and
\begin{multline}
\label{eq:SVHElagter}
(\partial_t+\bU^H\cdot\grad_H) \log(H^2 A^{cc}) = c_{zz}^{-1} (\partial_t+\bU^H\cdot\grad_H) c_{zz}
\\ = - 2(\div_H\bU) + (c_{zz}^{-1}-1)/\lambda
\end{multline}
and $x+x^{-1}\ge2$, $\forall x>0$.
%%%%%%%%%%%%%%%%%%%%%%%
\end{proof}

\begin{remark}[Saint-Venant extension to weakly-sheared RANS models]
\label{rem:rans}
Despite the similarity between \eqref{eq:massconservation}--\eqref{eq:momentumconservation}--\eqref{eq:stresseps}--\eqref{eq:Sigmah}--\eqref{eq:Sigmazz} % with $\nu_\epsilon=\Gcal_\epsilon\lambda$, % \mu_\epsilon better notation ??
% our former formulation in \cite{bouchut-boyaval-2015} of a viscoelastic Saint-Venant system for Maxwell fluids
and the 2D system in the recent work \cite{gavrilyuk-ivanova-favrie-2018} that extends Saint-Venant to weakly-sheared RANS models, % the systems differ only in a minus sign % in \eqref{eq:Sigmah} and \eqref{eq:Sigmazz}
% , otherwise it is similar to the non-conservative viscoelastic extension of Saint-Venant models derived in recent works of us \cite{bouchut-boyaval-2015,boyaval-hal-01661269} from the standard formulations of UCM
the latter has no known conservative formulation as opposed to the former. % let alone symmetrc hyperbolicity 
% For genuinely multi-dimensional flows, 
% it is difficult to define solutions to the quasilinear systems above with ``non-conservative products''.
This is a well-known ``apparent similarity'' between RANS and Maxwell equations, see e.g. \cite{speziale-2000}. % even in 3D
% not known to be easily interpretable. % herard coquel
% SPEZIALE1988211,speziale-1998,speziale-2000
\end{remark}

%%%%%%%%%%%%%%%%%%%%%%%%%%%%%%%%%%%%%%%%%%%%%%%%%%%%%

% \subsubsection{Saint-Venant-Maxwell % system
% is  symmetric-hyperbolic} % of conservation laws

\begin{proposition}
\label{prop:symhyp}
% smooth isothermal viscoelastic motions of compressible UCM fluids defined by 
Smooth solutions to % the homogeneous !! quasilinear system 
\eqref{SVM} with $\bA^H\in S^{++}(\R^{d\times d}),A^{cc}>0$ 
are % equivalently defined by
in bijection with smooth solutions $q=(H,H\bU,H\bF^H,H\bY^H,HY^{cc})$ to % the conservation laws 
\beq\label{SVM1}
\begin{aligned}
& \partial_t(H \bU) + \div_H\left( H \bU \otimes \bU + \left(\frac{g}2H^2 + \Gcal_\epsilon A^{cc} H^3\right)\bI - \Gcal_\epsilon H \bF^H \bA^H (\bF^H)^T \right) =
\\
& \quad - g H \grad_H z_b - k H \bU % \label{eq:momentumconservation}
\\
& \partial_t H + \div_H( H \bU ) = 0 % \label{eq:massconservation}
\\
& \partial_t(H\bF^H) + \div_H(H\bU\otimes\bF^H-H\bF^H\otimes\bU) = 0 % O(\epsilon^2) \label{eq:strain}
\\
& \partial_t( H  \bY^H) + \div_H( H \bU \bY^H) = - H \bY^H\left( \bZ^H-\bY^H \right)\bY^H/\lambda % \label{eq:AH}
\\
& \partial_t( H Y^{cc}) + \div_H( H \bU Y^{cc}) = H\left( H^{-2}(Y^{cc})^{-3}-Y^{cc}\right)/4\lambda % \label{eq:Acc}
\end{aligned}
\eeq % $\bY^H:=(\bA^H)^{-2}$, $Y^{cc}:=(A^{cc})^{1/4}$
when $\bA^H=(\bY^H)^{-\frac12},A^{cc}=(Y^{cc})^{4}$ is defined componentwise by identification with
the square-root matrix-inverse of $\bY^H=Y^{\alpha\beta}\be_\alpha\otimes\be_\beta\in S^{++}(\R^{d\times d})$,
$\bZ^H:=(\bA^H\bF^H(\bF^H)^{T})^{-1}+(\bF^H(\bF^H)^{T}\bA^H)^{-1}$, % in matrix notations 
and we recall $H|\bF^H|=\hat H>0$.
% $Z^{\gamma\delta} =  [F^{-1}]^{\gamma}_i [F^{-1}]^{\beta}_i [A^{-1}]^{\beta\delta} + [A^{-1}]^{\gamma\alpha} [F^{-1}]^{\alpha}_i [F^{-1}]^{\delta}_i$
Furthermore, for some algebraic term $\tilde\Dcal$ without sign a priori, the functional
\beq%gin{multline}
\label{energy2}
\tilde E =\frac{H}2\Big( |\bU|^2 + gH +
%  \\ \Gcal_\epsilon \left( \tr (\bF^H \bA^H (\bF^H)^T) + H^2 A^{cc} - \log(H^2 A^{cc}|\bF^H \bA^H (\bF^H)^T|) \right)
\left( \tr \bSigma^H + \Sigma^{zz} % - \log(\Sigma^{zz}|\bSigma^H|) 
 + \tr\left(\bY^H\bY^H\right) \right) 
\Big)
% E =\frac{H}2\left( % \sum_i|U^i|^2 + gH + \mu  \left( F^i_\alpha A_{\alpha\beta} F^i_\beta + H^2 A_{cc} - \log(A_{cc}|\bA_h|) \right) \right)
\eeq%nd{multline} 
strictly convex in $q\in \Acal^+_H := \{H>0\,,\ \bY^H=(\bY^H)^T>0\,, Y_{cc}>0\}$ % \{H>0\,,\ \bA^H=(\bA^H)^T>0\,, A_{cc}>0\} 
satisfies % the additional conservation law
\begin{multline}
\label{eq:energy_spatialSVM}
\partial_t \tilde E + 
\div_H\left(\bU (\tilde E + \frac{g}2H^2) + H (\Sigma^{zz}-\bSigma^H) % \Gcal_\epsilon (H^2 A_{cc}\bI-\bF^H \bA^H (\bF^H)^T)
\cdot\bU \right)
\\ 
= 
-kH|\bU|^2 -gH\bU\cdot\grad_Hz_b % U^i\partial_iz_b
-H\tilde D \,.
\end{multline}
Thus $\tilde E$ defines a mathematical entropy for % the (homogeneous) system
\eqref{SVM1}, 
\eqref{eq:energy_spatialSVM} defines a \emph{strictly convex extension} for % the (homogeneous) system
\eqref{SVM1},
and \eqref{SVM1} is % thus therefore 
a \emph{symmetric-hyperbolic} system of conservation laws on the % domain i.e. a convex
open set $\Acal^+_H \equiv \{H>0\,,\ \bA^H=\bA^H>0\,,\ A_{cc}>0 \}$. %
\end{proposition}

\begin{proof}
It is a lengthy but straightforward computation to show the bijection between smooth solutions, i.e. the equivalence between \eqref{SVM} and \eqref{SVM1}. % for smooth solutions.
Next, recalling Godunov-Mock theorem \cite{godlewski-raviart-1996},
it suffices to show that $\tilde E$ is (jointly) strictly convex in $q$ % $(H,H\bU,H\bF^H,H(\bA^H)^{-2},H(A^{cc})^{1/4})$,
i.e. the Lagrangian energy $\tilde E/H$ is (jointly) strictly convex in $(H^{-1},\bU,\bF^H,\bY^H,Y^{cc})$, recall e.g. \cite{bouchut-2003}.
Now, to that aim, note that $\tilde E/H$ is the sum of (a) $|\bU|^2/2$ strictly convex in $\bU$,
plus (b) $gH + \Gcal_\epsilon H^2 (Y^{cc})^{-4}$ 
strictly convex in $(H^{-1},Y^{cc})\in(\R^+_*)^2$ --  % A_{cc}^{1/4}
compute for instance the Hessian matrix
$$
\begin{pmatrix}
2gH^3 + 6\mu H^4 A_{cc} & - 2\mu  H^3 A_{cc}^{3/4}                
\\
- 2\mu  H^3 A_{cc}^{3/4} & 2\mu  H^2 A_{cc}^{1/2}
\end{pmatrix}\; \text{--,}
$$
and (c) $\tr(\bF^H(\bY^H)^{-\frac12}(\bF^H)^T) + \tr\left(\bY^H\bY^H\right) $ which is strictly convex in $(\bF^H,\bY^H)$ on $\Acal_H^+$
as we already proved % in dimension $3$
(for any dimension !) in Prop.~\ref{prop:symmetrichyperbolic}.
\end{proof}

\begin{corollary}
\label{cor:strongsol2}
Consider the SVM system \eqref{SVM1} 
% i.e. the % system of
% conservation laws 
\begin{equation}
\label{eq:quasilinear2}
% \partial_t q + A_x(q)\partial_x q + A_y(q)\partial_y q = B(q)
\partial_t q + \grad_qF_i(q)\partial_i q = B(q)
\end{equation}
with the smooth % $C^\infty$ flux and source 
functionals $F_i,B$. % analytic in $q=(\rho,\rho u^i,\rho Y^{\alpha\beta},\rho F^i_\alpha)$
For all state $q_0 \in \Acal_H^{+}$, % := \{\rho>0,\ \bA=\bA^{T}>0 \}$
and for all % $\left(1+ d + d(d+1)/2 + d^2\right)$-dimensional % initial 
perturbation $\tilde q_0 \in H^s(\RR^2)$ in Sobolev space $H^s$ with $s>2$ % 1+d/2 
such that $q_0 + \tilde q_0$ is compactly supported in $\Acal_H^{+}$,
there exists $T>0$ and a unique classical solution $q\in C^1([0,T)\times\RR^2)$ to \eqref{eq:quasilinear2} 
such that $q(t=0)=q_0 + \tilde q_0$. 

Furthermore, $q - q_0 \in C^0([0,T),H^s)\cap C^1([0,T),H^{s-1})$. % = \tilde q
\end{corollary}
\begin{proof}The proof is the same as Cor.~\eqref{cor:strongsol} in the general (non shallow) case.\end{proof}

To our knowledge, Cor.~\ref{cor:strongsol2} is the first well-posedness result % existence of unique solutions to
for the Cauchy problem of a 2D viscoelastic Saint-Venant system with Maxwell fluids.
Moreover, note that the % mathematical 
structure of the 2D viscoelastic Saint-Venant system is similar to the 3D % compressible 
full UCM system of Section~\ref{sec:strictlyconvex}.
Then, damping can be similarly expected on large time for $E$, in a similar non-standard way since $E$ is different from $\tilde E$ yielding a convex extension of SVM.
And numerical difficulties with standard discretization % methods 
can also be expected.
However, note \eqref{energy} simplifies to
\beq%gin{multline}
\label{energySVM}
E =\frac{H}2\Big( |\bU|^2 + gH + %  \\ \Gcal_\epsilon \left( \tr (\bF^H \bA^H (\bF^H)^T) + H^2 A^{cc} - \log(H^2 A^{cc}|\bF^H \bA^H (\bF^H)^T|) \right)
\left( \tr \bSigma^H + \Sigma^{zz} - \log(A^{cc}|\bA^H|) \right) 
\Big)
% E =\frac{H}2\left( % \sum_i|U^i|^2 + gH + \mu  \left( F^i_\alpha A_{\alpha\beta} F^i_\beta + H^2 A_{cc} - \log(A_{cc}|\bA_h|) \right) \right)
\eeq%nd{multline} 
here on using the incompressibility condition $H|\bF_h|=\hat H$.

Last, recalling that our full UCM formulation is a viscoelastic extension of polyconvex elastodynamics,
note that our 2D viscoelastic Saint-Venant system % with Maxwell fluids 
is obtained from a different reduction procedure than e.g. shell and plate models from (standard) elastodynamics.
It uses non-standard boundary conditions for elastodynamics (i.e. free-surface on top of the layer).
It may thus be interesting to study applications of the non-standard, apparently new, 2D reduction of (standard) elastodynamics 
in the formal limit $\lambda\to\infty$ when $\bA^H,A^{cc}$ is constant.

%%%%%%%%%%%%%%%%%%%%%%%%%%%%%%%%%%%%%%%%%%%%%%%%%%%%%

\subsection{Illustrative flow examples}

To probe the viscoelastic Saint-Venant-Maxwell % shallow-water
model \eqref{SVM} % e.g. 
in a % geophysical 
context, % as usual for new conceptual model
it is useful to first imagine simple flows in idealized settings.

For instance, let us look for a 1D shear flow 
$\bU\circ\bPhi_t^H=\partial_t\bPhi_t^H,\bF^H\circ\bPhi_t^H=\grad_H\bPhi_t^H$
where $\bPhi_t^H(\ba)=\ba+X(t,b)\be_a$ is a solution to the Lagrangian description \eqref{eq:SVUCM0detaillag}
for % the half-plane 
$t,a\in\R,b>0$ using % boundary condition 
$X(t,b=0)=\Delta X H(t)$, $\Delta X >0$,
and $\bPhi_t^H(\ba)=\ba$ if $t\le0$. We denoted $H(t)\equiv1_{t>0}$ Heaviside step function.

Such a 1D solution with $|\bF^H|=1=H/\hat H$ has already been considered % many times
using various % formulations of the UCM
incompressible viscoelastic flow models, of course.
Assuming $\bA^H(\ba)=\bI$ if $t\le0$, one gets % on noting
% \bF^H = \begin{pmatrix}
%          1 & \partial_b X
%          \\ 0 & 1
%         \end{pmatrix}
% 
% [\bF^H]^{-1} = \begin{pmatrix}
%          1 & -\partial_b X
%          \\ 0 & 1
%         \end{pmatrix}
with $A^{aa}(t,b)=1+|\partial_b X|^2$, $A^{ab}(t,b)=-\partial_b X$, $A^{bb}=1$:
$$
\bA^H = \int_0^t ds\; M'(t-s) A^{\alpha,\beta} \be_a\otimes\be_b \quad M(\tau)=e^{-\tau/\lambda} \,.
$$
When % the bottom is flat 
$\grad_H z_b=0$ and % frictionless 
$k=0$,
it naturally leads, for the displacement $X(t,b)$, to the same ``Stokes first problem'' as e.g. in K-BKZ theory
$$
\partial^2_{tt} X(t,b) = \Gcal_\epsilon \partial^2_{bb} X(t,b) + \Gcal_\epsilon \int_0^t ds\; M'(t-s) \partial^2_{bb} X (s,b) \quad t,b>0 \,.
$$
Then, on recalling $\partial^2_{bb} X(t,b)=0=\partial_{t} X(t,b)$ when $t\le0$, we solve
$$
\partial_{t} X(t,b) = \Gcal_\epsilon \int_0^t ds\; M(t-s) \partial^2_{bb} X (s,b) \quad t,b>0
$$
using Laplace transform $\hat X(\omega,b)=\int_0^\infty dt\; e^{-\omega t}X(t,b)$ 
% $\omega \hat X = \hat M \partial_{bb}^2 \hat X hence \hat X = \frac{\Delta X}\omega e^{-g\sqrt{\frac\omega{\hat M}}}
\cite[p.197]{CarslawJaegerBook} and obtain: 
\begin{equation}
\label{eq:sol}
% \frac{X(t,b=\lambda\sqrt{\Gcal_\epsilon}y)}{\Delta X}= e^{-y} - y \int_y^{\frac{t}\lambda} d r e^{-r} \frac{I_1\left(-\sqrt{r^2-y^2}\right)}{\sqrt{r^2-y^2}}
X(t,b=\lambda\sqrt{\Gcal_\epsilon}y) = 
\Delta X \left( e^{-y} 
% - y \int_y^{\frac{t}\lambda} d r e^{-r} \frac{I_1\left(-\sqrt{r^2-y^2}\right)}{\sqrt{r^2-y^2}}
+ y \int_y^{\frac{t}\lambda} d r e^{-r} \frac{I_1\left(\sqrt{r^2-y^2}\right)}{\sqrt{r^2-y^2}}
 \right) H(t-b/\sqrt{\Gcal_\epsilon}) 
\end{equation}
where $I_1$ denotes the fist-order modified Bessel function of the first kind.

That is, to probe \eqref{SVM} in hydraulics,
one could first try to apply the 1D solution above e.g. to the flow % shear wave
generated in a shallow reservoir by sudden longitudinal % impulse
displacements of a flat wall, choosing $\sqrt{\Gcal_\epsilon}>0$ as the front % wave 
speed and % next
$\lambda>0$ so that the amplitude % of $X$
decays like in Fig.~\ref{fig1} on small times.
% Of course, it may not work !! The viscoelatsic would therefore not be useful on any (whatever small) time interval
But letting alone the % physical 
assumptions about the dynamics, % SCL + boundary condition dynamics 
the assumed 1D kinematics is a strong limitation for application to real flows.
And the new systems % of conservation laws
proposed in this work should definitely improve the latter limitation~!

\begin{figure}
\centering\includegraphics[scale=.7]{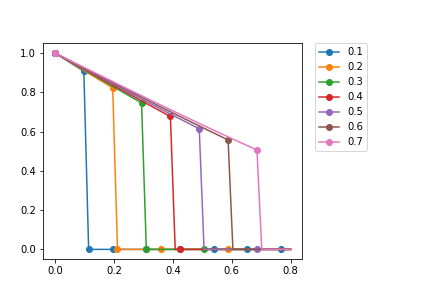} 
\caption{\label{fig1} Solution $X/\Delta X$ of the Stokes first problem given by \eqref{eq:sol} as a function of $y=b/\lambda\sqrt{\Gcal_\epsilon}$ at $t/\lambda \in \{.1,.2\dots.7\}$ using numerical integration.}
\end{figure}

Now, to probe 
% the viscoelastic Saint-Venant-Maxwell % shallow-water
% model 
\eqref{SVM} in a more realistic multi-dimensional setting,
one may want to first compute \emph{simple} multi-dimensional %(i.e. non 1D)
solutions possessing symmetries.
% with % a system of conservation laws like 
% \eqref{SVM}.
%
For instance, using cylindrical coordinates $(R,\Theta)$ % $\be_R,\be_\Theta)$ 
and $(r,\theta)$ for both the material and spatial frames, one may want to % use \eqref{SVM} so as to 
compute supposedly axisymmetric (also called azimuthal or rotational) $r=R$, $\theta=\Theta+\psi(t,R)$ shear waves % satisfying 
% % such that % in matrix notations : r,\theta lines and R,\Theta columns 
% $$
% \bF^H = \begin{pmatrix}
%         1 & 0
%         \\
%         \kappa(t,R) := R\partial_R\psi & 1
%         \end{pmatrix}
% % [\bF^H]^{-1} = \begin{pmatrix}
% %         1 & 0
% %         \\
% %         -\kappa & 1
% %         \end{pmatrix}
% \quad
% \bA^H = \bI + \frac1\lambda
%         \begin{pmatrix}
%         0 & - \int_0^t ds e^{\frac{s-t}\lambda} \kappa
%         \\
%         - \int_0^t ds e^{\frac{s-t}\lambda} \kappa & \int_0^t ds e^{\frac{s-t}\lambda} \kappa^2
%         \end{pmatrix}
% % assuming \bA^H(0) = \bI
% $$
% The equations of motions are 
\cite{haddow-erbay-2002}:
\begin{align}
- R |\partial_t \psi|^2 & = \partial_R S^{rR} + (S^{rR}-S^{\theta\Theta}-\kappa S^{\theta R})/R
\\
\partial_{tt}^2 \psi & = R \partial_R S^{\theta R} + 2 S^{\theta R}
\end{align}
%   author    = {Haddow, J. B. and Erbay, H. A.},
%   title     = {Some aspects of finite amplitude transverse waves in a compressible hyperelastic solid}
with
% for $i\in\{r,\theta\}$, $\alpha\in\{R,\Theta\}$, $S^i_\alpha = \partial_{F^i_\alpha} \tilde e$, 
% $ \tilde e = % g \frac{H}2 + \Gcal_\epsilon A^{cc} \frac{H^2}2 +  
% \Gcal_\epsilon F^i_\alpha F^j_\beta A_{\alpha\beta}$ % with $H = |\bF^H|^{-1}$, $A^{cc} = H $ constant 
% hence 
$\kappa(t,R) := R\partial_R\psi$,
$S^{rR}(t)$ uniform in space, % - \left( g \frac{H}2 + \Gcal_\epsilon A^{cc} H^2 \right) *    + 
and
% \begin{align}
$$
S^{\theta R} = \kappa - \int_0^t ds e^{\frac{s-t}\lambda} \kappa \,,
\quad
S^{\theta \Theta} = \int_0^t ds e^{\frac{s-t}\lambda} \kappa^2 - \kappa \int_0^t ds e^{\frac{s-t}\lambda} \kappa \,.
$$
% \end{align}

But even if % we assume 
such axisymmetric solutions exist, they % however 
do not seem easily constructed anyway.
In practice, it is easier to \emph{numerically simulate} multi-dimensional shear waves with a generic discretization method.
This will be the subject of future specialized works. Recall indeed that standard discretization methods need to be adapted, 
so as to generically simulate SVM on large times with the dissipative inequality as a stability property for the discrete system 
(indeed, the latter inequality does not correspond to the convex extension of the symmetric-hyperbolic system).

\section{Conclusion} %%%%%%%%%%%%%%%%%%%%%%%%%%%%%%%%%%%%%%%%%%%%%%%%%%%%%%%%%%%%%%%%%%%%%%%%%%%%%%%

In this work, we have derived % obtained 
new \emph{symmetric hyperbolic} systems of conservation laws 
to model viscoelastic flows with Upper-Convected Maxwell fluids,
either 3D compressible or 2D incompressible with hydrostatic pressure and a free surface.
The % new 
systems yield the first well-posedness results for causal multi-dimensional viscoelastic motions satisfying the locality principle
(i.e. information % perturbations 
propagates at finite-speed) as %using the 
small-time smooth solutions to Cauchy initial-value problems.

The systems also suggest a promising route to unify models for solid and fluid motions. % in continuum mechanics
Like K-BKZ theory for viscoelastic fluids with fading memory, % regular kernel
they extend standard symmetric-hyperbolic systems % with famous physical interpretations 
(polyconvex elastodynamics and Saint-Venant shallow-water systems).
However, they are formulated differently,
with the help of an additional material metric variable.
Now, using the same methodology, % i.e. an additional material metric variable, 
other viscoelastic models % that use different rheological equations like LCM, recall
with a K-BKZ integro-differential formulation
could in fact be similarly formulated as systems of conservation laws.
Moreover, varying the relaxation limit of the additional material metric variable should yield (symmetric-hyperbolic formulations of) many possible flow models
in between elastic solids and fluids, like elasto-plastic models.
New rheological extensions of the polyconvex elastodynamics and Saint-Venant shallow-water systems will be studied in future works.

To precisey apply our new system, in hydraulics in particular, future works shall also consider numerical simulations.
% But first,
Note then that standard discretization methods shall first be adapted like e.g. in \cite{boyaval-hal-02262298} to handle large-time motions, % weak solutions 
since the physical energy functional that dissipates is not the strictly convex functional yielding a strictly convex extension.

\end{document}